\documentclass[10pt,english]{amsart}
\usepackage{amsmath,amssymb,amsthm,amsfonts,graphicx,color}
\usepackage{amssymb,geometry}
\usepackage{bbm}
\usepackage{mathtools}
\usepackage{amsthm}
\usepackage{amsfonts}
\usepackage{amssymb}
\usepackage{enumitem}
\usepackage{physics}
\usepackage[dvipsnames]{xcolor}
\usepackage{graphicx}
\usepackage{mathtools}
\usepackage{amsthm}
\usepackage{amsfonts}
\usepackage{amssymb}
\usepackage{amsmath,graphicx,color}
\usepackage{geometry}
\usepackage{bbm}
\usepackage{enumitem}
\usepackage{physics}

\usepackage[colorlinks=true]{hyperref}

\makeatletter
\def\@currentlabel{2.1}\label{e:dispaa}
\def\@currentlabel{2.21}\label{e:dispau}
\def\@currentlabel{2.22}\label{e:dispav}
\def\@currentlabel{2.23}\label{e:dispaw}
\def\@currentlabel{2.24}\label{e:dispax}
\def\theequation{\thesection.\@arabic\c@equation}
\makeatother

\hypersetup{linkcolor=black,citecolor=black,filecolor=black,urlcolor=black} 

\definecolor{dullmagenta}{rgb}{0.4,0,0.4}   
\definecolor{darkblue}{rgb}{0,0,0.4}

\newcommand{\A}{\mathcal{A}}

\newcommand{\R} {\mathbb R}
\newcommand{\ES}{\mathbb S}
\newcommand{\cuad}{{\sqcap\kern-.68em\sqcup}}

\newcommand{\be}{\begin{equation}}
\newcommand{\ee}{\end{equation}}

\newcommand{\supp}{\mathop{\mbox{\normalfont supp}}\nolimits}

\renewcommand{\(}{\left(}
\renewcommand{\)}{\right)}
\newcommand{\qb}{{\root 4\of \beta}}
 \newcommand{\bw}{{\boldsymbol{\mathrm{w}}}}
 \newcommand{\w}{{{\mathrm{w}}}}
  \newcommand{\mv}{{{\mathrm{v}}}}
  \newcommand{\mU}{{{\mathrm{U}}}}
  \newcommand{\mr}{{\boldsymbol{\mathrm {r}}}}

 \newcommand{\pC}{{\mathring C}}
\newcommand{\wt}[1]{\widetilde{#1}}
\newcommand{\mbV}{\mathbf {V}}
\newcommand{\mbW}{\mathbf{W}}
\newcommand{\eps}{\epsilon}
\newcommand{\bs}[1]{\boldsymbol{#1}}
\newcommand{\Id}{\operatorname{Id}}
\newcommand{\pr}{\parallel}
\newcommand{\blu}[1]{{\color{black}{#1}}}

\def\bR {\mathbb{R}}

\makeatletter
\newcommand*\rel@kern[1]{\kern#1\dimexpr\macc@kerna}
\newcommand*\widebar[1]{%
  \begingroup
  \def\mathaccent##1##2{%
    \rel@kern{0.8}%
    \overline{\rel@kern{-0.8}\macc@nucleus\rel@kern{0.2}}%
    \rel@kern{-0.2}%
  }%
  \macc@depth\@ne
  \let\math@bgroup\@empty \let\math@egroup\macc@set@skewchar
  \mathsurround\z@ \frozen@everymath{\mathgroup\macc@group\relax}%
  \macc@set@skewchar\relax
  \let\mathaccentV\macc@nested@a
  \macc@nested@a\relax111{#1}%
  \endgroup
}
\makeatother

\newtheorem{theorem}{Theorem}[section]
\newtheorem{lemma}{Lemma}[section]
\newtheorem{proposition}{Proposition}[section]
\newtheorem{corollary}{Corollary}[section]
\newtheorem{definition}{Definition}[section]
\newtheorem{remark}{Remark}[section]
\newtheorem{example}{Example}[section]

\renewcommand{\theequation}{\thesection.\arabic{equation}}

\begin{document}

\title[Phase separating solutions]{Phase separating solutions for two component systems in general planar domains}
  \date{}
 
\author{Micha{\l } Kowalczyk}
\address{Departamento de Ingenier\'{\i}a Matem\'atica and Centro
de Modelamiento Matem\'atico (UMI 2807 CNRS), Universidad de Chile, Casilla
170 Correo 3, Santiago, Chile.}
\email {kowalczy@dim.uchile.cl}
\thanks{M. Kowalczyk was partially funded by Chilean research grants FONDECYT 1210405 and ANID projects ACE210010 and FB210005. He also acknowledges the hospitality of the Sapienza University in Rome where part of this work was done during his visit in March 2019.  A.~Pistoia and G. Vaira  were  partially supported by project Vain-Hopes within the program VALERE: VAnviteLli pEr la RicErca}

\author{Angela Pistoia}\address{Dipartimento SBAI, Sapienza Università  di Roma, via Antonio Scarpa 16, 00161 Roma, Italy.}
\email {angela.pistoia@uniroma1.it}

\author{Giusi Vaira}\address{Dipartimento di Matematica, Università   di Bari, Via Edoardo Orabona 4, 70125 Bari, Italy.}
\email {giusi.vaira@uniba.it}

\begin{abstract}
In this paper we consider a two component system of coupled non linear Schr\"odinger  equations modeling the phase separation in  the binary mixture of  Bose-Einstein condensates and other related problems. Assuming the existence of solutions in the limit of large interspecies scattering length $\beta$ the system reduces to a couple of scalar problems on subdomains of pure phases \cite{nttv}. Here we show that given a solution to the limiting problem under some additional non degeneracy assumptions there exists a family of solutions parametrized by  $\beta\gg 1$. 
\end{abstract}

\maketitle
  \section{Introduction}

\subsection{Motivation}

The  Gross–Pitaevskii system  \cite{15,23} consisting of two  coupled nonlinear Schr\"odinger equations,
\begin{equation}\label{gp}\begin{aligned}&-\boldsymbol{\iota}\frac\partial{\partial t}\Phi_j=\Delta\Phi_j-V_j(x)\Phi_j-\mu_j|\Phi_j|^2\Phi_j-\sum\limits_{i\not=j}\beta_{ij}|\Phi_i|^2\Phi_j,\ x\in\Omega,\ t>0,\\
& \Phi_j=\Phi_j(x,t)\in\mathbb C,\ \Phi_j(x,t)=0,\ x\in\partial\Omega,\ j=1,2\end{aligned}\end{equation}
is  a mathematical model for the binary Bose–Einstein condensate for the unknown condensate wave functions
$\Phi_j$, $j=1,2$. Here $\Omega$ is a bounded smooth domain in $ \mathbb R^2$, and the nonnegative constants
$\mu_j$’s and $\beta_{ij}’s$ are the intraspecies and interspecies scattering lengths which represent the interactions between like and unlike particles, respectively. 
It is natural to assume that $\beta_{ij}$'s are symmetric, i.e. $\beta_{ij}=\beta_{ji}$ if $j\not=i.$
The functions $V_j(x)$, $j=1,2$, represent the magnetic trapping potentials. 
To find solitary wave solutions of the system \eqref{gp}, we set $\Phi_j(x,t)=e^{-{\boldsymbol\iota}\lambda_jt}u_j(x),$  $\lambda_j\in\mathbb R$ and $u_j\in\mathbb R.$
Then we may transform the system \eqref{gp} into a  system of semilinear elliptic equations given by
\begin{equation}\label{s-gp}-\Delta u_j+(V_j(x)+\lambda_j)u_j=\mu_ju_j^3+\sum\limits_{i\not=j}\beta_{ij}u_i^2u_j,\ x\in\Omega,\ 
u_j(x,t)=0,\ x\in\partial\Omega\quad  j=1,2,\end{equation}
which are time independent vector Gross-Pitaevskii/Hartree-Fock equations \cite{11,12} for the condensate wave
functions $u_j$. It was shown in  \cite{25} that there are two distinct scenarios of spatial separation: (i) potential separation, caused
by the external trapping potentials in much the same way that gravity can separate fluids of different specific weight; 
(ii) phase separation, which persists in the absence of external potentials. In the fluid analogy, phase separated
condensates can be compared to a system of two immiscible fluids, such as oil and water.

Actually, in a binary mixture of Bose–Einstein condensates  in the absence of external potentials,  i.e. $V_i=0$, a segregation phenomena  occurs
when    intra species scattering lengths $\mu_j$
are constants and the parameter $\beta:=\beta_{12}$ is   large.  In this case the two states  repel each other and form segregated
domains like the mixture of oil and water.  Such a phenomenon is called phase separation of a binary mixture of
Bose–Einstein condensates  and has been investigated extensively by experimental and theoretical physicists
(\cite{15,22,25}). From a mathematical point of view, a lot of work has been done to study the segregation phenomena (\cite{cl,ctv,sz,tt,ttvw,tv}).
In particular, in \cite{nttv} the authors find the governing equations of the limiting functions  of the bound state solutions of the system \eqref{s-gp} as $\beta\to\infty$:
  if $u_{1,\beta}$ and $u_{2,\beta}$ are $L^\infty(\Omega)-$uniformly bounded solutions of
 \begin{equation}\label{s-be}
  \left\{\begin{aligned}
 &-\Delta u_{1,\beta} + \lambda_1u_{1,\beta}=\mu_1u_{1,\beta}^3-\beta u_{1,\beta}u_{2,\beta}^2\quad \hbox{in}\ \Omega, \\
 &-\Delta u_{2,\beta}+\lambda_2u_{2,\beta}=\mu_2u_{2,\beta}^3-\beta u_{2,\beta}u_{1,\beta}^2\quad \hbox{in}\ \Omega, \\
 &u_{1,\beta}=u_{2,\beta}=0\ \hbox{on}\ \partial\Omega.
 \end{aligned}\right.
  \end{equation}
then, up to a subsequence, as $\beta$ approaches $+\infty$ they converge in $C^{0,\alpha}(\overline\Omega)\cap H^1(\Omega)$ to a pair of functions $u_1$ and $u_2$ having compact disjoint supports (namely $u_1 u_2\equiv 0$ in $\Omega$) which solve
 \begin{equation}\label{s-be-lim}
  \left\{\begin{aligned}
 &-\Delta u_1 +\lambda_1u_1=\mu_1u_1^3 \quad \hbox{in}\ \Omega\cap \{u_1>0\}, \\
 &-\Delta u_2+\lambda_2u_2=\mu_2u_2^3 \quad \hbox{in}\ \Omega\cap \{u_2>0\}, \\
 &u_1 =0\ \hbox{on}\ \partial\(\Omega\cap \{u_1>0\}\),\ u_2 =0\ \hbox{on}\ \partial\(\Omega\cap \{u_2>0\}\).
 \end{aligned}\right.
  \end{equation}
  
It is quite natural to ask if {\em any solutions to the limiting equation \eqref{s-be-lim}  can be seen as the  limiting functions of a bound state solutions of the system \eqref{s-be}}. In  the present paper we  address this question and give a positive answer.\\
\subsection{Statement of the result}  
 Our objective is to  construct  
 phase separating solutions of the more general (than (\ref{s-be}))  system
  \begin{equation}\label{s}
  \left\{\begin{aligned}
 &-\Delta u_1=f(u_1,x)-\beta u_1u_2^2\quad \hbox{in}\ \Omega, \\
 &-\Delta u_2=f(u_2,x)-\beta u_2u_1^2\quad  \hbox{in}\ \Omega, \\
 &u_1=u_2=0\ \hbox{on}\ \partial\Omega.
 \end{aligned}\right.
  \end{equation}
 as the parameter $\beta$ is large enough. Here $\Omega$ is a bounded,  open domain in  $\mathbb R^2$ with  smooth boundary,
   $f\colon \R^2\times\Omega\to \R$ is   sufficiently smooth  and odd in the first variable 
 \begin{equation}
 \label{hip:f odd}
 f(u, x)=-f(-u,x).
 \end{equation}
 and $f$ {\it separates phases} in the following sense:
  
 \begin{definition}\label{def: segregate}
 We say  that the  function $f\in C^{1}(\R\times\Omega)$ separates phases in $\Omega$
 if  the problem
    \begin{equation}\label{p}
 -\Delta w=f(w,x)\ \hbox{in}\ \Omega,\ w=0\ \hbox{on}\ \partial\Omega
  \end{equation}
 has a solution $w$ such that   $\Gamma=\{x\mid w(x)=0\}\subset\Omega$ is a regular, simple closed curve dividing $\Omega$ into disjoint, open  components $\Omega_i$, $i=1,2$ with $\partial\Omega_i\cap \Omega=\Gamma$, $\Omega=\Gamma\cup\Omega_1\cup\Omega_2$  and 
 \begin{equation}\label{omega}
   \partial_\nu w(x):=\omega(x) > 0,
   \quad \hbox{on}\  \Gamma,
   \end{equation}
  where $\nu$ denotes the choice of the unit normal to $\Gamma$ exterior to the fixed component of $\Omega$.
\end{definition}
To explain the definition let us suppose that $\nu$ is the exterior to $\Omega_2$ so that $w>0$ in $\Omega_1$ and $w<0$ in $\Omega_2$. Since $f$ is odd in the first variable it holds
$$
-\Delta w =f(w, x), \quad \mbox{in}\quad \Omega_1\quad\hbox{and}\quad
-\Delta (-w) =-f(-w,x),   \quad \mbox{in}\quad \Omega_2.
$$
The vector function 
\begin{equation}
\label{def: w0}
\bs w^0=(w_1, w_2),\qquad w_1(x)= w(x)\mathbbm{1}_{\Omega_{1}}(x)\quad \hbox{and}\quad w_2(x)= -w(x)\mathbbm{1}_{\Omega_{2}}(x),
\end{equation}
would be a smooth solution to (\ref{s}) if not for the jump of the derivative along $\Gamma$.  
As we will show in the rest of the paper modifying $\bs w^0$ suitably near $\Gamma$ the non-continuity of its derivative can be remedied by adding to it a very small function.  The solution of (\ref{s}) obtained this way represents a two component system whose phases are separated along $\Gamma$. To carry out the construction we also need the following  non-degeneracy condition.

  \begin{definition}\label{def nondegener}
 We  say that the phase separating  solution to the  problem (\ref{p}) is non-degenerate if:  
 \begin{itemize}
 \item[(a)]
 Each of the following linear problems 
  \begin{equation}\label{def-non-de_0} 
  \begin{aligned}
    -\Delta \psi &=f_{u}(w_i,x) \mathbbm{1}_{\Omega_{i}}\psi\quad  \hbox{in}\ \Omega_i,\\
    \psi&=0 \quad \hbox{on}\quad \partial\Omega_i,
        \end{aligned} 
  \end{equation}
$i=1,2$, has only the trivial solution. 
\item[(b)]
The problem 
 \begin{equation}\label{def-non-de} 
  \begin{aligned}
    -\Delta \psi &=\left[f_{u}(w_1,x) \mathbbm{1}_{\Omega_{1}}+f_{u}(w_2,x) \mathbbm{1}_{\Omega_{2}} \right]\psi\quad  \hbox{in}\ \Omega,\\
    \psi&=0 \quad \hbox{on}\quad \partial\Omega,
        \end{aligned} 
  \end{equation}
has only the trivial solution. 
\end{itemize}

 \end{definition}


With all the above taken into account the main result of this paper is:
\begin{theorem}\label{thm:main}
Suppose that the function $f\in C^{3,\gamma}(\R\times\Omega)$,  with some $0< \gamma< 1$,  and that  $w\in H^s(\Omega)$, $s>11/2$  is a  non-degenerate  solution to the  problem  (\ref{p}). Then the system (\ref{s}) has a solution $(u_1, u_2)$ such that as $\beta \to \infty$
\begin{equation}\label{est: calpha}
\|w_i-u_i\|_{C^\alpha(\Omega)}=\mathcal O(\beta^{-(1-\alpha)/4}), \qquad 0\leq \alpha< 1,
\end{equation}
and 
\begin{equation}\label{est: c2alpha}
\|w_i  -u_i\|_{C^{2,\alpha}(K)} =\mathcal O(\beta^{-1/4}),
\end{equation}
over the compacts  $K\subset \Omega\setminus \Gamma$ with the additional restriction $0<\alpha<\frac{1}{2}$.
\end{theorem}

  \begin{remark}We  believe that the non-degeneracy condition of the phase separating solution as stated in Definition \ref{def nondegener} is true for generic functions $f$ or for generic domains $\Omega$. The proof should rely on some classical transversality argument as in \cite{q,st,u}. 
 \end{remark}

\begin{remark}
Our theorem can be in  extended to the case when $\Gamma$ consists of two or more disjoint, simple closed curves as long as it divides $\Omega$ into two disjoint components. To avoid complicated notations and minor technical difficulties we will carry out the proof for a simple, closed curve.
\end{remark}
 \begin{remark} We conjecture that our result also holds in a $d-$dimensional domain, once  the curve $\Gamma$ \eqref{def: segregate} is replaced by a $(d-1)-$dimensional regular manifold. The construction of the solution close to the $\Gamma$ is  obtained more or less in  the same way. What is much more difficult is the gluing of the one dimensional profile around the manifold with the phase separating solution.  The radial case, i.e. $\Omega$ is the unit ball in $\mathbb R^d,$ $d\ge1$, when  $f$   as in the Example \ref{ex 1.3} does not depend on $x$, 
  has been solved  by Casteras and Sourdis \cite{cs}.\end{remark}
 \begin{remark} In principle our result can be extended to systems with $m$ equations at least in a two-dimensional setting. In this case the function $f$ 
   separates $m$ phases in $\Omega$ according to   \eqref{def: segregate}, i.e. 
    the set   $\Gamma=\{x\mid w(x)=0\}\subset\Omega$ is a nonempty, one dimensional boundary set dividing $\Omega$ into $m$ disjoint, open  components $\Omega_i$, $i=1,\ldots,m$ such that  $\Omega=\(\cup_{i=1}^m\Omega_i\)\cup\(\cup_{i=1}^{m-1}\Gamma_i\),$ $\Gamma_i:=\partial\Omega_i\cap\partial\Omega_{i+1}\cap\Omega$
 and
$    \partial_\nu w(x)=\omega_i(x) > 0$ along each  curve
   $\Gamma_i$.
  The profile of the solution close to each 
curve $\Gamma_i$ is the usual one-dimensional profile and far away it is nothing but the multiphase separating solution.  \end{remark}
 
 Before discussing several examples where our Theorem applies some comments are in place. As for  the smoothness assumption on $f$ we require that $f\in C^{3,\gamma}(\R\times \Omega)$ (these assumptions can be somewhat weakened, see Example \ref{ex 1.3}).
As a consequence  the function $w_1(x)+w_2(x)$ has Lipschitz continuous extension to the whole $\Omega$ and so the function 
\[
f_{u}(w_1,x) \mathbbm{1}_{\Omega_{1}}+f_{u}(w_2,x) \mathbbm{1}_{\Omega_{2}}.
\] In what follows we will need to calculate various functions defined on the curve $\Gamma$ and  their derivatives, for instance  the function  $\omega$, the curvature $\kappa$ of $\Gamma$ etc. 
We will assume, without being specific,  that they are as regular as needed. Also, it will be evident from the discussion that  $w\in H^s(\Omega)$ with   $s\geq 11/2$ is sufficient to guarantee enough smoothness for our calculations. Note that {\it a priori} we have $\omega\in H^{9/2}(\Gamma)\cap C^{3,\alpha}(\Gamma)$, $\alpha\in (0, \frac{1}{2})$.  

 In principle it is possible and physically justified to allow $\Gamma\cap\partial\Omega$ to consist for example of a discrete set of points.  This leads to serious, but not unsurmountable technical difficulties however by definition $\Gamma\cap\partial\Omega=\emptyset$. To keep the paper at reasonable length we chose not to deal  with the case $\Gamma\cap\partial\Omega\neq\emptyset$ here.

The curve  $\Gamma$ divides $\Omega$ into two disjoint components $\Omega_1$ and  $\Omega_2$. We agree that the choice of $\nu$ in (\ref{omega}) is made in such a way that it is outward to the fixed set $\Omega_i$ at all points of $\Gamma$. 

As  stated for the purpose of the definition  the problem (\ref{p}) is supplied with the  homogeneous Dirichlet  boundary conditions on $\partial\Omega$ in order to determine the function $w$. In general the type of boundary conditions  should  depend on the particular physical context. The most common choice would be to impose either the Dirichlet or the Neumann boundary conditions or some combination of the two. Our method allows for some flexibility in the choice of the boundary conditions and can be easily adopted to deal with non homogeneous  Dirichlet boundary conditions 
  \begin{equation}\label{cond:bdry}
 w=g \quad \hbox{on}\quad \partial\Omega.
 \end{equation}

We will now discuss some  examples. 
 \begin{example}\rm Suppose that $f(u,x)=f(u)$ where $f$ is odd and $|f(u)-f'(0) u|\leq c|u|^p$, $p\geq 2$. Let $w$ be a sign-changing solution to
   \begin{equation}\label{p 2}
  \begin{aligned}
 -\Delta w&=f(w)\quad \hbox{in}\quad  \Omega, \\
 w&=0\ \hbox{on}\quad  \partial\Omega.
 \end{aligned}
  \end{equation}
  Assume that
  $ \Gamma:=\{x\in\Omega\mid w(x)=0\}$ is a regular, simple closed curve not intersecting  the boundary $\partial\Omega$.
   Then  we may set  $\Omega_1:=\{x\in\Omega\mid w(x)>0\}$, $\Omega_2=\mathrm{int}\,(\Omega\setminus \Omega_1)$. 
  \end{example}
\begin{example} \rm
Suppose that $f\equiv 0$ and that the domain $\Omega$ is at least doubly connected and  that its boundary consists of smooth,  simple closed curves $G_j$, $j=1,\dots, k$, $k\geq 1$ and $\partial\Omega$. Let us fix one of the components of the boundary, say $G_k$ and consider the following problem
\[
\begin{aligned}
-\Delta w&=0, \quad \mbox{in}\quad \Omega\\
w&=-1, \quad \mbox{on}\quad  \partial\Omega\cup G_1\cup\dots\cup G_{k-1}\\
w&=1, \quad \mbox{on}\quad G_k.\\
\end{aligned}
\]
By the choice of the boundary conditions  $\Gamma=\{x\in\Omega\mid w(x)=0\}$ is non empty and it is reasonable to assume that at least generically it will be a smooth, simple, close curve.  
\end{example}

  \begin{example}\label{ex 1.3}\rm
  Another possible situation of the binary  phase separation is given by the solution of the following system
  \[
  \begin{aligned}
  -\Delta w_1&=g(w_1, x), \quad\mbox{in}\quad \Omega_1\\
  -\Delta w_2&=h(w_2, x), \quad\mbox{in}\quad \Omega_2
  \end{aligned}
  \]
  such that 
  \[
  w_1=0=w_2, \quad\mbox{on} \quad \partial \Omega, \qquad \partial_\nu w_1=\partial_\nu w_2>0, \quad\mbox{on}\quad \Gamma 
  \]
  In this case the function $w=\mathbbm{1}_{\Omega_1} w_1+\mathbbm{1}_{\Omega_2} w_2$ is a weak solution of (\ref{p}) with 
  \[
  f(u,x)=\mathbbm{1}_{\Omega_1} g(u,x)+\mathbbm{1}_{\Omega_2} h(u,x).
  \]
  Strictly speaking the regularity hypothesis in the Definition \ref{def: segregate} may not be satisfied in this case. This can be remedied by assuming additional smoothness of the function $\omega$ and the curve $\Gamma$. If for instance  $\nu$ is exterior to $\Omega_2$  assuming that $h(u,x)$ is odd in $u$ we can define
\[
\bs w^0 =(\mathbbm{1}_{\Omega_1} w_1,-\mathbbm{1}_{\Omega_2} w_2)
\]
as the model of the phase separating solution of  (\ref{s}). Our result generalizes straightforwardly  to  this case at the expense of some additional technical hypothesis. 
  \end{example}
  
The method of the proof of  Theorem \ref{thm:main} is in part motivated by  the approach in \cite{KPV} and relies on careful separation of the problem into the outer equation 
whose solution is approximately $\bs w^0$, given in (\ref{def: w0}), and the inner equation whose one dimensional, leading order solution is given by a suitable scaled solution of the ODE system (\ref{sis}) below. The challenge is to combine them locally, near $\Gamma$ in a smooth way. To this end we will introduce lower order corrections to the inner and the outer solution as well as some  modulations functions. The latter are needed to deal with the problem of small divisors of the linearized inner equation.  

 This paper is organized as follows: in section \ref{sec dtn} we introduce the Dirichlet-to-Neumann map needed to improve the initial approximation and guarantee the  smallness of its error. In sections \ref{subsec: ode}--\ref{sec proof Lemma w} we study the one dimensional, inner solution. Section \ref{sec: 3} is devoted to the construction of the approximate solution. The proof of the main result is carried out in section \ref{sec: thm main}. Lemma  \ref{lem phi} is  proven in section \ref{pr lem phi}.  Finally in section \ref{sec: 5} we prove an auxiliary result needed in the proof of Proposition  \ref{prop: inner lineal}.

 In this paper we will use $c$ to denote a positive constant whose value may change from line to line. The symbol $g_1 \lesssim g_2$ will be used to mean  $g_1\leq cg_2$ when  comparing two quantities.  By $1>\alpha\geq 0$ we will denote the exponent  in the H\"older  norms $C^{k, \alpha}$. We will not specify the precise value of $\alpha$  unless necessary and it may vary from line to line.

  \section{The linear theory for the outer and the inner problem}
  \setcounter{equation}{0}
    {\color{black}{\subsection{The Dirichlet-to-Neumann map and the nondegeneracy condition}\label{sec dtn}

We will show that  the non-degeneracy condition (\ref{def-non-de}) implies invertibility of certain operator defined   in terms of the Dirichlet-to-Neumann maps associated with the linearization of   (\ref{p}) around $w_i$, $i=1,2$. To define this map we introduce two symmetric,   bilinear forms $\bs a_i$ in $H^{1/2}(\Gamma)$ as follows 
  \[
 \bs a_i(g_1,g_2)=\int_{\Omega_i}\nabla \varphi_1\cdot\nabla \varphi_2 - f_{u}(w_i,x)\varphi_1\varphi_2,
 \]
  where $\varphi_{k{|_\Gamma}}=g_k$, and
   \begin{equation}\label{def dton} 
  \begin{aligned}
  -\Delta\varphi_k  &=f_{u}(w_i,x)\varphi_k \quad  \hbox{in}\ \Omega_i,\\
    \varphi_k&=0 \quad  \hbox{on}\ \partial\Omega\cap \partial\Omega_i.
    \end{aligned} 
  \end{equation}
{\color{black}{By the nondegeneracy assumption (Definition \ref{def nondegener} (a)) we see  that $\varphi_k$ satisfying the above conditions is unique, or in other words $g_k\equiv 0$ implies $\varphi_k\equiv 0$. }}  Then, by definition,  $b=\bs D_i(g)$ is the value of the Dirichlet-to-Neumann map of $g$ on $\Gamma$ if 
  \[
 \bs a_i(g,h)=\int_\Gamma bh, \qquad \forall h\in H^{1/2}(\Gamma).
 \] 
 Using the bilinear form $\bs a_i$ and the standard functional analytic argument for $i=1,2$ we can obtain $\bs D_i$ as  densely defined, self-adjoint operator in $L^2(\Gamma)$. The operator we are interested in here is actually the sum of the two $\bs D=\bs D_1+\bs D_2$. Our condition is:
 \begin{equation}
 \label{cond on D}
 \mbox{the operator}\ \bs D\colon \mathrm{Dom}(\bs D)\to L^2(\Gamma)\ \mbox{has bounded inverse}.
 \end{equation} 
We claim  that (\ref{def-non-de}) implies (\ref{cond on D}). Indeed it suffices to show that 
\[
\bs D(g)=0\Longrightarrow g\equiv 0.
\]
By the definition of the  maps $\bs D_i$  there exist functions $\varphi_i\in H^2(\Omega_i)$, $\varphi_{i{|_\Gamma}}=g$ such that 
\[
\int_\Gamma (\partial_{n_1}\varphi_1+\partial_{n_2}\varphi_2)h=0, \qquad \forall h\in H^{1/2}(\Gamma).
\]
It follows that 
\[
\varphi=\varphi_1\mathbbm{1}_{\Omega_{1}}+\varphi_2\mathbbm{1}_{\Omega_{2}}
\]
is a weak solution of (\ref{def-non-de}) and thus $\varphi\equiv 0$ {\color{black}{(condition (b) in Definition \ref{def nondegener})}} which proves the claim. 
  }}

   \subsection{The ODE system}\label{subsec: ode}
  
 In this section we describe some basic properties of the system 
 \begin{equation}\label{sis}
  \left\{\begin{aligned}
 -V_1''+V_1V_2^2&=0\quad \hbox{in}\ \mathbb R, \\
 -V_2''+V_2V_1^2&=0\quad  \hbox{in}\ \mathbb R.\\
 \end{aligned}\right.
  \end{equation}
 In \cite{blwz} it has been proved that there exists a unique solution to this system such that $V_1,V_2 >0$, in $\mathbb R$ and
   $V_1(0)=V_2(0)=1,$  $V_1(t)=V_2(-t),$  
   $V'_1(t)>0,$  
  \begin{equation}\label{v}
  \begin{aligned}
  V_1(t)&=A t +B+\mathcal O\(e^{-ct^2}\)\ \hbox{as}\ t\to+\infty,\\ 
  V_1(t)&= \mathcal O\(e^{-ct^2}\)\ \hbox{as}\ t\to-\infty,
  \end{aligned}
  \end{equation}
   for some $A>0,$  $B \in\mathbb R$ and $c>0.$
 
 Let $\bs L_0$ be the linearization of (\ref{sis}) around $\bs V=(V_1, V_2)$:
  \[
 \bs {L}_0 =\left(\begin{array}{cc} -\frac{d^2}{dx^2}+ V_2^2 & 2 V_1 V_2\\
 2 V_1 V_2 & -\frac{d^2}{dx^2}+ V_1^2 
 \end{array}\right).
 \]
Note that (\ref{sis}) is invariant under translation $\tau\mapsto \bs V(\cdot- \tau)$ and scaling $\mu\mapsto \mu\bs V(\mu\cdot)$.   
  This has a simple consequence for the  homogeneous system
   \begin{equation}\label{sis-lin}
  \bs L_0 \bs Z=0, \quad \hbox{in}\quad \R, \quad 
  \bs Z=(Z_1, Z_2).
  \end{equation}
  Indeed 
    \begin{equation}
    \label{sis-lin-sol}
    \bs X=\bs V'
   \qquad \hbox{and}\qquad  \bs Y=t \bs V'+\bs V, \end{equation}
   which are the derivatives of the solution with respect to the invariance parameters, belong to the fundamental set of the linearized operator.
More precise information about this operator, which was proven in \cite{blwz} (Proposition 5.2) is summarized below. 
\begin{lemma}
\label{lin properties}
\begin{enumerate}
\item The only bounded solution of the problem $\bs L_0 \bs Z=0$ in $\bR$ is $\bs X=\bs V'$. 
\item The solution $\bs V$ is stable, in other words for any $\bs f\in C^\infty_0(\bR)$ we have
\[
\langle \bs L_0 \bs f, \bs f\rangle\geq 0.
\]
\end{enumerate}
\end{lemma}
   
For future use we point out that $V'_1(t)=-V_2'(-t)$  and  
  \begin{equation}\label{z1}
  \begin{aligned}
  V'_1(t)&=A+ \mathcal O\(e^{-ct^2}\)\ \hbox{as}\ t\to+\infty,\ \\ 
  V'_1(t)&= \mathcal O\(e^{-ct^2}\)\ \hbox{as}\ t\to-\infty,
  \end{aligned}
  \end{equation}
  and
   \begin{equation}\label{z2}
   \begin{aligned}
  tV'_1(t)+V_1(t)&=2At+B +\mathcal O\(e^{-ct^2}\)\ \hbox{as}\ t\to+\infty,\ \\ 
   tV'_1(t)+V_1(t)&= \mathcal O\(e^{-ct^2}\)\ \hbox{as}\ t\to-\infty.
   \end{aligned}
   \end{equation}
Our aim is to build a positive solution $(u_1,u_2)$ to \eqref{s} which looks like  $(w_1,w_2)$ far away from $\Gamma$ (the outer solution) and 
a suitable scaling of $(V_1,V_2)$  close to $\Gamma$ (the inner solution).

 \subsection{The inner linear operator} \label{sec: line oper}

We recall  that $\Gamma$  is a regular, simple closed curve. Denote its  length  by $|\Gamma|$.  By $s\in [0, |\Gamma|]$ we will denote the arc length parameter on $\Gamma$.  Let $b_0\colon [0, |\Gamma|]\to \bR$ be  defined by
\begin{equation}
\label{b}
b_0(s)=\sqrt{\omega(s)\over A},
\end{equation}  
see \eqref{omega} for the definition of the function $\omega$ and  \eqref{v} for the definition of the constant $A$.
For notational reasons it is convenient to introduce a small parameter
\[
\epsilon=\frac{1}{\qb}.
\]
Additionally we introduce two more periodic  functions $b_1, b_2\in C^{3,\alpha}([0,|\Gamma|])$ and define
\[
b(s)=b_0(s)+\eps b_1(s)+\eps^2 b_3(s).
\]
The results of the present section do not depend on the specific properties of the function $b_1, b_2$ and we postpone their precise definitions until section  \ref{sec:match}.
Let $\pC$ be the right cylinder $\R\times[0, |\Gamma|]$. 
Denote
\begin{equation}
\label{def:inner lin oper}
 \bs {\mathfrak L} =\left(\begin{array}{cc} -\frac{\partial^2}{\partial x^2}-\eps^{2}b^{-2}\left(y\right)\frac{\partial^2}{\partial y^2}+ V_2^2 & 2 V_1 V_2\medskip\\
 2 V_1 V_2 & -\frac{\partial^2}{\partial x^2}-\eps^{2}b^{-2}\left(y\right)\frac{\partial^2}{\partial y^2}+ V_1^2 
 \end{array}\right)
 \end{equation}
In this section we will study the  equation 
\begin{equation}
\begin{aligned}
\bs {\mathfrak L}\bs \varphi&=\bs g, \qquad \mbox{in}\qquad \pC.
\end{aligned}
\label{eq: inner 1}
\end{equation}
In other words we suppose that  the function $\bs \varphi$ satisfies  periodic boundary conditions at  $y=0$ and $y=|\Gamma|$.


Consider the following eigenvalue problem
\begin{equation}
\label{eq: per 1}
\begin{aligned}
-\psi''=\omega^2  b^2(y) \psi, &\qquad \mbox{in}\qquad (0,  |\Gamma|),\\
\psi(0)=\psi(|\Gamma|), &\qquad  \psi'(0)=\psi'(|\Gamma|).
\end{aligned}
\end{equation}
We denote  the eigenvalues and eigenfunctions respectively by $\omega^2_{k}$ and $\psi_{k}$, $k=0, \dots$.
We note that $\omega_0=0$ is the eigenvalue corresponding to $\psi_0\equiv\frac{1}{\sqrt{|\Gamma|}}$. {\color{black}{For large $k$, by Weyl's formula \cite{leviatan} (Corollary 1.41), we have
\begin{equation}
\label{eq: per 2}
\omega_{2k}^2=\left(\frac{2\pi k}{\ell_0}\right)^2+\mathcal O\left(\frac{1}{k^2}\right),\qquad \omega_{2k-1}^2=\left(\frac{2\pi k}{\ell_0}\right)^2+\mathcal O\left(\frac{1}{k^2}\right)\qquad k\to \infty,
\end{equation}
with 
\[
\ell_0=\int_0^{|\Gamma|} |b(s)|\,ds.
\]
It is then convenient to relabel the eigenvalues using  the map
\[
2k\longmapsto k,\qquad 2k-1\longmapsto -k.  
\]
Accordingly, from now on we will denote the eigenvalues by $\omega_k$ and eigenfunctions by $\psi_k$ with  $k\in \mathbb Z$. 
}}

First we will need some preliminary results.
Given a function $\bs g\in L^2(\pC)$, we expand it  in terms of the eigenfunctions of (\ref{eq: per 1}):
\begin{equation}
\label{exp:g}
\bs g(x,y)=\sum_{{\color{black}{k=-\infty}}}^\infty \bs g_k(x) \psi_{k}(y).
\end{equation}
We look for a solution of (\ref{eq: inner 1}) separating  variables in the form
\begin{equation}
\label{eq: inner 2}
\begin{aligned}
\bs L_{\eps\omega_{k}}\bs \varphi_k&=\bs g_k, \qquad \mbox{in}\qquad \pC,
\end{aligned}
\end{equation}
where
 \[
 \bs {L}_\omega =\left(\begin{array}{cc} -\frac{d^2}{dx^2}+ V_2^2 & 2 V_1 V_2\\
 2 V_1 V_2 & -\frac{d^2}{dx^2}+ V_1^2 
 \end{array}\right)+\omega^2\Id
 \]
and $\omega\geq 0$. 
In what follows  we will study the following Dirichlet problem:
 \begin{equation}
 \label{R problem}
 \begin{aligned}
 \bs L_\omega \bs \varphi &=\bs g, \qquad \mbox{in}\ (-R, R),\\
 \bs \varphi(\pm R)&=0,
 \end{aligned}
 \end{equation}
 where $R>0$. A solution of  (\ref{eq: inner 2})   will be obtained taking $R\to \infty$. 
{ \color{black}{Below we will show  that $\bs L_\omega$ satisfies the $(\psi)$ property (\cite{defig} Definition 1.1)  hence the existence and uniqueness of solutions for (\ref{R problem}).} }
 Our goal is to obtain uniform in $\omega\in [0,\infty)$ estimates for the solution and for this purpose we will study (\ref{R problem}) in the case when $\bs g\in C^\alpha((-R,R))$ is such that its weighted H\"older norm defined by
 \[
 \|\bs g\|_{C^\alpha_\theta((-R,R))}=\|\bs g\cosh(\theta x)\|_{C^0((-R,R))}+\sup_{-R<x<R}\cosh(\theta x)\left[\bs g\right]_{\alpha,(x, x+1)}
 \]
is bounded. Obviously if $\bs g \in C^\alpha_\theta((-R,R))$ then
\[
|\bs g(x)|\leq  \|\bs g\|_{C^\alpha_\theta((-R,R))} \sech(\theta x).
\]
{\color{black}{Now we recall the basic fact about the systems satisfying the $(\psi)$ property (\cite{defig} Theorem 1.2) , which in the case at hand is the comparison principle of the form:
\[
\bs L_\omega \bs u=\bs F\geq 0, \quad \bs u(\pm R)= 0\Longrightarrow \bs u\geq 0.
\]
To check the $(\psi)$ property  we define $\bar{\bs u}(x)=(\bar u_1(x), \bar u_2(x))$ by 
\begin{equation}
\label{def bar u}
\begin{aligned}
\bar u_1(x)&= K \int _{-\infty}^x\!\int_{-\infty}^y \sech(a t)\,dtdy+M\int_{-\infty}^x \sech(a y)\,dy, \\
 \bar u_2(x)&= K \int ^{\infty}_x\!\int^{\infty}_y \sech(a t)\,dtdy+M\int^{\infty}_x \sech(a y)\,dy.
\end{aligned}
\end{equation}
{\color{black}{Note that $\bar{\bs u}(x)> 0$ if $K,M$ are positive}}.  Because of the symmetry of the operator $\bs L_\omega$ due to  $V_1(x)=V_2(-x)$ to check that  positive constants $a$, $K, M$,  can  be chosen
so that $\bar{\bs u}$ is a {\color{black}{positive}} supersolution of our problem it suffices to do this for $x\geq 0$.  
We have
\[
\begin{aligned}
\bar u''_1(x)&=K\sech(ax)-Ma \tanh(ax)\sech(ax),\\
\bar u''_2(x)&=K\sech(ax)+Ma \tanh(a x)\sech(a x),
\end{aligned}
\]
and using 
\[
V_1(x)\geq (mx+l), \quad m, l>0, x>0, 
\]
and \[
  V^2_2(x)\geq \frac{1}{2}\mathbbm{1}_{[-\delta, \delta]}(x), \qquad \mbox{for some}\qquad \delta>0,
  \]
we find   
\[
\begin{aligned}
\bs L_\omega\bar{\bs u}&\geq \left(\begin{array}{c}
-K\sech(ax)+M a \tanh(a x)\sech(a x)+V_2^2\left(K \int _{-\infty}^x\!\int_{-\infty}^y \sech(a t)\,dt+M\int_{-\infty}^x \sech(a y)\,dy\right)\medskip\\
-K\sech(ax)-Ma \tanh(a x)\sech(a x)+V_1^2  \left(K \int ^{\infty}_x\!\int^{\infty}_y \sech(a t)\,dt+M\int^{\infty}_x \sech(a y)\,dy\right)
\end{array}
\right)\\
&\geq  \left(\begin{array}{c} \mathbbm{1}_{[-\delta, \delta]^c}(x)\left(-2K+Ma \tanh(ax)\right) e^{\,-ax}+\frac{1}{2}\mathbbm{1}_{[-\delta, \delta]}(x)\left(K a^{-2}+Ma^{-1}-4K+Ma \tanh(ax)\right)\medskip
\\
\left[-2(K+Ma) +(mx+l)^2\left(K a^{-2} +Ma^{-1} \right)\right] e^{\,-a x}
\end{array}\right)
\end{aligned}
\]
Take 
\begin{equation}
\label{def M}
M=\frac{8K}{a\tanh(a\delta)}, \qquad a\leq \min\left\{1/4,l\tanh^{1/2}(\delta a)(2\tanh(\delta a)+8)^{-1/2}\right\}
\end{equation}
to get
\begin{equation}
\label{eq: comp 1} 
\bs L_\omega\bar{\bs u}\geq  CK e^{\,-ax}\bs 1, \qquad x\geq 0,
\end{equation}
where $C$ is a constant depending on $a$ and $\delta$ and $\bs 1=(1,1)$. The constant $K$ at this point is  arbitrary.
Similar argument for $x\leq 0$ gives
\[
\bs L_\omega\bar{\bs u}\geq  CK \sech(ax)\bs 1, \qquad -\infty<x<\infty.
\]
This shows that the $(\psi)$ property holds.

Denote by $\bs\varphi_R$ the solution of (\ref{R problem}) with the right hand side $\bs g_R\in C^\alpha_\theta((-R, R))$ such that $ \|\bs g\|_{C^\alpha_\theta((-R,R))}=1$. Restricting $a$ further if necessary 
\[
a\leq\min  \left\{1/4,l\tanh^{1/2}(\delta a)(2\tanh(\delta a)+8)^{-1/2}, \theta\right\},
\]
and taking $K$ large we get
\[
\bs L_\omega (\bar{\bs u}-\bs \varphi_R)(x)\geq 0, \qquad -R<x<R.
\]
Using the comparison principle we find
\[
\bs \varphi_R(x)\leq C\bar{\bs u}(x).
\]
The lower bound can be obtained similarly. Finally, with $\bs g_R\in C^0_\theta((-R, R))$ arbitrary  we find 
\begin{equation}
\label{eq: comp 2}
|\bs \varphi_R(x)|\leq C\|\bs g_R\|_{C^0_\theta((-R,R))}|\bar{\bs u}(x)|.
\end{equation}
}}

In case that  the functions $\bs g_R$ converge in $C^0_\theta(\R)$ locally (e.g. $\bs g_R=\mathbbm 1_{[-R,R]}\bs g$ for some $\bs g\in  C^0_\theta(\R)$) we can use the estimate (\ref{eq: comp 2}), standard regularity arguments and the Arzela-Ascoli Theorem to pass to the limit locally in $C^0_\theta(\R)$ for $\bs \varphi_R$ and get a weak solution to 
\begin{equation}
\label{eq: comp 3}
\bs L_\omega \bs\varphi=\bs g, \qquad \mbox{in} \qquad \R,
\end{equation}
such that
\begin{equation}
\label{eq: comp 4}
|\bs \varphi(x)|\leq C\|\bs g\|_{C^0_\theta(\R)}|\bar{\bs u}(x)|.
\end{equation}  
Again, standard arguments give that $\bs \varphi$ is a classical solution. 
From (\ref{eq: comp 4}) we see that the function $\bs \varphi$ is locally bounded by $C\|\bs g\|_{C^0_\theta(\R)}$ but in principle $|\varphi_1(x)|$ ($|\varphi_2(x)|$) may grow linearly as $x\to \infty$ ($x\to -\infty$ respectively). We will show that this does not happen when $\omega=0$ and  if we suppose additionally that 
\begin{equation}
\label{cond:orto g}
\int_\R  \bs g\cdot \bs X\,dx=0, \qquad \int_\R  \bs g\cdot \bs Y\,dx=0
\end{equation}
(Recall $\bs X=\bs V'$, $\bs Y=( x\bs V'+\bs V)$). \blu{The following Lemma can be found in \cite{sourdis 1}. For completeness we include it here together with an alternative demonstration. } 
\begin{lemma}
\label{lem:decay est}
There exists a solution  $\bs \varphi$  to (\ref{eq: comp 3}) with $\omega=0$ satisfying (\ref{eq: comp 4})  such that we have
\begin{equation}
\label{est:phi weight}
\|\bs\varphi\|_{C^2_\theta(\R)}\leq C\|\bs g\|_{C^0_\theta(\R)},
\end{equation}
provided that  the orthogonality conditions (\ref{cond:orto g}) hold.
\end{lemma}
\begin{proof}
We claim that there exists $\bar R>0$  such for all $R>\bar R$   if $\varphi_R$ is a solution of (\ref{R problem}) then 
\begin{equation}
\label{est:phi R weight}
\|\bs\varphi_{R}\|_{C^0_\theta((-R,R))}\leq C\|\bs g_R\|_{C^0_\theta((-R,R))}
\end{equation}
provided that it holds 
\begin{equation}
\label{cond:orto g R}
\int_{-R}^R  \bs g_R\cdot \bs X\,dx=0, \qquad \int_{-R}^R  \bs g_R\cdot \bs Y\,dx=0.
\end{equation}
To prove the claim we argue by contradiction. We suppose that there exists a sequence  $R_n\to \infty$ and sequences of functions $\bs \varphi_n$ and $\bs g_n$ that satisfy
\begin{equation}
\label{hip:contr n}
\|\bs g_n\|_{C^0_\theta((-R_n,R_n))}\to 0, \qquad \|\bs\varphi_{n}\|_{C^0_\theta((-R_n,R_n))}=1.
\end{equation}
By (\ref{eq: comp 2}) we have
\begin{equation}
\label{est:phi n}
|\bs \varphi_n(x)|\leq \min\{\|\bs g_n\|_{C^0_\theta((-R_n,R_n))}\bar {\bs u}(x), e^{\,-\theta|x|}\}, \qquad x\in [-R,R].
\end{equation}
By (\ref{cond:orto g R}) we have
\[
\begin{aligned}
\langle \bs L_0\varphi_n, \bs X\rangle&=-\left(\bs\varphi'_n\cdot \bs X\right)(x)\Big|_{-R_n}^{R_n}=0,\\
\langle \bs L_0\varphi_n, \bs Y\rangle&=-\left(\bs\varphi'_n\cdot \bs Y\right)(x)\Big|_{-R_n}^{R_n}=0,
\end{aligned}
\]
hence using the asymptotic behavior of $\bs V$ and $\bs V'$ we get
\[
\begin{aligned}
-A \varphi'_{n,1}(R_n) -A\varphi'_{n,2}(-R_n)&=\(|\bs\varphi'_n(R_n)|+|\bs\varphi'_n(-R_n)|\)\mathcal O\(e^{\,-c R_n^2}\),\\
-(2A R_n+B)\varphi'_{n,1}(R_n)+(-2A R_n+B)\varphi'_{n,2}(-R_n)&=\(|\bs\varphi'_n(R_n)|+|\bs\varphi'_n(-R_n)|\)\mathcal O\(R_ne^{\,-c R_n^2}\).
\end{aligned}
\]
From (\ref{eq: comp 4}) using the equation  we further conclude
\begin{equation}
\label{eq: comp 6}
|\bs \varphi''_n(x)|\leq C\|\bs g_n\|_{C^0_\theta((-R_n,R_n))}\sech(\theta x), \qquad |\bs \varphi_n'(x)|\leq C\|\bs g_n\|_{C^0_\theta((-R_n,R_n))}.
\end{equation}
 It follows 
 \begin{equation}
\label{eq:est der phi}
|\varphi'_{n,1}(R_n)|+ |\varphi'_{n,2}(-R_n)|\leq \|\bs g_n\|_{C^0_\theta((-R_n,R_n))} \mathcal O\(R_ne^{\,-c R_n^2}\).
\end{equation}
Integrating twice on $(x, R_n)$, $x\geq 0$ the equation
\[
-\varphi_{n,1}''(x)=-V_2^2\varphi_{n,1}-2 V_1V_2 \varphi_{n,2}+g_{n,1} 
\]
we get
\[
\varphi_{n,1}(x)=\varphi'_{n,1}(R_n)(R_n-x)+\mathcal O\(\|\bs g_n\|_{C^0_\theta((-R_n,R_n))}\) e^{\,-\theta x}+\mathcal O\(e^{\,-c x^2}\).
\]
It follows that for some large, but independent on $n$ number $\wt x>0$ 
\[
|\varphi_{n,1}(x)|\leq \frac{e^{\,-\theta x}}{4} , \qquad x\in[\wt x, R_n],
\]
for all sufficiently large $n$. Similarly we show
\[
|\varphi_{n,2}(x)|\leq \frac{e^{\,-\theta x}}{4} , \qquad x\in[-R_n,-\wt x].
\]
These last two estimates and (\ref{est:phi n}) show that for $n$ large
\[
|\bs \varphi_n(x)|\leq \frac{1}{2}\sech(\theta x),\qquad x\in [-R_n, R_n],
\]
which contradicts (\ref{hip:contr n}). 

To finish the proof of the Lemma we use the approximation of the problem on $\R$ by the Dirichlet problem (\ref{R problem}) replacing $\bs g$ by 
\[
\bs g_R=\bs g \mathbbm{1}_{[-R, R]}(x)+\lambda_{R,1} e^{\,-x^2}\bs 1+\lambda_{R,2} xe^{\,-x^2}\bs 1,
\]
with $\lambda_{R,i}$, $i=1,2$ chosen so that $\bs g_R$ satisfies the (\ref{cond:orto g R}). Note that 
\[
\lambda_{R, i}=\mathcal O\(R e^{\,-\theta R}\).
\]
Passing to the limit to obtain $\varphi$ satisfying (\ref{lem:decay est}) is standard. 
\end{proof}
Now we consider the problem (\ref{eq: comp 3}) with $\omega>0$. Our goal is to find a solution that decays at an exponential rate independent on $\omega$ and for this  the right hand side of the equation can not be an arbitrary, exponentially decaying function. Naturally we would like to impose the orthogonality conditions (\ref{cond:orto g}) but there is an additional complication coming form the fact that  the functions $\bs X$ and $\bs Y$ are not in the fundamental set of $\bs L_\omega$. The following {\it a priori} estimate settles it.
\begin{corollary}
\label{cor:decay est omega}
There exists $\bar\omega>0$ such that if  $\varphi\in C_\theta^{2,\alpha}(\R)$ is a solution of (\ref{eq: comp 3}) with  $\omega\in[0,\bar\omega]$ satisfying
\begin{equation}
\label{cond:ort omega}
\int_\R  \bs g\cdot \bs X\,dx=\omega\int_\R \bs \varphi \cdot \bs X\,dx, \qquad \int_\R  \bs g\cdot \bs Y\,dx=\omega\int_\R\bs \varphi \cdot \bs Y\,dx,
\end{equation}
then
\begin{equation}
\label{est:phi omega}
\|\bs\varphi\|_{C_\theta^{2,\alpha}(\R)}\leq C\|\bs g\|_{C_\theta^{\alpha}(\R)}.
\end{equation}
\end{corollary}
\begin{proof}
The proof is an easy consequence of Lemma \ref{lem:decay est} if we take $\bar\omega<\frac{1}{\sqrt{2C}}$, where $C$ is the constant on the right hand side of (\ref{est:phi weight}).
\end{proof}

We are in position to solve the equation (\ref{eq: inner 1}).
We introduce the weighted H\"older spaces $C^\alpha_\theta(\pC)$ and $C^{k,\alpha}_{\theta}(\pC)$ equipped with the norms
\[
\begin{aligned}
\|\bs h\|_{C^\alpha_\theta(\pC)}&= \|\bs h\cosh(\theta x)\|_{C^0(\pC)}+\sup_{x\in \R}\cosh(\theta x)\left[\bs h\right]_{\alpha, (x, x+1)\times [0,|\Gamma|]},\\
\|\bs h\|_{C^{k,\alpha}_{\theta}(\pC)}&=\|\bs h\cosh(\theta x)\|_{C^0(\pC)}+\sum_{j=1}^k\sup_{x\in \R}\cosh(\theta x) \bs\| D^j h\|_{\alpha, (x, x+1)\times [0,|\Gamma|]}\\
&\qquad +\sup_{x\in \R}\cosh(\theta x)\left[\bs D^k h\right]_{\alpha, (x, x+1)\times [0,|\Gamma|]},
\end{aligned}
\]
and  the weighted Sobolev spaces $L^2_\theta(\pC)$ and $H_\theta^k(\pC)$ equipped, respectively, with the norms
\[
\begin{aligned}
\|\bs h\|^2_{L^2_\theta(\pC)}&=\int_{\pC}\cosh^2(\theta x) |\bs h(x,y)|^2 b^2_0(y)\,dxdy,\\
\|\bs h\|^2_{H^k_{\theta}(\pC)}&=\|\bs h\|^2_{L^2_\theta(\pC)}+\int_{\pC}\cosh^2({\theta}x)\sum_{j=1}^k|D^j\bs h(x,y)|^2 b^2_0(y)\,dxdy.
\end{aligned}
\]
For any $\bar\omega\leq \frac{1}{\sqrt{2C}}$ (c.f. Corollary \ref{cor:decay est omega}) and  for a given $\eps$ let $\bar K_\eps$ be the largest positive integer   such that 
\begin{equation}
\label{def keps}
\eps\omega_{k}<\bar \omega, \quad\mbox{for}\quad  {\color{black}{|k|}}\leq \bar K_\eps.
\end{equation}
Note that $\bar K_\eps=\mathcal O(\eps^{-1})$ by the Weyl theorem.  We will suppose that $\bs g\in {L^2_\theta}(\pC)$. For such $\bs g$ we consider the expansion (\ref{exp:g}) and write
\[
\bs g={\bs g}^\parallel+\bs g^\perp,
\]
where
\[
{\bs g}^\parallel=\sum_{{\color{black}{k=-\bar K_\eps}}}^{\bar K_\eps} \bs g_k\psi_{k}, \qquad \bs g^\perp=\bs g-{\bs g}^\parallel.
\]
When convenient we will write $P_{K}\bs g$ to be the projection of $\bs g$ on the first $2K$ Fourier modes of $\bs g$ in the expansion (\ref{exp:g}) so that $P_{\bar K_\eps}\bs g={\bs g}^\parallel$ and $\bs g^\perp=(\operatorname{Id}-P_{\bar K_\eps})\bs g$.  With these notations we decompose the equation (\ref{eq: inner 1}) into the following problems
\begin{equation}
\label{eq:modes}
\bs L_{\eps\omega_{k}}\bs \varphi_k=\bs g_k,\qquad\mbox{in}\ \R \qquad {\color{black}{|k|}}=0, \dots, \bar K_\eps
\end{equation}
and 
\begin{equation}
\label{eq:complement}
\bs{\mathfrak L}\bs \varphi^\perp=\bs g^\perp, \qquad \mbox{in}\ \pC.
\end{equation}
\begin{proposition}\label{prop: inner lineal}
Suppose that $\theta<1$ and  $\bs g\in L^2_\theta(\pC)\cap C^{\alpha}_\theta(\pC)$ is decomposed $\bs g={\bs g}^\pr+\bs g^\perp$ as above. 
There exists $\bar \omega>0$, small but $\eps$ independent,  such that with $\bar K_\eps$ defined in (\ref{def keps}) the  following hold:
\begin{itemize}
\item[(i)]
For any $k\in \{{\color{black}{-\bar K_\eps}}, \dots, \bar K_\eps\}$ {\color{black}{ there  exist constants $\lambda_{\bs X, k}$, $\lambda_{\bs Y, k}$ and a  solution $\bs \varphi_k$ of  
\begin{equation}
\label{eq:modes_lagrange}
\bs L_{\eps\omega_{k}}\bs \varphi_k=\bs g_k+\lambda_{\bs X, k}\bs X\sech x+\lambda_{\bs Y, k}\bs Y\sech x,\qquad\mbox{in}\ \R, \qquad {\color{black}{|k|}}=0, \dots, \bar K_\eps
\end{equation}
}}
such that for any {\color{black}{$\alpha\in (0,1)$, $k\neq 0$,}}
\begin{equation}
\label{est:mode holder}
 \|\bs \varphi_k\|_{{C^{2,\alpha}_{\theta}}(\R)}\leq C |k|^{{\color{black}{-\alpha}}}\|{\bs g}_k\|_{C^\alpha_\theta(\R)}.
\end{equation}
Moreover, for 
\[
\bs \varphi^\pr=\sum_{{\color{black}{k=-\bar K_\eps}}}^{\bar K_\eps} \bs \varphi_k\psi_{k}
\]
we have for any $\alpha\in (0,1)$
\begin{equation}
\label{est:glob hold}
\|\bs\varphi^\pr\|_{C^{2,\alpha}_\theta(\pC)}\lesssim \eps^{-1+\alpha}\|\bs g^\pr\|_{C^\alpha_\theta(\pC)}.
\end{equation}
When $g^\pr\in C^{1, \alpha}_\theta(\pC)$ we have
\begin{equation}
\label{est:glob hold_2}
\|\bs \varphi^\pr\|_{C^{2,\alpha}_\theta(\pC)}\lesssim |\ln\eps|\|\bs g^\pr\|_{C^{1,\alpha}_\theta(\pC)}
\end{equation}

\item[(ii)]
Assuming that $\theta<\frac{1}{4}\bar\omega$ there exists a unique solution $\bs \varphi^\perp$   of (\ref{eq:complement}) such that
\begin{equation}
\label{est:c2alpha}
\|\bs \varphi^\perp\|_{C^{2,\alpha}_\theta(\pC)}\leq C\|\bs g^\perp\|_{C^\alpha_\theta(\pC)}.
\end{equation} 
\end{itemize}
\end{proposition}

\begin{proof}
As long as (\ref{cond:ort omega}) is satisfied for each individual mode part (i) follows directly from Corollary \ref{cor:decay est omega} since, {\color{black}{by Proposition \ref{lem apend 1}
\[
|g_k(x)|=\left|\int_0^{|\Gamma|} \bs g(x,y)\psi_{k}(y)b^2_0(y)\,dy\right|\lesssim |k|^{-\alpha}\|\bs g^{\parallel}(x,\cdot)\|_{C^\alpha(0,|\Gamma|)},
\]
hence
\[
\|\bs g_k\|_{{C^{\alpha}_\theta}(\R)}=\left\|\int_0^{|\Gamma|} \bs g(x,y)\psi_{k}(y)b^2_0(y)\,dy\right\|_{{C^{\alpha}_\theta}(\R)}\lesssim |k|^{-\alpha} \|{\bs g}^\pr\|_{C^\alpha_\theta(\pC)}, \qquad {\color{black}{-\bar K_\eps}}\leq k\leq \bar K_\eps, \quad k\neq 0.
\]
}}
However (\ref{cond:ort omega}) does not hold in general and this is why we need to introduce the Lagrange multipliers in (\ref{eq:modes_lagrange}). We fix $k$ and and look for a solution of (\ref{eq:modes_lagrange}) in the form of a fixed point problem as follows
\begin{equation}
\label{eq:modes_lagrange 2}
\bs L_0\bs \varphi_k=-\eps^2\omega_{k}^2\wt {\bs \varphi}_k+\bs g_k+\lambda_{\bs X, k}\bs X\sech x+\lambda_{\bs Y, k}\bs Y\sech x,
\end{equation}
with $\wt{\bs \varphi}_k\in C^{\alpha}_\theta(\R)$ given. We chose 
\[
\begin{aligned}
\lambda_{\bs X, k}&=\left(\int_\R \left(\eps^2\omega_{k}^2\wt{\bs \varphi}_k -\bs g_k\right)\cdot \bs X\right)\left(\int_R|\bs X|^2\sech x\right)^{-1}, \\
\lambda_{\bs Y, k}&=\left(\int_\R \left(\eps^{2}\omega_{k}\wt{\bs \varphi}_k -\bs g_k\right)\cdot \bs Y\right)\left(\int_R|\bs Y|^2\sech x\right)^{-1}.
\end{aligned}
\]
We have 
\[
|\lambda_{\bs X, k}|+|\lambda_{\bs Y, k}|\leq C\left(\eps^2\omega_{k}^2 \|\wt{\bs \varphi}_k\|_{C^\alpha_\theta(\R)}+\|\bs g_k\|_{C^\alpha_\theta(\R)}\right).
\]
By Lemma \ref{lem:decay est} we know that there exists a solution to (\ref{eq:modes_lagrange 2}) such that 
\[
\|\bs\varphi_k\|_{C^\alpha_\theta(\R)}\leq C\left(\eps\omega_{k} \|\wt{\bs \varphi}_k\|_{C^\alpha_\theta(\R)}+\|\bs g_k\|_{C^\alpha_\theta(\R)}\right).
\]
From this and a straightforward implementation of the fixed point argument the first estimate of part (i) follows. {\color{black}{To finish the proof we note that as long as $0\leq k\leq \bar K_\eps\sim \frac{1}{\eps}$ we have (see Lemma \ref{lemm apend unif})
\[
\|\psi_{k}\|_{C^0 (0,|\Gamma|)} \leq C.
\] 
To show (\ref{est:glob hold}) we first prove that 
\begin{equation}
\label{est:glob c0}
\|\bs \varphi^\pr\|_{C^0_\theta(\pC)}\leq C\eps^{-1+\alpha}\|\bs g^\pr\|_{C^{\alpha}_\theta(\pC)}.
\end{equation}
Indeed we have
\[
\begin{aligned}
\|\bs \varphi^\pr\|_{C^0_\theta(\pC_\eps)}\leq \sum_{{\color{black}{k=-\bar K_\eps}}}^{\bar K_\eps} \|\bs\varphi_k\|_{C^0_\theta(\R)} \|\psi_{k}\|_{C^0(0,|\Gamma|)}\lesssim\|\bs g^\pr\|_{C^\alpha_\theta(\pC)} \sum_{{\color{black}{k=1}}}^{\bar K_\eps} k^{-\alpha}\lesssim \eps^{\alpha-1} \|\bs g^\pr\|_{C^\alpha_\theta(\pC)}.
\end{aligned}
\]
From this we get (\ref{est:glob hold}) by  standard elliptic regularity argument.  The second estimate is proven in a similar way using an easy to show fact that the  modes  of a  $C^1(0, |\Gamma|)$  function decay like $k^{-1}$ for $k$ large. 
}}

To show (ii) we fix a mode $\omega\geq \bar\omega$ and consider  (\ref{eq:complement}) projected on this mode
\begin{equation}
\label{eq:psiomega}
\bs L_\omega \bs \varphi^\perp_\omega=\bs g^\perp_\omega.
\end{equation}
We claim that assuming that $\theta<\frac{1}{4}\bar\omega$ there exists a unique solution $\bs \varphi^\perp$   of (\ref{eq:complement}) such that
\[
\|\bs \varphi^\perp\|_{H^2_\theta(\pC)}\leq C\|\bs g^\perp\|_{L^2_\theta(\pC)}.
\]
It is clear that by a standard bootstrap argument we get (\ref{est:c2alpha}) from this.
 
To show the claim we  first show an {\it a priori} estimate. Write 
\[
\wt{\bs \psi}(x)=\cosh(\theta x)\bs\varphi^\perp_\omega(x), \qquad \wt{\bs h}=\cosh(\theta x)\bs g^\perp_\omega(x).
\]
Then 
\begin{equation}
\label{eq:wtpsi}
\bs L_0 \wt{\bs \psi}-2\theta\tanh(\theta x) \wt{\bs\psi}'+\left[\omega^2-\theta^2+2\theta^2\tanh^2(\theta x)\right]\wt{\bs \psi}=\wt{\bs h}.
\end{equation}
Multiply this equation by $\wt{\bs \psi}$ and integrate by parts:
\[
\langle \bs L_0\wt{\bs\psi}, \wt {\bs\psi}\rangle+\int_\R \left[\omega^2-\theta^2+2\theta^2\tanh^2(\theta x)-\theta^2\sech^2(\theta x)\right]|\wt{\bs \psi}|^2\,dx=\langle \wt{\bs \psi}, \wt{\bs h}\rangle
\]
From Lemma \ref{lin properties} (i) we infer
\[
\|\wt{\bs \psi}\|_{H^1(\R)}\leq C\|\wt{\bs h}\|_{L^2_\theta(\R)},
\]
and using the equation we get
\begin{equation}
\label{est:wtpsi}
\|\wt{\bs \psi}\|_{H^2(\R)}\leq C\|\wt{\bs h}\|_{L^2_\theta(\R)}.
\end{equation}
In particular the second order differential operator 
\[
\wt{\bs L}_{\omega, \theta}=\bs L_0-2\theta\tanh(\theta x) \frac{d}{dx}+\left[\omega^2-\theta^2+2\theta^2\tanh^2(\theta x)\right]
\]
is injective. Considering its adjoint
\[
\wt{\bs L}^*_{\omega, \theta}=\bs L_0+2\theta\tanh(\theta x) \frac{d}{dx}+\left[\omega^2-\theta^2+2\theta^2\tanh^2(\theta x)+2\theta^2\sech^2(\theta x)\right]
\]
in a similar way we show that it is also an injective operator and thus $\wt{\bs L}_{\omega, \theta}$ is surjective. From this we get the existence and uniqueness for the equation (\ref{eq:wtpsi}) and from the estimate (\ref{est:wtpsi}) we deduce the existence of solution of (\ref{eq:psiomega}) and also the estimate
\[
 \|{\bs \varphi^\perp_\omega}\|_{H^2_\theta(\R)}\leq C\|{\bs g}^\perp_\omega\|_{L^2_\theta(\R)}.
\]
The proof of (ii) is completed by summing up the Fourier modes and using the Plancherel identity. 
\end{proof}

{\color{black}{\subsection{An auxiliary Lemma}\label{sec proof Lemma w}

The result proven below  will be needed in section \ref{sec ansatz}.  \blu{We remark that  its proof   can be found in \cite{cs}. For completeness we include it here together with a somewhat different proof.}
\begin{lemma}\label{lema w}
There exists a one parameter family of  solution of the problem
\[
\bs L_0\bs \varphi =\bs V', \qquad\mbox{in}\quad \R,
\]
of the form 
\[
\bs \varphi=\bs W+C \bs V',\qquad \bs W=\frac{1}{2} x^2\bs V'+\mathcal O(e^{\,-c|x|^2}),
\]
where $C$ is an arbitrary  constant. Moreover $W_1(x)=-W_2(-x)$. 
\end{lemma}

\begin{proof}
To start the proof of the Lemma let $\bs \varphi=\frac{1}{2} x^2 \bs V'+\bs \psi$. Replacing this in the equation we get
\begin{equation}
\label{eq: lem 3.1 1}
\bs L_0\bs \psi=x\bs V'',
\end{equation}
so now the matter is to solve this last equation for $\bs \psi$ with the required properties.  To this end we solve first the following problem
\begin{equation}
\label{eq: lem 3.1 2}
\begin{aligned}
\bs L_0\bs \psi_R&=x\bs V'', \qquad \mbox{in}\qquad (-R, R),\\
\bs\psi_R(R)&=0=\bs\psi_R(-R).\\
\end{aligned}
\end{equation}
{\color{black}{We recall that the operator $\bs L_0$ satisfies the $(\psi)$ property  hence we get the existence and uniqueness for (\ref{eq: lem 3.1 2}). More precisely we have 
\begin{equation}
\label{eq: lem 3.1 3}
\bs L_0\bs u=\bs F\geq 0, \qquad \bs u(R)= 0, \quad \bs u(-R)= 0\Longrightarrow \bs u\geq 0.
\end{equation}
Next let us consider the function 
\[
\bar{\bs u}=(\bar u_1, \bar u_2), \qquad \bar u_1=K_1 \int_{-\infty}^x e^{\,-a_1 y^2}\,dy,\qquad \bar u_2=K_2\int^{\infty}_x e^{\,-a_2 y^2}\,dy,
\]
with positive constants $K_j$ and $a_j$, $j=1,2$ are to be chosen.  We have
\[
\bs L_0 \bar{\bs u}=\left(\begin{array}{c}
 2K_1a_1x e^{\,-a_1 x^2}+V_2^2 K_1\int_{-\infty}^x e^{\,-a_1 y^2}\,dy+2 V_1 V_2 K_2\int_x^\infty e^{\,-a_2 y^2}\,dy\\
-2K_2a_2 xe^{\,-a_2 x^2}+V_1^2 K_2\int_x^\infty e^{\,-a_2 y^2}\,dy +2 V_1 V_2 K_1\int_{-\infty}^x e^{\,-a_1 y^2}\,dy 
\end{array}\right)
\]
We know that for $x\geq 0$ 
\[
V_1(x)\geq m x+l,\qquad 0\leq V_2(x)\leq C e^{\,-c x^2}, \qquad   |\bs V''(x)|\leq C e^{\,-c x^2},
\]
with some positive constants $m, l, C, c$.
When $x\geq 0$ we get
\[
\bs L_0 \bar{\bs u}\geq \left(\begin{array}{c} 2K_1 a_1x e^{\,-a_1 x^2}\\
-2K_2a_2 xe^{\,-a_2 x^2}+(mx+l)^2 K_2\int_x^\infty e^{\,-a_2 y^2}\,dy
\end{array}\right)
\]
It follows that there exist $K_j>0$ large and  $a_j>0$ small such that  
\[
\bs L_0(\bar{\bs u}-\bs \psi_R)\geq 0, \qquad x\geq 0.
\]
Using the symmetry $V_1(x)=V_2(-x)$ we get similar estimate for $x\leq 0$ 
and since $\bar{\bs u}\geq 0$ at $x=R, -R$  by (\ref{eq: lem 3.1 3}) we get
\[ 
\bs\psi_R(x)\leq \bar{\bs u}(x).
\]
Another comparison argument based on (\ref{eq: lem 3.1 3}) gives a  lower bound (adjusting $K_j, a_j$ if necessary) on $\bar\psi_R$ so  we get
\[
|\bs \psi_R(x)|\leq \bar {\bs u}(x).
\]
Using this estimate and the equation we get
\[
|\bs \psi''_R(x)|\leq C e^{\,-c x^2}\qquad \Longrightarrow\qquad  |\bs \psi''_R(x)|\leq C,\qquad -R<x<R
\]
Finally we note that from $V_1(x)=V_2(-x)$, $V'(x)=-V_2'(-x)$ and the uniqueness of solutions of (\ref{eq: lem 3.1 2}) we get 
\begin{equation}
\label{eq: lem 3.1 4}
\psi_{R, 1}(x)=-\psi_{R,2}(-x).
\end{equation}

}}
Using the Arzela-Ascoli Theorem we conclude that $\bs\psi_R\to \bs \psi$ in $C^1_{loc}(\R)$ where $\bs\psi$ is a bounded solution of
\begin{equation}
\label{eq: lem 3.1 5}
\begin{aligned}
\bs L_0\bs \psi&=x\bs V'', \qquad \mbox{in}\qquad \R.
\end{aligned}
\end{equation}
This solution inherits the symmetry (\ref{eq: lem 3.1 4}), so that $\psi_{1}(x)=-\psi_{2}(-x)$,  and the bounds on  $\bs\psi_R$. Using this and the equation  we have with some positive constants $K, a$
\begin{equation}
\label{eq: lem 3.1 6}
|\psi_1(x)|\leq K \int^x_{-\infty} e^{\,-a y^2}\,dy, \qquad  |\psi_2(x)|\leq K\int^{\infty}_x e^{\,-ay^2}\,dy, \qquad |\bs \psi'(x)|+|\bs \psi''(x)|\leq K e^{\,-ax^2}.
\end{equation} 
Integrating twice the equation  we get also that $\lim_{x\to \infty}\psi_1(x)=c_1 A$ and $\lim_{x\to -\infty} \psi_2(x)=-c_1A$ for some constant $c_1$ hence
\begin{equation}
\label{eq: lem 3.1 7}
\bs \psi(x)=c_1 \bs V'(x)+\mathcal O(e^{\,-ax^2}).
\end{equation} 
{\color{black}{Setting  $\bs \varphi=\frac{1}{2} x^2 \bs V'+\bs \psi+(c-c_1)\bs V'$, with an arbitrary constant $c\in \R$, we obtain the family of  solutions in the form claimed}}. This ends the proof of the Lemma. 
\end{proof}
}}

\section{The approximate solution}\label{sec: 3}

\setcounter{equation}{0}
  \subsection{Local coordinates near $\Gamma$}\label{sec:loc coor}

  The  curve   $\Gamma$ is parametrized by the arc length  
   $$
   \gamma(s):=(\gamma_{1}(s), \gamma_{2}(s)),\qquad s   \in [0,\ell], 
   $$ 
   where $\ell:=|\Gamma|$. 
   The tangent vector and the inner normal at $\gamma(s)\in\Gamma$ are given respectively by 
   $$
   \tau(s):=(\dot{\gamma}_{1}(s   ), \dot{\gamma}_{2}(s   )), \qquad \nu(s   ):=(-\dot{\gamma}_{2}(s   ), \dot{\gamma}_{1}(s   )).
   $$ 
   Note that the orientation of $\Gamma$ is compatible with the choice of the unit normal in  the Definition~\ref{def: segregate}.
   For  $\delta>0$  small enough,   let 
$$\mathcal U_{\delta}:=\left\{x\in \Omega\ :\  \textrm{dist}(x,\Gamma)\le \delta\right\}$$ be a small  neighborhood of the curve $\Gamma$.  Taking $\delta$ smaller if needed  for any $x\in \mathcal U_{\delta}$ there exists the unique $(t , s)\in (-\delta,\delta)\times [0,\ell)$ such that 
\begin{equation}\label{coo}
 x=\gamma(s   )+ t\nu(s   )=(\gamma_{1}(s   )-t\dot\gamma_{2}(s), \gamma_{2}(s)+t\dot\gamma_{1}(s)).\end{equation}
By $S_\ell$ we denote the interval $[0,\ell]$ with the ends identified and by $\pC_{\delta, \ell}$ the truncated cylinder $(-\delta, \delta)\times S_\ell$. The map
\[
\begin{aligned}
X\colon \pC_{\delta, \ell}&\to \mathcal U_\delta,\\
(t,s)&\longmapsto x,
\end{aligned}
\]
is a diffeomorphism between $\mathcal U_\delta$ and $\pC_{\ell, \delta}$.  Notice that $\{X(0,s), s\in S_\ell\}=\Gamma$.
We also have 
$$\Omega_{1}\cap \mathcal U_{\delta}=\left\{x=X(t,s)\in \Omega\ :\  t> 0\right\}\quad  \hbox{and}\quad  
\Omega_{2}\cap \mathcal U_{\delta}=\left\{x=X(t,s)\in \Omega\ :\  t<0\right\}.$$
We will need later  the expression of the Laplace operator in local coordinates. If by $\kappa(s)$ we denote the curvature of $\Gamma$ and $\mathcal A(t,s)=(1-t\kappa(s))^2 $ then 
\begin{equation}\label{lap1}\Delta_{(t,s)}  =
   \partial_{tt} +\frac{1}{\A}\partial_{ss}+\frac{\partial_t \A}{2 \A}\partial_t  -\frac{\partial_s \A}{2\A^2}\partial_s.    
\end{equation}

{\color{black}{
We recall  (c.f (\ref{def: w0})) the initial outer  approximation of the solution 
\[
w^0_1(x)= w(x)\mathbbm{1}_{\Omega_{1}}(x), \qquad w^0_2(x)= -w(x)\mathbbm{1}_{\Omega_{2}}(x).
\]
For future use we will calculate some basic quantities associated with $w^0_i$ along $\Gamma$.  We gather the estimates needed later in the following:
\begin{lemma}\label{lem:est 1}
Let $x=X(t,s)$ be the local coordinates   described above. It holds
\begin{equation}
\label{est:on gb}
\begin{aligned}
\partial_t w^0_1\circ X^{-1}(0,s)&=\omega\circ X^{-1}(0,s), \qquad 
\partial_t w^0_2\circ X^{-1}(0,s)=-\omega\circ X^{-1}(0,s),
\\
\partial_s w^0_i\circ X^{-1}(0,s)&=0, \qquad \partial_{ss} w^0_i\circ X^{-1}(0,s)=0,
\\
\partial_{tt} w_i^0\circ X^{-1}(0,s)&=\kappa(s)\partial_t w_i^0\circ X^{-1}(0,s).
\end{aligned}
\end{equation}
\end{lemma}
\begin{proof}
We only proof the last assertion.  Using (\ref{lap1}) and the equation satisfied by $w$ we get along $\Gamma=\{t=0\}$ 
\[
\partial_{tt} w(0,s)-\kappa(s)\partial_t w(0,s)=0.
\]
We conclude invoking the definition of $w^0_i$. 
\end{proof}

In what follows for a function $\varphi$ defined in a small neighborhood of  $\Gamma$ we will write 
\[
\varphi(s,t)=\varphi\circ X^{-1}(s,t),
\]
and for a function $\varphi$ defined on $\Gamma$ we will write
\[
\varphi(s)=\varphi\circ X^{-1}(s,0),
\]
as long as it does not cause confusion.

}}


  \subsection{The ansatz}\label{sec ansatz}
In the rest of this paper  we will  denote $\eps=\beta^{-1/4}$ and consider the limit $\eps\to 0$ instead of  $\beta\to \infty$. 
In this section we will introduce a preliminary candidate for the solution \eqref{s} which consists of an inner and outer approximation and is of the form 
 
\begin{equation}\label{ans0}
  U_i(x)=\chi (x) \(\eps v^0_i(x)+\eps^2 v^1_i(x)\) +\chi_i(x)\(w^0_i (x)+\eps w^1_i(x)+\eps^2w^2_i(x)+{\color{black}{\sum_{j=3}^5 \eps^{j}w^j_i(x)}}\). 
  \end{equation}
  Here $\chi,\chi_1$ and $\chi_ 2$ are   cut-off functions such that
 \begin{equation}\label{cut}\begin{aligned}&{\rm supp}\ \chi (x)\subset\{|t|\le 2\eta\},\ {\rm supp}\ \chi_1 (x)\subset\{t >\eta\},\  {\rm supp}\ \chi_2 (x)\subset\{t <-\eta\},\\ & \chi= 1\ \hbox{if}\ |t|\le\eta\ \hbox{and}\ \chi+\chi_1+\chi_2=1,\end{aligned} \end{equation}
 where 
 \begin{equation}\label{eta}
 \eta= K\eps|\ln\eps|, \quad \eps\to 0,
 \end{equation}
 and $K>0$ is a constant to  be chosen later (see section \ref{subsect decomp}).
In particular
$$\chi_1(x)=0 \quad  \hbox{if}\quad  x\in\Omega_{2}\qquad  \hbox{and}\qquad  \chi_2(x)=0 \quad  \hbox{if}\quad x\in\Omega_{1}.$$
The functions $\bs v^j=(v^j_1,v^j_2)$, $j=0,1$   expressed in local variables \eqref{coo} of $\Gamma$   have form
\begin{equation}\label{inner}
\begin{aligned}
\bs v^0(t,s)&= b\mathbf V(\eps^{-1}b(t-\eps \zeta))\\
\bs v^1(t,s) &=\kappa\mathbf W(\eps^{-1}b(t-\eps \zeta))  
\end{aligned}
\end{equation}
As for the   functions $b$ and $\zeta$ we will assume {\it a priori} the following  expansions {\color{black}{
\begin{equation}
\label{exp:bzeta}
\begin{aligned}
b(s)&=b_0(s)+\eps b_1(s)+\eps^2 b_2(s), \\
{\color{black}{\zeta(s)}}&=\zeta_1(s)+\eps \zeta_2(s).
\end{aligned}
\end{equation}
Explicitly  the function $b_0$ is given in (\ref{b}) and is assumed to be of class $C^{3,\alpha}(0,|\Gamma|)$. 
For the purpose of  formal calculations below the functions $b_k, \zeta_k$, $k=1,2$ are assumed {\it a priori}  to be $C^{3,\alpha}(0,|\Gamma|)$ as well. The functions $w^j_i$, $j=1,\dots, 5$, $i=1,2$ will be determined below.}} 
\subsection{The matching error along  the interface}\label{sub-inter}
Expanding $ w_i^0$ as Taylor polynomial of order $3$ in $t$ and denoting its remainder by $T_3 w^0_i$, $i=1,2$, from Lemma \ref{lem:est 1} we get 
  \begin{equation}\label{w}
 \begin{aligned} 
 w^0_1(t,s)&= \omega(s) \(t+\frac12\kappa(s) t^2\)+(T_3 w_1^0)(t,s) \quad  \hbox{if}\ t>0,\ \\ 
 w^0_2(t,s)&=- \omega(s) \(t+\frac12\kappa(s) t^2\)+(T_3 w_2^0)(t,s)\quad  \hbox{if}\ t<0.
 \end{aligned}
 \end{equation}
Denoting 
 \[
 \bs v(t,s)=\eps \bs v^0(t,s)+\eps^2\bs v^1(t,s)
 \]
 and using  \eqref{v}, \eqref{z1} and \eqref{z2}, we find that when $0\leq t$:
\begin{equation}
\label{v1-w1}
\begin{aligned}
v_1(t,s)&=\eps bV_1\( \eps^{-1}b(t-{\color{black}{\eps}}\zeta)\)+\eps^2\kappa W_1\(\eps^{-1}b(t-{\color{black}{\eps}}\zeta)\) \\
& =\eps b\(A\eps^{-1}b(t-{\color{black}{\eps}}\zeta)+B\)+\frac{A\kappa}{2}b^2(t-{\color{black}{\eps}}\zeta)^2+\mathcal O_{C^{2,\alpha}}\( \eps e^{-ct^2\eps^{-2}}\)\\
&= Ab^2  \( t+\frac 12\kappa  t^2\)+\eps B b -{\color{black}{\eps}}A b^2 \zeta-{\color{black}{\eps}}A\kappa b^2\zeta  t +{\color{black}{\eps^2}}\frac{A}{2}\kappa b^2\zeta^2 +\mathcal  O_{C^{2,\alpha}}\( \eps e^{-c t^2\eps^{-2}}\).
\end{aligned}
\end{equation}
Similarly when $t\leq 0$  changing only $A$ to $-A$
\begin{equation}
\label{v2-w2}
\begin{aligned}
v_2(t,s)&= - Ab^2  \( t+\frac 12\kappa  t^2\)+\eps B b +{\color{black}{\eps}}A b^2 \zeta+{\color{black}{\eps}}A\kappa b^2\zeta  t -{\color{black}{\eps^2}}\frac{A}{2}\kappa b^2\zeta^2 +\mathcal  O_{C^{2,\alpha}}\( \eps e^{-c t^2\eps^{-2}}\).
\end{aligned}
\end{equation}
Moreover, by \eqref{v}, \eqref{z1} and \eqref{z2} we also deduce
\begin{equation}\label{v2-v1}\begin{aligned}& v_2(t,s)= \mathcal O_{C^{2,\alpha}}\( \eps e^{-ct^2\eps^{-2}}\),\qquad  \hbox{if}\ 0\geq t,\\
&v_1(t,s)=\mathcal O_{C^{2,\alpha}}\( \eps e^{-c t^2\eps^{-2}}\) ,\qquad \hbox{if}\quad t\leq 0.
\end{aligned}\end{equation}
To indicate that the expansions (\ref{v1-w1})---(\ref{v2-v1}) can be differentiated and are valid  in the $C^{2,\alpha}$ sense we have used the symbol $\mathcal O_{C^{2,\alpha}}(\cdot)$. 

From (\ref{exp:bzeta}),  (\ref{w}),   (\ref{v1-w1}) and  (\ref{v2-w2})  it follows  
\begin{equation}
\label{est:match 1}
\begin{aligned}
v_1-w^0_1&=  \eps\(2A b_1b_0 t+B b_0-Ab_0^2 \zeta_1-A\kappa b_0^2\zeta_1 t\) \\
&\quad+\eps^2\left[A\(b_1+2 b_2b_0\) t+Bb_1-A\(2 b_1b_0\zeta_1+b_0^2\zeta_2\)
-A\(\kappa b_0b_1\zeta_1+\kappa b_0^2\zeta_2\) t+\frac{A}{2}\kappa b_0^2\zeta^2_1\right] 
\\&\qquad +\mathcal O_{C^{2,\alpha}}\( \eps e^{-c t^2\eps^{-2}}\)-(T_3 w_1^0)(t,s)
\\
v_2-w^0_2&=\eps\(-2A b_1b_0 t+B b_0+Ab_0^2 \zeta_1+A\kappa b_0^2\zeta_1 t\) \\
&\quad+\eps^2\left[-A\(b_1+2 b_2b_0\) t+Bb_1+A\(2 b_1b_0\zeta_1+b_0^2\zeta_2\)
+A\(\kappa b_0b_1\zeta_1+\kappa b_0^2\zeta_2\) t-\frac{A}{2}\kappa b_0^2\zeta^2_1\right] \\
&\qquad 
+\mathcal O_{C^{2,\alpha}}\( \eps e^{-c t^2\eps^{-2}}\)-(T_3 w_2^0)(t,s)
\end{aligned}
\end{equation}
respectively for $0\leq t$ and $0\leq t$.
In the next section we will find a correction to the outer expansion of the form $\bs w^0+\eps\bs w^1+\eps^2 \bs w^2$  so that the  terms of order $\mathcal O( \eps)$ and $\mathcal O(\eps^2)$ in (\ref{est:match 1})  are removed.

\subsection{The outer approximation}

\subsubsection{The matching condition along $\Gamma$}\label{sec:match}

%

We recall that  by $\nu$ we denoted to unit normal of  $\Gamma$ outer to $\Omega_2$. For the purpose of this section  we will denote by $n_i$, $i=1,2$  the unit normals of $\Gamma$ outer to $\Omega_i$. Thus along $\Gamma$ we have
\[
\partial_{n_1}=-\partial_\nu=-\partial_t, \qquad \partial_{n_2}=\partial_\nu=\partial_t.
\]
Formally, the matching conditions  at order $\mathcal O(\eps)$  will be satisfied if we require that 
\begin{equation}
\label{cond:match e}
\begin{aligned}
w_1^1(t,s)&=  A\left(2 b_1 b_0- \kappa b_0^2\zeta_1\right) t+\left(B b_0-Ab^2_0\zeta_1\right) +(T_2 w_1^1)(t,s),\\
w_2^1(t,s)&=-A\left(2 b_1 b_0- \kappa b_0^2\zeta_1\right) t+\left(B b_0+Ab^2_0\zeta_1\right)+(T_2 w_2^1)(t,s).
\end{aligned}
\end{equation}
At order $\mathcal O\(\eps^2\)$ we impose
\begin{equation}
\label{cond:match e2}
\begin{aligned}
w_1^2(t,s)&=  Bb_1-A(2b_0b_1\zeta_1+b_0^2 \zeta_2-\frac{\kappa}{2}b_0^2\zeta_1)+A\(2b_2b_0+b_1^2-\kappa b_0^2\zeta_2-2\kappa b_0b_1\zeta_1\)t  +(T_2 w_1^2)(t,s),\\
w_2^2(t,s)&=Bb_1+A(2b_0b_1\zeta_1+b_0^2 \zeta_2-\frac{\kappa}{2}b_0^2\zeta_1)-A\(2b_2b_0+b_1^2-\kappa b_0^2\zeta_2-2\kappa b_0b_1\zeta_1\)t  +(T_2 w_2^2)(t,s).\end{aligned}
\end{equation}
This implies that on $\Gamma$  the following matching conditions at order $\mathcal O(\eps)$ should hold:
\begin{equation}
\label{conds:w1}
\begin{aligned}
w_1^1(0,s)&=Bb_0-Ab_0^2\zeta_1 , \qquad 
\partial_{n_1} w^1_1(0,s)=-2Ab_0b_1 +A\kappa b_0^2\zeta_1,\\
w_2^1(0,s)&=Bb_0+Ab_0^2\zeta_1,\qquad 
\partial_{n_2} w^1_2(0,s)=-2Ab_0b_1 +A\kappa b_0^2\zeta_1.
\end{aligned}
\end{equation}
Likewise at order $\mathcal O\(\eps^2\)$ we should require
\begin{equation}
\label{conds:w2}
\begin{aligned}
w_1^2(0,s)&=Bb_1-A(2b_0b_1\zeta_1+b_0^2 \zeta_2-\frac{\kappa}{2}b_0^2\zeta_1), \quad 
\partial_{n_1} w^2_1(0,s)= -A\(2b_2b_0+b_1^2-\kappa b_0^2\zeta_2-2\kappa b_0b_1\zeta_1\),\\
w_2^2(0,s)&=Bb_1 +A(2b_0b_1\zeta_1+b_0^2 \zeta_2-\frac{\kappa}{2}b_0^2\zeta_1),\quad 
\partial_{n_2} w^2_2(0,s)= -A\(2b_2b_0+b_1^2-\kappa b_0^2\zeta_2-2\kappa b_0b_1\zeta_1\).
\end{aligned}
\end{equation}

Note that conditions (\ref{conds:w1}) and (\ref{conds:w2}) share the general structure with respect to the unknowns $(b_1, \zeta_1)$ and $(b_2, \zeta_2)$ appearing on the right hand sides. Thus it is convenient to take a more general point of view in which both problems can be treated in a common framework.  
To this end suppose that we are given functions {\color{black}{$g_i\in H^{s+1/2}(\Gamma)$, $h_i\in H^{s}(\Gamma)$ with $s\geq 1$ to be specified later}}.
To solve the matching problem (\ref{conds:w1}) or  (\ref{conds:w2}) we consider the problem of finding  functions $k_1$ and $k_2$ in such a away that  the following linear problems can be solved for $\bs\psi=(\psi_1,\psi_2)$:
 
  \begin{equation}
 \label{plin-cru}
 \begin{aligned}
&\left\{ \begin{aligned}
 &-\Delta \psi_1-f_{u}(w^0_1,x)\psi_1= 0\ \hbox{in}\  \Omega_1,\\
 &\psi_1=0\ \hbox{on}\ \partial\Omega\cap \partial\Omega_{1}, \\
 &\psi_1=  -Ab_0^2k_1+g_1\ \hbox{on}\ \Gamma,\\
 &\partial_{n_1} \psi_1= -2A b_0  k_2+A\kappa b_0^2 k_1+h_1 \ \hbox{on}\ \Gamma,\\
 \end{aligned}\right.\ 
 &\left\{ \begin{aligned}
 &-\Delta \psi_2-f_{u}(w^0_2,x)\psi_2= 0\ \hbox{in}\ \Omega_2,\\
 &\psi_2=0\ \hbox{on}\ \partial\Omega\cap \partial\Omega_{2},  \\
 &\psi_2=Ab_0^2k_1+g_2\ \hbox{on}\ \Gamma,\\
 &\partial_{n_2} \psi_2=-2A b_0  k_2+A\kappa b_0^2 k_1+h_2 \ \hbox{on}\ \Gamma.\\
 \end{aligned}\right.\end{aligned}
 \end{equation}
This problem can be stated in terms of the Dirichlet-to-Neumann maps $\bs D_i$ and the map $\bs D=\bs D_1+\bs D_2$ defined in Section \ref{sec dtn} as the following system of equations for the unknowns $k_1, k_2$:
\[
\begin{aligned}
\bs D_1\(-Ab_0^2k_1+g_1\)&=-2A b_0  k_2+A\kappa b_0^2 k_1+h_1 \\
\bs D_2\(Ab_0^2k_1+g_2\)&=-2A b_0  k_2+A\kappa b_0^2 k_1+h_2
\end{aligned}
\]
which, after adding and subtracting becomes
\begin{equation}
\label{sys:dtn}
\begin{aligned}
-2A b_0  k_2+ \frac{1}{2}\left[\bs D_1\(Ab_0^2k_1\)- \bs D_2\(Ab_0^2k_1\)\right]+Ab_0^2\kappa k_1&=\frac{1}{2}\left[-\(h_1+h_2\)+\bs D_1(g_1)+\bs D_2(g_2)\right]\\
\bs D_1(Ab_0^2k_1)+\bs D_2(Ab_0^2k_1)&=\bs D_1(g_1)-\bs D_2(g_2)+h_2-h_1
\end{aligned}
\end{equation}
The second equation is of the form
\[
\bs D(Ab_0^2k_1)=\bs D_1(g_1)-\bs D_2(g_2)+h_2-h_1
\]
and can be solved since, by (\ref{cond on D}), the operator $\bs D$ has bounded inverse in $L^2(\Gamma)$. Once $k_1$ is determined, we can determine $k_2$ from the first relation in (\ref{sys:dtn}) using the fact that $b_0(s)\neq 0$, $s\in [0, |\Gamma|]$. 
{\color{black}{Summarizing we have:
\begin{lemma}\label{lem:dtn}
Given functions {\color{black}{$g_i\in H^{s+1/2}(\Gamma)$, $h_i\in H^{s}(\Gamma)$ with $s\geq 1$}} the problem (\ref{plin-cru}) has a unique solution $\psi_1\in H^{s+1}(\Omega_1)$, $\psi_2\in H^{s+1}(\Omega_2)$,  $k_1\in H^{s+1/2}(\Gamma)$, $k_2\in H^{s}(\Gamma)$.
 \end{lemma}
 \begin{proof}
 The proof follows directly from the properties of the Dirichlet-to-Neumann maps $\bs D_i$ described in section \ref{sec dtn}.  For instance under the hypothesis of the Lemma the right hand side of the second equation in (\ref{sys:dtn}) belongs to $H^s(\Gamma)$ hence $\bs D(A b_0^2 k_1)\in H^s(\Gamma)$ and then $A b_0^2 k_1\in H^{s+1/2}(\Gamma)$. Since $b_0>0$ is a smooth function we get $k_1\in H^{s+1/2}(\Gamma)$. We proceed similarly with the first equation in (\ref{sys:dtn}) to show $k_2\in H^{s}(\Gamma)$. 
 \end{proof}

}}
 \subsubsection{Improvement of the outer approximation} \label{sec imp outer}
 We will suppose that $b_0\in H^{s+1/2}(\Gamma)$, $s>1$. 
 To improve the outer approximation at order $\mathcal O(\eps)$ we need to solve
\begin{equation}
\label{sys:w12}
 \begin{aligned}
 -\Delta w^1_i-f_{u}(w^0_i,x)w^1_i &=0\quad \hbox{in}\quad \Omega_{i}\\
 w^1_i&=0\quad  \hbox{on}\quad \partial\Omega\cap\partial\Omega_i,\\
 \end{aligned}
 \end{equation}
 supposing additionally that along $\Gamma$ the matching conditions (\ref{conds:w1}) hold. In this case we have $g_i= Bb_0\in H^{s+1/2}(\Gamma)$,  $h_i\equiv 0$ and 
\begin{equation}
\label{con:meps}
\begin{aligned}
\bs D(Ab_0^2\zeta _1)&=\bs D_1(Bb_0)-\bs D_2(B b_0)\\
-2A b_0  b_1&=-\frac{1}{2}\left[\bs D_1\(Ab_0^2\zeta_1\)- \bs D_2\(Ab_0^2\zeta_1\)\right]-Ab_0^2\kappa \zeta_1+\frac{1}{2}\left[\bs D_1(B b_0)+\bs D_2(B b_0)\right]
\end{aligned}
\end{equation}
It follows by Lemma \ref{lem:dtn}   that there exists a unique solution of the above system such that $(\zeta_1, b_1)\in H^{s+1/2}(\Gamma)\times H^{s}(\Gamma)$ and $w^1_i\in H^{s+1}(\Omega_i)$, $i=1,2$.

Next, to find the improvement of the outer approximation at order $\mathcal O(\eps^2)$ we need to solve again the system (\ref{sys:w12}) with $w^1_i$ replaced by $w^2_i$ together with the boundary conditions (\ref{conds:w2}). However we should also take into account the fact that in the formal expansion of the outer problem terms of the order $\mathcal O(\eps^2)$ depending on $w_i^1$ will appear in the right hand side. Thus we look for $w_i^2$ in the form $w_i^2=\tilde w_i^2+\bar w^2_i$ where
\[
\begin{aligned}
-\Delta \tilde w_i^2-f_{u}(w_i^0,x)\tilde w_i^2&= \frac{1}{2} f_{uu}(w_i^0,x) \(w_i^1\)^2 \quad \hbox{in}\quad \Omega_{i}\\\
\tilde w^2_i&=0\quad \hbox{on}\quad \partial\Omega\cap\partial\Omega_i,\\
\tilde w^2_i&=0, \quad \hbox{on}\quad \Gamma,
 \end{aligned}
\]
and $\bar w_i^2$ solves  (\ref{sys:dtn}) with $\psi_i=\bar w^2_i$,  $k_1=\zeta_2$, $k_2=b_2$ and 
\[
\begin{aligned}
g_1=Bb_1+2Ab_0b_1\zeta_1, &\qquad g_2=Bb_1-2Ab_0b_1\zeta_1\\
h_1=-Ab_1+Ab_0b_1\zeta_1-\partial_{n_1}\tilde w^2_1, &\qquad h_2=-Ab_1+Ab_0b_1\zeta_1-\partial_{n_2}\tilde w^2_2.
\end{aligned}
\]
Similarly as above, by Lemma \ref{lem:dtn},  we obtain the unique solution such that $(\zeta_2, b_2)\in H^s(\Gamma)\times  H^{s-1/2}(\Gamma)$.  We summarize our discussion in the following:
\begin{corollary}\label{cor match}
Let $s>1$ be such that $b_0\in H^{s+1/2}(\Gamma)$. There exist functions $\bs w^j$,   $\zeta_j$, $b_j$,  satisfying the matching conditions (\ref{conds:w1}) and (\ref{conds:w2}), respectively for $j=1$ and $j=2$ such that $w^1_i\in H^{s+1}(\Omega_i)$, 
$(\zeta_1, b_1)\in H^{s+1/2}(\Gamma)\times H^s(\Gamma)$ and $w^2_i\in H^{s+1/2}(\Omega_i)$ $(\zeta_2, b_2)\in H^{s}(\Gamma)\times H^{s-1/2}(\Gamma)$.
\end{corollary}
We see that the  approximation at order $\mathcal O(\eps^2)$ entails loss of regularity with respect to the approximation at order $\mathcal O(\eps)$. For this reason we will  assume {\it a priori} that $b_0$ is as smooth as we need or in other words  we can take $s>1$ in the Corollary \ref{cor match} large enough to ensure that all the calculations below are justified. Since we need $b_j, \zeta_j\in C^{3,\alpha}(\Gamma)$, $j=1,2$ and the lower regularity {\it a priori} is that of $b_2\in H^{s-1/2}(\Gamma)$  standard embeddings imply that if $s> 4$ the required regularity will hold with some $\alpha\in (0,1)$. Additionally this implies  that $\bs w^j_i\in C^{3,\alpha}(\Omega_i)$.

{\color{black}{
\subsection{Further refinement of the outer expansion}

In this section we will further refine the outer expansion in $\Omega_i$.  At this moment we will not need to satisfy the matching conditions. The outer approximation we have calculated so far is of the form
\[
\bs \psi=(\psi_1\mathbbm{1}_{\Omega_1}, \psi_2\mathbbm{1}_{\Omega_2}), \qquad \psi_i=w^0_i+\eps w_i^1+\eps^2 w_i^2.
\]
The error of this approximation is
\[
\bs \Theta=(\Theta_1\mathbbm{1}_{\Omega_1}, \Theta_2\mathbbm{1}_{\Omega_2}), \qquad \Theta_i=-\Delta\psi_i-f(\psi_i, x), \qquad \mbox{in}\quad \Omega_i.
\]
Assuming enough regularity for $f(u, x)$ we can expand the error in powers of $\eps$:
\[
\Theta_i=\eps^3\Theta_i^3+\eps^4\Theta_i^4+\eps^5 \Theta_i^5+\mathcal O(\eps^6).
\]
By the non-degeneracy condition {\color{black}{(b)}}  in Definition \ref{def nondegener} we can solve for $j=3, 4, 5$ the following problems
\[
  \begin{aligned}
    -\Delta w^j-\left[f_{u}(w^0_1,x) \mathbbm{1}_{\Omega_{1}}+f_{u}(w^0_2,x) \mathbbm{1}_{\Omega_{2}} \right] w^j &=-\Theta_1^j\mathbbm{1}_{\Omega_1}-\Theta_2^j\mathbbm{1}_{\Omega_2},\quad  \hbox{in}\ \Omega,\\
    w^j&=0 \quad \hbox{on}\quad \partial\Omega,
        \end{aligned} 
 \]
We set 
\[
w_i^j=w^j\mathbbm{1}_{\Omega_i}, \quad j=3,4,5,
\]
and define the 
 the ansatz far away from the curve by setting
\begin{equation}\label{ans2}
\bw=(\mathbbm{1}_{\Omega_1}\w_1,\mathbbm{1}_{\Omega_2}\w_2), \qquad 
\w_i:=\sum_{j=0}^5 \eps^j w^j_i \quad \hbox{in}\quad \Omega_{i}.
\end{equation}
For convenience we set $\w_i\equiv 0$ in $\Omega_{3-i}$, $i=1,2$ whenever we need to consider  these functions in the whole $\Omega$. 
}}

{\color{black}{
   \section{Proof of Theorem \ref{thm:main}}\label{sec: thm main}
   \setcounter{equation}{0} 
\subsection{Decomposition into the inner and the outer problem} \label{subsect decomp}
 We define  smooth  cutoff functions $\chi ,\tilde \chi,\hat \chi$ and $\chi _i, \tilde \chi_i, \hat\chi_i$, $i=1,2$ as follows: let $\varphi$ be a smooth, even, cutoff function such that $\mathbbm{1}_{[-1,1]}\leq \varphi\leq \mathbbm{1}_{[-2,2]}$ and $\varphi'(x)\leq 0$, $x\geq 0$. With $\eta=K\eps|\ln \eps|\sim \beta^{-1/4}\ln\beta$ and $K$, $\wt m$, $\hat m$, $\bar m$ to be fixed later on we set
  \begin{equation}\label{cut1}
  \begin{aligned}
  &\chi(t)=\varphi(t/\eta),\qquad \hat\chi(t)=\varphi(\hat m t/\eta )\qquad \wt\chi(t)=\varphi(\wt mt/\eta ),\\
&  \chi_1(t)=(1-\chi(t))\mathbbm{1}_{(0,\infty)}, \qquad \bar \chi_1(t)=\(1-\varphi(\bar m t/\eta)\)\mathbbm{1}_{(0,\infty)},  \\
& \chi_2(t)=(1-\chi(t))\mathbbm{1}_{(-\infty, 0)}, \qquad  \bar\chi_2(t)=\left(1-\varphi(\bar m t/\eta)\right)\mathbbm{1}_{(-\infty, 0)}.  
  \end{aligned} \end{equation}
{\color{black}{We set $\bar \chi=\bar\chi_1+\bar\chi_2$. }} Furthermore, we chose $\wt m<1$, $\wt m<\hat m< 1$ and $\bar m>1$ such that 
   \begin{equation}
 \label{rr}
 \begin{aligned}
&\wt \chi \equiv 1\qquad \hbox{in}\quad  \supp \nabla \chi\quad \mbox{and}\quad \supp\nabla\bar  \chi_i,
\qquad \chi\equiv 1\quad \hbox{in}\quad  \supp \nabla\bar \chi_i,\\
& \wt\chi\chi=\chi,  \qquad \bar\chi_i\chi_i=\chi_i, \qquad \wt\chi\hat\chi=\hat\chi, \qquad (1-\hat\chi) \bar\chi=(1-\hat\chi)
\end{aligned}
\end{equation}

{\color{black}{

Recall  the definition of  inner approximate solution $\bs v(s,t)$ in (\ref{inner}).
This function depends  on the functions  $b=b(s)$ and $\zeta=\zeta(s)$ defined in (\ref{exp:bzeta}) and determined explicitly up to order $\mathcal O(\eps^2)$ in section \ref{sec imp outer} through the matching conditions. 
We introduce new stretched  variable  $\tau$ as follows
\begin{equation}
\label{var: stretched}
\tau=\eps^{-1} b(s)\left(t-\eps\zeta(s)\right),
\end{equation}
and we will refer to $(\tau, s)$ as the inner variables. 
We also modify the initial ansatz by adding two  new modulations functions $\mu=\mu(\tau, s)$ and $\xi=\xi(\tau, s)$ which are {\it a priori} assumed to be of order $o(\eps)$. 
We will denote
\[
\begin{aligned}
{\bs{\mv}}^0(\tau,s)&= b(s)\left(1+\mu(\tau, s)\right)\mathbf V\left(\left(1+\mu(\tau, s)\right)\tau+\xi(\tau, s)\right),\\
{\bs{\mv}}^1(\tau,s)&=\kappa(s)\mathbf W\left(\left(1+\mu(\tau, s)\right)\tau+\xi(\tau,s)\right),
\end{aligned}
\]
so that 
\begin{equation}
\label{ans inner}
{\bs{\mv}}(\tau,s)=\eps {\bs{\mv}}^0(\tau, s)  +\eps^2{\bs{\mv}}^1(\tau, s). 
\end{equation}
With the outer ansatz $\bw$ defined in (\ref{ans2}) we  set 
\begin{equation}
\label{def U_final}
\bs {\mathrm U}=\chi \bs {\mv}+(\chi_1 \w_1,\chi_2\w_2).
\end{equation}
This is the new initial approximation to our problem. 
Additionally we  introduce the  additive inner perturbation $\wt {\bs v}=(\wt v_1, \wt v_2)\in C^{2,\alpha}_\theta(\pC)$  and the outer perturbation  $\bar{\bs w}=(\bar  w_1, \bar w_2)\in C^{2,\alpha}(\Omega)$. 
We look for a solution of (\ref{model:2}) in the form
\begin{equation}
\label{def:sol u}
\bs u=\bs {\mathrm U}+\wt \chi\wt{\bs v}+\bar\chi \bar{\bs w}.
\end{equation}

Denoting 
\[
\bs F(\bs u,x)=\left(\begin{array}{c} f(u_1,x)-\eps^{-4} u_1u_2^2\\
f(u_2,x)-\eps^{-4} u_2u_1^2
\end{array}\right),
\]
our original problem (\ref{s}) becomes
\begin{equation}
\label{model:2}
\begin{aligned}
-\Delta  \bs u &=  \bs F(\bs u, x),  \qquad \mbox{in}\quad \Omega,\\
\bs u&=0, \qquad \mbox{on}\quad \partial \Omega.
\end{aligned}
\end{equation} 
To set up a fixed point argument it  is convenient to write (\ref{model:2}) as a linearized system. To do this we introduce two potentials
\begin{equation}
\label{def:qq}
-D_u \bs F(\bs U,x):=\bs q=\left(\begin{array}{cc}
-f_{u}(\mU_1,x){+}\eps^{-4} \mU_2^2 &2\eps^{-4} \mU_1\mU_2\\
2\eps^{-4} \mU_1\mU_2 & -f_{u}(\mU_2,x){+}\eps^{-4} \mU_1^2
\end{array}\right):=\left(\begin{array}{c} \bs q_1\\ \bs q_2\end{array}\right)
\end{equation}
and
\begin{equation}
\label{def:qzero}
\bs q^0=\left(\begin{array}{cc}
\eps^{-2} (\mv^0_2)^2 &2\eps^{-2} \mv^0_1\mv^0_2\\
2\eps^{-2} \mv^0_1\mv^0_2 & \eps^{-2} (\mv^0_1)^2
\end{array}\right):=\left(\begin{array}{c} \bs q^0_1\\ \bs q^0_2\end{array}\right)
\end{equation}
and denote
\[
\begin{aligned}
\bs E(\bs {\mU},x)&=\Delta \bs{\mU}+\bs{F}(\bs {\mU},x), \\
\bs N(\bs u, x)&= \bs F(\bs u,x)-\bs F(\bs {\mU},x)-D_u\bs F(\bs U,x)\cdot (\bs u-\bs{\mU}).
\end{aligned}
\]
We need two more cutoff functions: $\wt \rho$ compactly supported and such that $\wt\rho\wt\chi=\wt\chi$ and $\bar \rho_i$ supported in the set $|t|\geq \frac{\eta}{2\bar m}$ and satisfying $\bar\rho_i\bar\chi_i=\bar\chi_i$.  To achieve this we may take
\[
\wt\rho(t)=\chi\left(\frac{\wt \ell t}{\eta}\right), \quad \wt \ell <\frac{\wt m}{2}, \qquad \bar\rho_i(t)=\chi_i\left(\frac{\bar \ell t}{\eta}\right), \quad \bar\ell>{2\bar m}.
\]
Set $\bar \rho=\bar\rho_1+\bar\rho_2$ and  $\delta_i=\mv_i-\w_i$. With these notations  we decompose (\ref{model:2}):
\begin{equation}
\label{sys:O3}
\begin{aligned}
-\Delta \wt v_i+\bs q_i^0\cdot \wt{\bs v}&=- \wt\rho(\bs q_i-\bs q_i^0)\cdot  \wt{\bs v}+ \chi \(\Delta \mv_i+F_i(\bs {\mU},x)+f_u(\w_i,x)(1-\chi)\delta_i\)\\
&\qquad +\hat\chi N_i+\commutator{\Delta}{\chi}\delta_i +\commutator{\Delta}{\bar\chi}\bar w_i, \quad i=1,2, \\
-\Delta \bar w_i+ \bar\rho\bs q_i \cdot \bar {\bs w}&=(1-\chi)\(\Delta \w_i+F_i(\bs{\mU},x)-f_u(\w_i,x)\chi\delta_i\)+(1-\hat\chi)N_i+\commutator{\Delta}{\wt\chi}\wt v_i,\\
& \qquad \mbox{in}\quad\Omega, \quad i=1,2. 
\end{aligned}
\end{equation}
Multiply the first two equations by $\wt\chi$ and use the relevant relations in  (\ref{rr}) to get 
\[
\wt\chi(-\Delta \wt v_i+\bs q_i\cdot \wt{\bs v})= \chi \(\Delta \mv_i+F_i(\bs{\mU},x)+f_u(\w_i,x)(1-\chi)\delta_i\)+\hat\chi N_i+\commutator{\Delta}{\chi}\delta_i+\left[\Delta, \bar\chi_i\right]\bar w_i
\]
and multiply the remaining  equations by $\bar \chi$ each to get
\[
\bar\chi\(-\Delta \bar w_i+\bs q_i\cdot \bar {\bs w}\)=(1-\chi)\(\Delta \w_i+F_i(\bs{\mU},x)-f_u(\w_i,x)\chi\delta_i\)+(1-\hat\chi) N_i+\commutator{\Delta}{\wt\chi}\wt v_i. 
\]
Adding   the corresponding inner and outer equations we obtain (\ref{model:2}).

The first equation in (\ref{sys:O3}) is the inner problem and the second is the outer problem. Note that at this stage we do not require matching conditions betwen the inner and the outer part of the solution. This is due to the fact that  their respective commutators are small enough thanks to the $\mathcal O(\eps^3)$ approximation. However we need to guarantee that the inner unknown $\wt{\bs v}$ is exponentially decaying and for this reason we introduced the modulation functions  $\mu$ and $\xi$ in the ansatz.  
}}

\subsection{Norms of the perturbations}

{\color{black}{
In general we will measure the size of the perturbations $\bar{\bs w}$ and $\wt{\bs v}$, respectively,   in  H\"older and in weighted  H\"older  norms. Therefore we will suppose {\it a priori} that 
\begin{equation}
\label{hyp:size w}
\|\bar w_i\|_{C^{1,\alpha}(\Omega_i)}+\|\bar w_i\|_{C^{1,\alpha}(\Omega_{3-i})}\lesssim \eps^{4+\bar\varsigma}, \qquad i=1,2.
\end{equation}
The inner problem in (\ref{sys:O3}) is more conveniently stated in terms of the stretched variables $(\tau, s)\in \pC$. We will use the weighted H\"older norms for the perturbations and suppose {\it a priori}
\begin{equation}
\label{hyp:size v}
\|\wt{\bs v}\|_{C^{2,\alpha}_{\theta}(\pC)}\lesssim \eps^{\color{black}{2+\wt\varsigma}}.
\end{equation}
The constant $\theta$ appearing in the exponential weight will be adjusted {\it a posteriori}. 

We will now discus the modulation functions. Consistently with the assertions of  Lemma \ref{prop: inner lineal} we suppose
\begin{equation}
\label{hyp:ort mod}
(\operatorname{Id}-P_{\bar K_\eps})\mu\equiv 0, \qquad (\operatorname{Id}-P_{\bar K_\eps})\xi\equiv 0,
\end{equation}
and more explicitly we will assume the following. \blu{We set
\begin{equation}
\label{hyp:expl mod}
\mu(\tau, s)=\left(\sum_{{\color{black}{-\bar K_\eps}}}^{\bar K_\eps} \hat \mu_{k}\psi_{k}(s)\right)\phi_{\mu}(\tau), \qquad 
\xi(\tau, s)=\left(\sum_{{\color{black}{-\bar K_\eps}}}^{\bar K_\eps} \hat \xi_{k}\psi_{k}(s)\right)\phi_{\xi}(\tau),
\end{equation}
where $\hat\mu_{k}$, $\hat \xi_{k}$ are unknowns to be determined. At this point we do not need the precise form of the weight functions $\phi_\mu$ and $\phi_\xi$ (they will be introduced later on) but we suppose {\it a priori} that they are even functions such that 
\begin{equation}
\label{hyp: phi}
|\phi_\mu(\tau)|\lesssim \sech(\hat \lambda \tau), \qquad |\phi_\xi(\tau)|\lesssim \sech(\hat \lambda \tau),
\end{equation}
where $\hat \lambda>0$ is a large constant. 
We will denote
\[
\hat \mu(s)=\sum_{{\color{black}{-\bar K_\eps}}}^{\bar K_\eps} \hat \mu_{k}\psi_{k}(s), \qquad \hat \xi(s)=\sum_{{\color{black}{-\bar K_\eps}}}^{\bar K_\eps} \hat \xi_{k}\psi_{k}(s)
\]
 and introduce the norms to be used for the modulation functions $\mu(\tau, s)$ and $\xi(\tau, s)$
 \[
 \|\mu\|_{C^{k,\alpha}(\Gamma)}=\|\hat \mu\|_{C^{k,\alpha}((0,|\Gamma|))}, \qquad  \|\xi\|_{C^{k,\alpha}(\Gamma)}=\|\hat \xi\|_{C^{k,\alpha}((0,|\Gamma|))}.
 \]
 We will suppose that {\it a priori}
\begin{equation}
\label{hyp:size mod}
\|\mu\|_{C^{2,\alpha}(\Gamma)}+\|\xi\|_{C^{2,\alpha}(\Gamma)}\lesssim \eps^{{2-\hat \varsigma}}.
\end{equation}
We have as well (we agree that $\theta<\hat \lambda$)
 \begin{equation}
\label{hyp:size mod 2}
\|\mu\chi\|_{C^{k,\alpha}_{\theta}(\pC)}+\|\xi\chi\|_{C^{k,\alpha}_{\theta}(\pC)}\lesssim  \left(\|\mu\|_{C^{k,\alpha}(\Gamma)}+\|\xi\|_{C^{k,\alpha}(\Gamma)}\right).
\end{equation}
Above $\bar\varsigma, \wt\varsigma, \hat\varsigma>0$ are constants to be determined later.

}

In what follows it is important to estimate some functions coming from  the outer equation  and appearing in the inner equation and vice versa. Note that in the nonlinear term $\chi N_i$ of the inner equation in (\ref{sys:O3}) the  outer perturbation 
$\bar {\bs w}$ appears multiplied by $\bar \chi$. Thus, consider a function $\bar\phi\in C^{1,\alpha}(\Omega)$. We express $\chi\bar \chi\bar \phi$ in terms of the stretched variables $(\tau, s)$ to get  
\begin{equation}
\label{transf:out-in}
\|\chi\bar \chi\bar \phi\|_{C^{1,\alpha}_\theta(\pC)}\lesssim e^{\,2\theta K|\ln\eps|}\|\bar\phi\|_{C^{1,\alpha}(\Omega)}.
\end{equation}  
We also have a commutator term, which can be estimated as follows
\begin{equation}
\label{transf:com 1}
\left\|\commutator{\Delta}{\bar \chi}\bar \phi\right\|_{C^{\alpha}_\theta(\pC)}\lesssim e^{2K\theta |\ln\eps|/\bar m}\left(\eps^{-2}\|\bar\phi\|_{C^\alpha(\Omega)}+\eps^{-1}\|\bar\phi\|_{C^{1,\alpha}(\Omega)}\right).
\end{equation}

Considering the outer equation in (\ref{sys:O3}) we note that the nonlinear function $(1-\chi) N_i$ contains "inner" functions of the form $\wt \chi \wt \phi$. We estimate them as follows
\begin{equation}
\label{transf:in-out}
\|(1-\chi)\wt \chi \wt \phi\|_{C^\alpha(\Omega)}\lesssim \eps^{-\alpha} e^{\,-K\theta|\ln\eps|} \|\wt \phi\|_{C^\alpha_\theta(\pC)}.
\end{equation}
The commutator term satisfies
\begin{equation}
\label{transf:com 2}
\left\|\commutator{\Delta}{\wt \chi}\wt \phi\right\|_{C^\alpha(\Omega)}\lesssim \eps^{-2-\alpha} e^{\,-K\theta|\ln\eps|/{\wt m}} \|\wt \phi\|_{C^{1,\alpha}_\theta(\pC)}.
\end{equation}

}}

\subsection{Expansion of the nonlinear term}\label{sec exp non}

{\color{black}{With the notation as above we set 
\[
\bs{\rho}=\bs{u}-\bs{\mU}.
\]
We will separate the components of the outer equation as follows
\[
\bar w_i=\bar w_{i1}+\bar w_{i2}, \qquad \bar w_{ij}=\bar w_i\mathbbm{1}_{\Omega_j}, \qquad i, j=1,2.
\]
Then 
\[
\rho_i=\wt\chi \wt v_i+\bar\chi_1\bar w_{i1}+\bar \chi_2\bar w_{i2}.
\]
%
%
We also write
\begin{equation}
\label{decomp nonlinear}
\bs N(\bs {\mU}+ \bs \rho,x)=\bs N_f(\bs {\mU}+ \bs \rho,x)-\eps^{-4} \bs M(\bs {\mU}+ \bs \rho), 
\end{equation}
where
\[
\begin{aligned}
\bs N_f(\bs {\mU}+ \bs \rho,x)&=\left(\begin{array}{c}f(U_1+\rho_1, x)-f(U_1, x)-f_u(U_1, x)\rho_1\\
f(U_2+\rho_2, x)-f(U_2, x)-f_u(U_2, x)\rho_2
\end{array}
\right),\medskip\\
\bs M(\bs {\mU}+ \bs \rho)&=\left(\begin{array}{c} U_1\rho^2_2+2 U_2\rho_1\rho_2+\rho_1\rho_2^2\\
U_2\rho_1^2+2 U_1\rho_1\rho_2+\rho_1^2\rho_2
\end{array}
\right).
\end{aligned}
\]
In what follows we will need to estimate nonlinear functions of the perturbations localized near $\Gamma$ and also supported in the outer region. These estimates are straightforward  for ${\bs N}_f$ but the second term $\bs M(\bs {\mU}+ \bs \rho)$ is more complicated. 

Expanding $f$ in  the Taylor polynomial we get for $i,j=1,2$: 
\begin{equation}
\label{est:nf}
\chi_j |\bs N_{f,i}(U_i+\rho_i,x)|\lesssim \chi_j|\rho_i|^2\lesssim \chi_j (\wt \chi^2 |\wt{\bs v}|^2+|\bar w_{ii}|^2+|\bar w_{i 3-i}|^2).
\end{equation}
For later purpose we calculate
\begin{equation}
\label{est:m11}
\chi_i \bs M_i(\bs U+\bs \rho)=\chi_iU_i\rho^2_{3-i}+2 \chi_i U_{3-i}\rho_i\rho_{3-i}+\chi_i \rho_i\rho_{3-i}^2
\end{equation}
We have also 
\begin{equation}
\label{est:m12}
\chi_ {3-i}\bs M_i(\bs U+\bs \rho)=\chi_{3-i}U_i\rho^2_{3-i}+2 \chi_{3-i} U_{3-i}\rho_i\rho_{3-i}+\chi_{3-i} \rho_i\rho_{3-i}^2.
\end{equation}


}}

\subsection{The  outer problem}
Let $\bw=(\w_1, \w_2)$ be the outer approximation defined in (\ref{ans2}). We  set
\[
E^{out}_i:=-\Delta \w_i-f(\w_i,x).
\]
Let $\bs{\mv}=(\mv_1, \mv_2)$ be the inner approximation defined in (\ref{ans inner}). When needed we will  denote  $\bs{\mv}=\bs{\mv}(\tau, s;\mu, \xi)$ to indicate the dependence on the unknown modulation functions. Keep in mind that  the function $\bs{\mv}$ depends on the local variables $(t, s)$ through (\ref{var: stretched}). 

The error of the outer approximation and the matching error between the inner and the outer approximation  is estimated in the following:
\begin{lemma}\label{lem:error out}
It  holds 
\begin{equation}
\label{est:err out}
\|E_i^{out}\|_{C^{\alpha}(\Omega_i)}\lesssim\eps^5, \qquad i=1,2,
\end{equation}
with some $\alpha>0$. In addition for any $C>1$ we have
\begin{equation}
\label{est:matching}
\begin{aligned}
\|\w_i-\mv_i\|_{C^{0,\alpha}(\Omega_i\cap \{C^{-1}<|t|/\eps|\ln\eps|<C\})}&\lesssim (\eps|\ln\eps|)^{3-\alpha}+\eps^{1-\alpha} e^{\,-c|\ln\eps|^2}\\
&\qquad +\eps^{1-\alpha}|\ln\eps| \blu{e^{\,-\hat\lambda C^{-1}|\ln\eps|}}\left(\|\mu\|_{C^\alpha(\Gamma)}+\|\xi\|_{C^\alpha(\Gamma)}\right),\\
\|\partial_t(\w_i-\mv_i)\|_{C^{0,\alpha}(\Omega_i\cap \{C^{-1}<|t|/\eps|\ln\eps|<C\})}&\lesssim (\eps|\ln\eps|)^{2-\alpha}+\eps^{-\alpha} e^{\,-c|\ln\eps|^2}\\
&\qquad +\eps^{-\alpha} |\ln\eps| \blu{e^{\,-\bar\lambda C^{-1}|\ln\eps|}}\left(\|\mu\|_{C^{1,\alpha}(\Gamma)}+\|\xi\|_{C^{1,\alpha}(\Gamma)}\right).
\end{aligned}
\end{equation}
\end{lemma}
\begin{proof}
The proof of (\ref{est:err out}) follows directly by construction and is omitted.   

To prove (\ref{est:matching}) we write
 \[
 \w_i-\mv_i=\w_i-\mv_i(\cdot;\mu, \xi)=\w_i-\mv_i(\cdot;0,0)+\mv_i(\cdot;0,0)-\mv_i(\cdot;\mu,\xi).
 \]
 By construction 
 \begin{equation}
 \label{est:error out 1}
\|\w_i-\mv_i(\cdot;0,0)\|_{C^{0,\alpha}(\Omega_i\cap \{C^{-1}<|t|/\eps|\ln\eps|<C\})}\lesssim (\eps|\ln\eps|)^{3-\alpha}+\eps^{1-\alpha} e^{\,-c|\ln\eps|^2} 
 \end{equation}
 Similarly as in (\ref{transf:in-out}), taking into account (\ref{hyp:expl mod}),   we estimate
 \begin{equation}
 \label{est:error out 2}
 \|\mv_i(\cdot;\mu, \xi)-\mv_i(\cdot;0,0)\|_{C^{0,\alpha}(\Omega_i\cap \{C^{-1}<|t|/\eps|\ln\eps|<C\})}\lesssim \eps^{1-\alpha}|\ln\eps| \blu{e^{\,-\hat\lambda C^{-1}|\ln\eps|}}\left(\|\mu\|_{C^\alpha(\Gamma)}+\|\xi\|_{C^\alpha(\Gamma)}\right).
 \end{equation}
 Adding (\ref{est:error out 1}) and (\ref{est:error out 2}) we obtain the first estimate in (\ref{est:matching}). 
 
 To show the second estimate in (\ref{est:matching}) we write 
 \[
 \bw^0+\eps \bw^1+\eps^2\bw^2= \bar \bw, \qquad \bw-\bar \bw=\mr.
 \]
 By construction
 \begin{equation}
 \label{est:error out 3}
 \|\partial_t(\bar\w_i-\mv_i(\cdot;0,0))\|_{C^{0,\alpha}(\Omega_i\cap \{C^{-1}<|t|/\eps|\ln\eps|<C\})}\lesssim (\eps|\ln\eps|)^{2-\alpha}+\eps^{-\alpha} e^{\,-c|\ln\eps|^2},
 \end{equation}
 and
  \begin{equation}
 \label{est:error out 4}
 \|\partial_t r_i\|_{C^{0,\alpha}(\Omega_i\cap \{C^{-1}<|t|/\eps|\ln\eps|<C\})}\lesssim \eps^3.
 \end{equation}
 Estimating as in  (\ref{transf:in-out})
 \begin{equation}
 \label{est:error out 5}
 \|\partial_t(\mv_i(\cdot;0,0)-\mv_i(\cdot;\mu,\xi))\|_{C^{0,\alpha}(\Omega_i\cap \{C^{-1}<|t|/\eps|\ln\eps|<C\})}\lesssim \eps^{-\alpha} \blu{e^{\,-\hat\lambda C^{-1}|\ln\eps|}}\left(\|\mu\|_{C^{1,\alpha}(\Gamma)}+\|\xi\|_{C^{1,\alpha}(\Gamma)}\right)
 \end{equation}
 The last three estimates give the second bound in (\ref{est:matching}). 
 
 \end{proof} 
  
As a byproduct of the proof we observe that the map $(\mu, \xi)\mapsto \mv(\cdot; \mu, \xi)$ is Lipschitz in the following sense
\begin{corollary}\label{cor: lipschitz}
The following estimates hold
\begin{equation}
\label{est:error out 6}
\begin{aligned}
&\|\mv_i(\cdot;\mu+\delta\mu,\xi+\delta\xi)-\mv_i(\cdot;\mu,\xi)\|_{C^{0,\alpha}(\Omega_i\cap \{C^{-1}<|t|/\eps|\ln\eps|<C\})}\\
&\qquad \lesssim \eps^{1-\alpha} |\ln\eps|\blu{e^{\,-\hat\lambda C^{-1}|\ln\eps|}}\left(\|\delta\mu\|_{C^\alpha(\Gamma)}+\|\delta\xi\|_{C^\alpha(\Gamma)}\right),\\
&\|\partial_t(\mv_i(\cdot;\mu+\delta\mu,\xi+\delta\xi)-\mv_i(\cdot;\mu,\xi))\|_{C^{0,\alpha}(\Omega_i\cap \{C^{-1}<|t|/\eps|\ln\eps|<C\})}\\
&\qquad \lesssim \eps^{-\alpha}|\ln\eps| \blu{e^{\,-\hat\lambda C^{-1}|\ln\eps|}}\left(\|\delta\mu\|_{C^{1,\alpha}(\Gamma)}+\|\delta\xi\|_{C^{1,\alpha}(\Gamma)}\right).
\end{aligned}
\end{equation}  
\end{corollary}
  
{\color{black}{We  will further recast  the outer problem to give it a more convenient form. We have
\[
\begin{aligned}
\bar \rho q_{ii}&=-\bar\rho f_{u}(\mU_i,x)+\mathcal O(e^{\,-c|\ln\eps|^{2}})\\
&\qquad =-f_{u}(w_i^0,x)+(1-\bar\rho)f_{u}(w_i^0,x)-\bar\rho\( f_{u}(\mU_i,x)-f_{u}(w_i^0,x)\)+
\mathcal O(e^{\,-c|\ln\eps|^{2}})\\
\bar\rho q_{ij}&=\mathcal O(e^{\,-c|\ln\eps|^{2}}), \qquad i\neq j.
\end{aligned}
\]
We note that 
\[
f_{u}(w_i^0,x)= f_{u}(w_i^0,x)\mathbbm{1}_{\Omega_{i}}.
\]
Thus, we have
\[
\bar\rho \bs q= 
\left(\begin{array}{cc} -f_{u}(w_1^0,x)\mathbbm{1}_{\Omega_{1}}& 0\\
0 &-f_{u}(w_2^0,x)\mathbbm{1}_{\Omega_{2}}
\end{array}
\right)
+\eps  \bs{q}^{1},
\]
where $ \bs{q}^{1}$ is a $2\times 2$ matrix whose entries are H\"older continuous  functions. 

To solve the nonlinear outer problem we will first solve the following linear system with $i=1,2$:
\begin{equation}
\label{eq:outer lin}
\begin{aligned}
-\Delta \psi_i-f_{u}(w^0_i,x) \mathbbm{1}_{\Omega_{i}} \psi_i+\eps \left(\bs{q}^{1}\cdot\bs{\psi}\right)_i&= (1-\chi) h_i,\quad\mbox{in}\quad \Omega\setminus \Gamma,\\
 \psi_i&=0, \quad \hbox{on}\quad  \partial\Omega. \\
\end{aligned}
\end{equation}  
We assume that the right hand sides in (\ref{eq:outer lin}) are H\"older continuous. We look for a solution in the form 
\begin{equation}
\label{decomp:psi}
\bs{\psi}=(\psi_1, \psi_2)=(\mathbbm{1}_{\Omega_{1}}, \mathbbm{1}_{\Omega_{2}})\bar \psi+(\mathbbm{1}_{\Omega_{2}}\varphi_1,\mathbbm{1}_{\Omega_{1}}\varphi_2)
\end{equation}
where $\bar \psi\in C^{2,\alpha}$ and $\varphi_1\in C^{2,\alpha}(\Omega_2)$,  $\varphi_2\in C^{2,\alpha}(\Omega_1)$ are determined from the following system
\begin{equation}
\label{eq:outer lin psi}
\begin{aligned}
-\Delta \bar \psi-&\left\{\left[f_{u}(w^0_1,x)+ \eps q^1_{11}\right]\mathbbm{1}_{\Omega_{1}} +\left[f_{u}(w^0_2,x) + \eps q^1_{22}\right]\mathbbm{1}_{\Omega_{2}}\right\}\bar \psi +\eps\left[q^1_{12} \mathbbm{1}_{\Omega_{1}}\varphi_2+q^1_{21} \mathbbm{1}_{\Omega_{2}}\varphi_1\right] \\
&\qquad\qquad\qquad\qquad= (1-\chi)\left(h_1\mathbbm{1}_{\Omega_{1}}+h_2\mathbbm{1}_{\Omega_{2}}\right),\quad\mbox{in}\quad \Omega,\\
& \bar \psi=0, \quad \hbox{on}\quad  \partial\Omega. \\
\end{aligned}
\end{equation}  
and 
\begin{equation}
\label{eq:varphi_12}
\begin{aligned}
-\Delta\varphi_1+\eps q^{1}_{11} \mathbbm{1}_{\Omega_2} \varphi_1+\eps q^1_{12}\mathbbm{1}_{\Omega_2}\bar\psi&=(1-\chi)\mathbbm{1}_{\Omega_2} h_1, \quad\mbox{in}\ \Omega_2\\ 
-\Delta\varphi_2+\eps q^{1}_{22} \mathbbm{1}_{\Omega_1}\varphi_2+\eps q^1_{21}\mathbbm{1}_{\Omega_1}\bar\psi&=(1-\chi)\mathbbm{1}_{\Omega_1} h_2, \quad\mbox{in}\ \Omega_1\\
\varphi_1&=0=\varphi_2, \quad \mbox{on}\ \Gamma\\
\varphi_1&=0=\varphi_2, \quad \mbox{on}\ \partial\Omega.
\end{aligned}
\end{equation}
The last boundary condition could be vacuous for one of the functions $\varphi_i$. 
\begin{lemma}\label{lem out 1}
Assuming that  $(1-\chi) h_i\in C^{\alpha}(\Omega)$, $i=1,2$,  for each sufficiently small $\eps>0$ the system (\ref{eq:outer lin psi})--(\ref{eq:varphi_12}) has a unique solution $\bar \psi\in  C^{2,\alpha}(\Omega)$ and $\varphi_1\in C^{2,\alpha}(\Omega_2)$,  $\varphi_2\in C^{2,\alpha}(\Omega_1)$ such that 
\begin{equation}
\label{est:lin outer}
\|\bar\psi\|_{C^{2,\alpha}(\Omega)}+\|\varphi_1\|_{C^{2,\alpha}(\Omega_2)} +\|\varphi_2\|_{C^{2,\alpha}(\Omega_1)} \lesssim \|(1-\chi) h_1\|_{C^\alpha(\Omega)}+\|(1-\chi) h_2\|_{C^\alpha(\Omega)}.
\end{equation} 
Moreover the function $\bs{\psi}$ defined in (\ref{decomp:psi}) is a solution of (\ref{eq:outer lin}).
\end{lemma} 
\begin{proof}
By the hypothesis of non degeneracy (\ref{def-non-de}) given H\"older continuous functions $\varphi_i$ we can solve (\ref{eq:outer lin psi}) uniquely. Standard elliptic theory and non degeneracy imply  the following estimate:
\begin{equation}
\label{eq:lin inner 1}
\begin{aligned}
\|\bar \psi\|_{C^{2,\alpha}(\Omega)}\leq C\left(\|(1-\chi)h_1\|_{C^{\alpha}(\Omega_1)}+\|(1-\chi)h_2\|_{C^{\alpha}(\Omega_2)}\right)+\mathcal O(\eps)\left(\|\varphi_1\|_{C^{\alpha}(\Omega_2)}+\|\varphi_2\|_{C^{\alpha}(\Omega_1)}\right).
    \end{aligned} 
  \end{equation}
Since, for a given $\bar \psi\in C^{\alpha}(\Omega)$ each of the equations in (\ref{eq:varphi_12}) is a perturbation of the Poisson equation in $\Omega_i$ similarly we get the existence of a unique solutions such that 
\begin{equation}
\label{eq:lin inner 2}
\begin{aligned}
\|\varphi_{1}\|_{C^{2,\alpha}(\Omega_2)}&\leq C\|(1-\chi) h_1\|_{C^\alpha(\Omega_2)}+\mathcal O(\eps)\|\bar\psi\|_{C^\alpha(\Omega)}\\
\|\varphi_{2}\|_{C^{2,\alpha}(\Omega_1)}&\leq C\|(1-\chi) h_2\|_{C^\alpha(\Omega_1)}+\mathcal O(\eps)\|\bar\psi\|_{C^\alpha(\Omega)}
\end{aligned}
\end{equation} 
Combining (\ref{eq:lin inner 1})--(\ref{eq:lin inner 2}) with a straightforward fixed point argument we get the existence and uniqueness together with the estimate (\ref{est:lin outer}).  Checking directly we verify the rest of the Lemma. 
\end{proof}

Our next objective is to estimate the right hand side of the outer equation in (\ref{sys:O3}). Denote by $\bs e_1=(1,0)$, $\bs e_2=(0,1)$ and 
\[
\bs h= h_1 \bs e_1 + h_2 \bs e_2\qquad 
h_i= \(\Delta \w_i+F_i(\bs{\mU},x)-f_u(\w_i,x)\chi(\mv_i-\w_i)\)+(1-\hat\chi)N_i +\left[\Delta, \wt\chi\right]\wt v_i
\]
Lemma \ref{lem out 1} suggests that we should further decompose 
\[
\bs h=\mathbbm{1}_{\Omega_1} \bs h+\mathbbm{1}_{\Omega_2} \bs h=\sum_{i,j} h_{ij} \bs e_i, \qquad h_{ij}= h_i \mathbbm{1}_{\Omega_j} .
\]
Observe that 
\[
(1-\chi)\bs h=\sum_{i,j} \chi_j h_{ij}\bs e_i.
\]
As we will see the size of the error $\chi_jh_{ij}$  will depend on whether $i=j$ or $i\neq j$.
\begin{lemma}\label{lem out 2}
The following estimates  hold for $i=1,2$:
\begin{equation}
\label{est error 3}
\begin{aligned}
\|h_{ii}\|_{C^\alpha(\Omega_i)}& \lesssim (\eps|\ln \eps|)^{5-\alpha}+e^{\,-c|\ln \eps|^2}+ \eps^{-4-\alpha}\|\bar w_{3-i}\|^2_{C^\alpha(\Omega_i)}\\
&\quad +\eps^{-4-\alpha}\left(e^{\,-2K\theta |\ln\eps|/\hat m} \|\wt{\bs v}\|_{C_\theta^\alpha(\pC)}^2+\eps^2 e^{\,-K\theta|\ln\eps|/{\wt m}} \|\wt {\bs v}\|_{C^{1,\alpha}_\theta(\pC)}\right)\\
&\quad + \eps^{-4-\alpha}\|\bar w_{i}\|_{C^\alpha(\Omega_i)}\|\bar w_{3-i}\|^2_{C^\alpha(\Omega_i)}+\eps^{-\alpha} \left(\|\bar w_i\|^2_{C^\alpha(\Omega_i)}+\|\bar w_{3-i}\|^2_{C^\alpha(\Omega_i)}\right)
\\
\|h_{i 3-i}\| _{C^\alpha(\Omega_i)}& \lesssim e^{\,-c|\ln \eps|^2} + \eps^{-4-\alpha}\left(e^{\,-2K\theta |\ln\eps|/\hat m} \|\wt{\bs v}\|_{C_\theta^\alpha(\pC)}^2+\eps^2 e^{\,-K\theta|\ln\eps|/{\wt m}} \|\wt {\bs v}\|_{C^{1,\alpha}_\theta(\pC)}\right)\\
&\qquad +\eps^{-4-\alpha} \left(\|\bar w_i\|^2_{C^\alpha(\Omega_{3-i})}+\|\bar w_{3-i}\|^2_{C^\alpha(\Omega_{3-i})}\right)\\
&\qquad + \eps^{-4-\alpha}\|\bar w_{i}\|_{C^\alpha(\Omega_{3-i})}\|\bar w_{3-i}\|^2_{C^\alpha(\Omega_{3-i})}
\end{aligned}
\end{equation}
\end{lemma}
\begin{proof}
We start with the estimate for $h_{ii}$. We have
\begin{equation}
\label{def:hii}
h_{ii}=\chi_i \left(\Delta \w_i+F_i(\bs{\mU},x)-f_u(\w_i,x)\chi(\mv_i-\w_i)\right)+\chi_i(1-\hat\chi) N_i +\chi_i\left[\Delta, \wt\chi\right]\wt v_i.
\end{equation}
To estimate the first term on the right hand side above we denote $\delta_i=\mv_i-\w_i$ and write
\[
\begin{aligned}
\chi_i\left( \Delta \w_i+F_i(\bs{\mU},x)-f_u(\w_i,x)\chi\delta_i\right)&=\chi_i E^{out}_i+\chi_i\left(f(\w_i+\chi\delta_i, x)-f(\w_i,x)-f_u(\w_i,x)\chi\delta_i\right)\\
&\qquad -\eps^{-4}\chi_i(\w_i+\chi \delta_i)(\w_{3-i}+\chi \delta_{3-i})^2.
\end{aligned}
\]
Let us consider the last term above. By definition of the outer approximation $\w_{3-i}\equiv 0$ in $\Omega_i$. In addition, by definition of the inner solution (\ref{inner}) and {\it a priori} assumed size of the modulation functions $\mu$ and $\xi$ we see that 
\[
\|\chi_i(\w_i+\chi \delta_i)(\w_{3-i}+\chi \delta_{3-i})^2\|_{C^\alpha(\Omega_i)}=\|\chi_i\chi^2\mv^2_{3-i}(\w_i+\chi \delta_i)\|_{C^\alpha(\Omega_i)}\lesssim \eps^{1-3\alpha} e^{\,-c|\ln \eps|^2}. 
\]
From this  and Lemma \ref{lem:error out} it follows
\[
\begin{aligned}
\|\chi_i\left( \Delta \w_i+F_i(\bs{\mU},x)-f_u(\w_i,x)\chi\delta_i\right)\|_{C^\alpha(\Omega_i)}&\lesssim \|\chi_i E^{out}_i\|_{C^\alpha(\Omega_i)}+\|\chi_i\chi\delta_i\|^2_{C^\alpha(\Omega_i)}\\
&\quad +\eps^{-4}\|\chi_i(\w_i+\chi \delta_i)(\w_{3-i}+\chi \delta_{3-i})^2\|_{C^\alpha(\Omega_i)}\\
&\lesssim (\eps|\ln\eps|)^{5-\alpha}+e^{\,-c|\ln \eps|^2}
\end{aligned}
\]
To estimate the second term on the right hand side of (\ref{def:hii}) we use the decomposition (\ref{decomp nonlinear}). From (\ref{est:nf}) we we have
\[
\begin{aligned}
\|\chi_i (1-\hat\chi)\bs N_{f,i}(U_i+\rho_i,x)\|_{C^\alpha(\Omega_i)}&\lesssim \|\chi_i(1-\hat\chi)\wt \chi^2 \wt{\bs v}\|_{C^\alpha(\Omega_i)}^2+\|\chi_i\bar w_{ii}\|_{C^\alpha(\Omega_i)}^2\\
&\lesssim \eps^{-\alpha} \left(e^{\,-K\theta |\ln\eps|/\hat m} \|\wt{\bs v}\|_{C_\theta^\alpha(\pC)}^2+\|\bar w_{i}\|_{C^\alpha(\Omega_i)}^2\right).
\end{aligned}
\]
Also from  (\ref{decomp nonlinear}),  using (\ref{est:m11})  and $|\wt\chi U_i|\lesssim \eps|\ln\eps|$ we find 
\[
\begin{aligned}
\eps^{-4}\|\chi_i (1-\hat\chi)\bs M_i(\bs U+\bs \rho)\|_{C^\alpha(\Omega_i)}&\lesssim \eps^{-4-\alpha} \left(e^{\,-K\theta |\ln\eps|/\hat m}\eps|\ln\eps| \|\wt{\bs v}\|_{C_\theta^\alpha(\pC_\eps)}^2
+\|\bar w_{3-i}\|^2_{C^\alpha(\Omega_i)}\right)\\
&+\eps^{-4-\alpha} e^{\,-c|\ln\eps|^2} \left(\|\bar w_{1}\|^2_{C^\alpha(\Omega_i)}+\|\bar w_{2}\|^2_{C^\alpha(\Omega_i)}\right)\\
&+\eps^{-4-\alpha}\|\bar w_{i}\|_{C^\alpha(\Omega_i)}\|\bar w_{3-i}\|^2_{C^\alpha(\Omega_i)}.
\end{aligned}
\]
Finally, the commutator term is estimated using (\ref{transf:com 2}) by
\begin{equation}
\label{est:comm_out ii}
\left\|\commutator{\Delta}{\wt \chi}{\wt v_i}\right\|_{C^\alpha(\Omega_i)}\lesssim \eps^{-2-\alpha} e^{\,-K\theta|\ln\eps|/{\wt m}} \|\wt {\bs v}\|_{C^{1,\alpha}_\theta(\pC)}
\end{equation}
Summarizing these estimates (and taking $\eps$ small) we get the first bound in (\ref{est error 3}). 

To show the second bound in (\ref{est error 3}) we write for $i=1,2$: 
\begin{equation}
\label{def:hi3-i}
h_{i3-i}=\chi_{3-i} \left(\Delta \w_{i}+F_{i}(\bs{\mU},x)-f_u(\w_i,x)\chi\delta_i\right)+\chi_{3-i} (1-\hat\chi) N_{i} +\chi_{3-i}\left[\Delta, \wt\chi\right]\wt v_{i}.
\end{equation}
In the $\supp\chi_{3-i}$ we have $\w_{i}\equiv 0$, and $\mv_{i}=\mathcal O(e^{\,-c|\ln\eps|^2})$ by construction hence
\[
\|\chi_{3-i} \left(\Delta \w_{i}+F_{i}(\bs{\mU},x)-f_u(\w_i,x)\chi\delta_i\right)\|_{C^\alpha(\Omega_{3-i})}\lesssim  e^{\,-c|\ln\eps|^2}.
\]
Similarly as above
\[
\|\chi_{3-i} (1-\hat\chi)\bs N_{f,i}(U_i+\rho_i,x)\|_{C^\alpha(\Omega_{3-i})}\lesssim \eps^{-2\alpha} \left(e^{\,-2K\theta |\ln\eps|/\hat m} \|\wt{\bs v}\|_{C_\theta^\alpha(\pC)}^2+\|\bar w_{i3-i}\|_{C^\alpha(\Omega_{3-i})}^2\right)
\]
From (\ref{est:m12}) we get
\[
\begin{aligned}
\eps^{-4}\|\chi_{3-i}(1-\hat\chi)\bs M_i(\bs U+\bs \rho)\|_{C^\alpha(\Omega_{3-i})}&\lesssim \eps^{-4-\alpha}e^{\,-2K\theta |\ln\eps|/\hat m}\eps|\ln\eps| \|\wt{\bs v}\|_{C_\theta^\alpha(\pC)}^2\\
&+\eps^{-4-\alpha} \left(\|\bar w_{1}\|^2_{C^\alpha(\Omega_{3-i})}+\|\bar w_{2}\|^2_{C^\alpha(\Omega_{3-i}i)}\right)\\
&+\eps^{-4-\alpha}\|\bar w_{i}\|_{C^\alpha(\Omega_{3-i})}\|\bar w_{3-i}\|^2_{C^\alpha(\Omega_{3-i})}.
\end{aligned}
\]
The commutator term is estimated in a similar was as in the estimate for $h_{ii}$. Combing the above bounds we conclude the proof.
\end{proof}

As a byproduct of the above we have
 \begin{corollary}\label{cor:lipschitz 2}
The maps
\[
(\bs{\rho}, \mu, \xi)\longmapsto h_{ij}(\cdot;\bs{\rho}, \mu, \xi)
\]
are Lipschitz in the following sense
\begin{equation}
\label{est:lipschitz 2}
\begin{aligned}
&\|h_{ij}(\cdot;\bs{\rho}, \mu, \xi)-h_{ij}(\cdot;\bs{\rho}+\delta\bs\rho, \mu+\delta\mu, \xi+\delta\xi)\|_{C^\alpha(\Omega_j)} \\
&\qquad \lesssim \eps^{-\alpha} \blu{e^{\,-K\hat\lambda|\ln\eps|}}\left(\|\delta\mu\|_{C^{1,\alpha}(\Gamma)}+\|\delta\xi\|_{C^{1,\alpha}(\Gamma)}\right) \\
& \qquad\qquad +\eps^{\wt\varsigma}\left(\eps^{-\alpha} e^{\,-K\theta|\ln\eps|/{\wt m}} +\eps^{-2-\alpha}e^{\,-2K\theta |\ln\eps|/\hat m}\right) \|\delta\wt{\bs v}\|_{C_\theta^\alpha(\pC)}\\
 &\qquad \qquad +\eps^{\bar\varsigma-\alpha}\sum_{i=1,2}\left(\|\delta\bar{w}_{i}\|_{C^{\alpha}(\Omega_i)}+\|\delta\bar{w}_{3-i}\|_{C^{\alpha}(\Omega_{3-i})}\right),
\end{aligned}
\end{equation}
where, respectively,  $\bar\varsigma, \wt\varsigma>0$ are  the constants in (\ref{hyp:size w}) and (\ref{hyp:size v}). 
  \end{corollary}
  }}

{\color{black}{ 
 \subsection{The error of the  inner approximation}\label{sec. 4.5}
 We denote the right hand side of the first equation in (\ref{sys:O3}) by ${\bs g}=(g_1, g_2)$:
 \begin{equation}
 \label{rhs 4.10}
 g_i=- \wt\rho(\bs q_i-\bs q_i^0)\cdot  \wt{\bs v}+ \chi \(\Delta \mv_i+F_i(\bs {\mU},x)+f_u(\w_i,x)(1-\chi)\delta_i\)+\hat\chi N_i+\commutator{\Delta}{\chi}\delta_i +\commutator{\Delta}{\bar\chi}\bar w_i.
  \end{equation}
 The function ${\bs g}$  represents the  inner error  we will calculate now. 
  
 Below we will will use the same symbol for  functions of the local, inner variables $(t,s)$ and the stretched variables $(\tau, s)$ defined in (\ref{var: stretched}).  
Using (\ref{lap1}) it is straightforward to find the expression of the Laplacian in terms of $(\tau, s)$. We denote 
\[
\A_\eps(\tau, s)=\left(1-\eps\kappa(s)(\tau b^{-1}(s)+\zeta(s))\right)^2,
\]
and 
\[
\dot b(s)=\frac{d}{ds} b(s), \quad \ddot b(s)= \frac{d^2}{ds^2} b(s),\quad  \dot\zeta(s)=\frac{d}{ds} \zeta(s), \quad \dot\kappa(s)=\frac{d}{ds} \kappa(s),
\]
with similar rule for other functions of $s$. 
With this notation 
\begin{equation}
\label{exp:laplace inner}
\begin{aligned}
\Delta&= \eps^{-2}\(b^2\partial_{\tau\tau}+\eps^2\A_\eps^{-1}\partial_{ss}\)-\eps^{-1}\kappa b \A_\eps^{-1/2}\partial_\tau+a_{11}\partial_{\tau\tau}+a_{12}\partial_{\tau s}+c_{1}\partial_\tau+\eps c_2\partial_s,\\
\end{aligned}
\end{equation}
where
\[
\begin{aligned}
a_{11}&=\A_\eps^{-1}(\dot b b^{-1}\tau-b\dot \zeta)^2, \qquad a_{12}=2\A_\eps^{-1}(\dot bb^{-1}\tau-b\dot\zeta),\\
c_1&=\left(\eps\A_\eps^{-3/2}\dot\kappa(\tau b^{-1} +\zeta)(\dot b b^{-1} -b\dot \zeta)+\A^{-1}_\eps(\ddot b b^{-1} \tau-2\dot b\dot \zeta-b\ddot \zeta)\right), \qquad c_2=\A_\eps^{-3/2} \dot \kappa(b^{-1}\tau+\zeta).
\end{aligned}
\]
For later purpose we set
\[
\begin{aligned}
\mathcal D^0_\eps&=\eps^{-2}\(b^2\partial_{\tau\tau}+\eps^2 \A_\eps^{-1}\partial_{ss}\)-\eps^{-1}\kappa b \A_\eps^{-1/2}\partial_\tau,\\ \
\mathcal D^1_\eps&=a_{11}\partial_{\tau\tau}+a_{12}\partial_{\tau s}+c_{1}\partial_\tau+\eps c_2\partial_s.
\end{aligned}
\]
Our first task is to calculate 
\[
\bs {E}^{in}(\bs\mv)= \left(\begin{array}{c} \Delta\mv_1-\eps^{-4} \mv_1\mv_2^2 \\
\Delta\mv_2-\eps^{-4} \mv_1^2\mv_2\end{array}\right), 
\]
in the neighborhood  $\mathcal U_\chi=\{\chi\neq 0\}$, the interior of the $\supp \chi$, around $\Gamma$.    
Recall that the approximate solution $\bs \mv$ depends on the modulation functions $\mu$ and $\xi$. For notational convenience we set
\[
m(\tau, s)= b(s)(1+\mu(\tau, s)), \qquad k(\tau, s)=\tau\mu(\tau, s)+\xi(\tau, s),
\]
so that 
\[
\bs \mv(\tau, s)=\eps m \mathbf V(\tau+k)+\eps^2\kappa(s) \mathbf W(\tau+k).
\]
We will denote
\[
\begin{aligned}
&\mbV'(\tau+k)=\left.\frac{d}{dx} \mbV(x)\right|_{x=\tau+k}, \quad \mbV''(\tau+k)=\left.\frac{d^2}{dx^2} \mbV(x)\right|_{x=\tau+k}, \\
&\mbW'(\tau+k)=\left.\frac{d}{dx} \mbW(x)\right|_{x=\tau+k}, \quad \mbW''(\tau+k)=\left.\frac{d^2}{dx^2} \mbW(x)\right|_{x=\tau+k}.
\end{aligned}
\]
We calculate 
\[
\begin{aligned}
\mathcal D^0_\eps (\eps m\mbV)&=\eps^{-1}\(b^2\partial_{\tau\tau}+\eps^2 \A_\eps^{-1}\partial_{ss}\)(m\mbV)-\kappa b \A_\eps^{-1/2}\partial_\tau (m\mbV)
\\ &=\eps^{-1} b^2 m\mbV''-\kappa b m \mbV' +\mathbf \Omega +\bs{\Theta}_1
\end{aligned}
\]
where 
\[
\begin{aligned}
\mathbf\Omega&=\eps^{-1}\mbV\left(b^2\partial_{\tau\tau} m+\eps^2\A_\eps^{-1} \partial_{ss} m\right)+\eps^{-1}\mbV'\left(b^2(2\partial_\tau m+m\partial_{\tau\tau} k)
+\eps^2\A_\eps^{-1}m\partial_{ss} k\right)\\
&\qquad+2\eps^{-1}m b^2\mbV''\partial_\tau k,\\
\bs{\Theta}_1&=\eps^{-1}\left(2b^2\partial_\tau m\partial_\tau k+2\eps^2\A_\eps^{-1} \partial_s m\partial_s k\right)\mbV'+\eps^{-1}m\left(b^2(\partial_\tau k)^2+\eps^2\A_\eps^{-1}(\partial_s k)^2\right)\mbV''\\
&\qquad -\kappa b \A_\eps^{-1/2}\partial_\tau m\mbV-\eps\kappa^2 b (\tau b^{-1}+\zeta)\A_\eps^{-1/2} m \mbV'.
\end{aligned}
\]
Similarly we have
\[
\eps^2\mathcal D^0_\eps(\kappa\mbW)=\(b^2\partial_{\tau\tau}+\eps^2\A_\eps^{-1}\partial_{ss}\)(\kappa\mbW)-\eps\kappa b \A_\eps^{-1/2}\partial_\tau (\kappa\mbW)= b m\kappa \mbW''+\bs{\Theta}_2
\]
where
\[
\bs{\Theta}_2=b \kappa \left(2b\partial_\tau k+b(\partial_\tau k)^2-\mu\kappa\right)\mbW''+b^2 \kappa\partial_{\tau\tau}k \mbW'+\eps^2 \A_\eps^{-1}\partial_{ss}(\kappa\mbW)-\eps\kappa b \A_\eps^{-1/2}\partial_\tau (\kappa\mbW).
\]
{\color{black}{Note that the terms involved in $\mathcal D^1_\eps$ are similar to those of $\mathcal D^0_\eps$ but multiplied additionally  by $\eps$.
}}

The components of the  nonlinear term in $\bs E^{in}$ are, with $i=1,2$:
\[
\begin{aligned}
-\eps^{-4} \mv_i\mv^2_{3-i}&=-\eps^{-1}(m V_i+\eps\kappa W_i)(m V_{3-i}+\eps\kappa W_{3-i})^2\\
&=-\eps^{-1} b^2m V_iV_{3-i}^2-\blu{\kappa bm} (V_{3-i}^2 W_i+2 V_{3-i} V_i W_{3-i})+\Theta_{3i}
\end{aligned}
\]
where $\mathbf {\Theta}_3=(\Theta_{31}, \Theta_{32})$ is
\[
\begin{aligned}
\Theta_{3i}&=-\eps^{-1}m\blu{(m^2-b^2)}\mu V_i V_{3-i}^2-\kappa m\mu \left(V_{3-i}^2 W_i-2 V_{3-i}V_i W_{3-i}\right)\\
&\qquad - \eps m\kappa^2\left(V_i W_{3-i}^2+2 V_{3-i} W_i W_{3-i}\right)-\eps^2 \kappa^3 W_i W_{3-i}^2.
\end{aligned}
\]
We will denote 
\[
\mathbf{\Theta}_4=\mathcal D^1_\eps(\eps m \mbV+\eps^2 \kappa \mbW)= \left(a_{11}\partial_{\tau\tau}+a_{12}\partial_{\tau s}+c_{1}\partial_\tau+\eps c_2\partial_s\right)(\eps m \mbV+\eps^2 \kappa \mbW). 
\]
Adding the above identities and using the definition of $\mbW$ we get
\[
\bs E^{in}=\mathbf \Omega+\sum_{j=1}^4\mathbf{\Theta}_j.
\]
We will estimate the terms on the right hand side of the above identity and also study its Lipschitz character as functions of the modulations functions $\mu, \xi$ in the weighted H\"older space $C^{\alpha}_\theta(\pC)$. Before doing this we will derive a more explicit formula for $\mathbf\Omega=\mathbf \Omega_0+\mathbf\Omega_1$, where  
\begin{equation}
\label{def omega zero}
\begin{aligned}
\mathbf\Omega_0&=\eps^{-1}b(\mbV+\tau\mbV') (b^2\partial_{\tau\tau}\mu+\eps^2\A_\eps^{-1}\partial_{ss} \mu)+\eps^{-1}b \mbV'(b^2\partial_{\tau\tau}\xi+\eps^2\A_\eps^{-1}\partial_{ss}\xi)
\\
&\quad +4\eps^{-1}b^3\mbV'\partial_\tau\mu+2\eps^{-1}b^3\mbV''(\mu+\tau\partial_\tau\mu+\partial_\tau \xi),
\\
\mathbf\Omega_1&=\mbV \A_\eps^{-1}(\eps \ddot bb^{-1}m+2\eps\dot b\partial_s\mu)+\eps^{-1}\mbV'b^2\mu\partial_{\tau\tau}k+\eps \A_\eps^{-1} \mbV' \mu\partial_{ss}k+2\eps^{-1}b^2\mu\mbV''\partial_\tau k.
\end{aligned}
\end{equation}
The operator  $(\mu, \xi)\mapsto \mathbf\Omega_0$ will play an important role in the derivation of the modulation equation later on. We recall that $\chi=\chi(t/\eta)$ with $\eta = K\eps|\ln\eps|$. Using
\[
t=\eps\left(\tau b(s)+\zeta(s)\right)
\]
  below we will consider the cut off function $\chi$ as a function of $(\tau, s)$. Since the functions $b$ and $\zeta$ have already been determined and they are at least of class $C^{3,\alpha}(0,|\Gamma|)$ we will allow the constants in the estimates below to depend on them.  With this said the proof of the following Lemma is straightforward.

\begin{lemma}\label{lem:est inner error}
The following estimates hold
\[
\begin{aligned}
\|\chi \bs \Omega_0\|_{C^\alpha_\theta(\pC)}&\lesssim \eps^{-1}\left(\|\mu\|_{C^{2,\alpha}(\Gamma)}+\|\xi\|_{C^{2,\alpha}(\Gamma)}\right),\\
\|\chi \bs \Omega_1\|_{C^\alpha_\theta(\pC)}&\lesssim \eps e^{\,2 K\theta|\ln\eps|}|\ln\eps| +\|\mu\|_{C^{1,\alpha}(\Gamma)} + \eps^{-1}\left(\|\mu\|^2_{C^{2,\alpha}(\Gamma)}+\|\xi\|^2_{C^{2,\alpha}(\Gamma)}\right),
\end{aligned}
\]
and
\[
\begin{aligned}
\|\chi\mathbf\Theta_1\|_{C^\alpha_\theta(\pC)}&\lesssim \eps e^{\,2 K\theta|\ln\eps|}|\ln\eps|+\left(\|\mu\|_{C^{2,\alpha}(\Gamma)}+ \|\xi\|_{C^{2,\alpha}(\Gamma)}\right)+ \eps^{-1}\left(\|\mu\|^2_{C^{2,\alpha}(\Gamma)}+\|\xi\|^2_{C^{2,\alpha}(\Gamma)}\right)
\\
\|\chi\mathbf\Theta_2\|_{C^\alpha_\theta(\pC)}&\lesssim \eps e^{\,2K\theta|\ln\eps|}|\ln\eps|+\left(\|\mu\|_{C^{2,\alpha}(\Gamma)}+ \|\xi\|_{C^{2,\alpha}(\Gamma)}\right),
\\
\|\chi\mathbf\Theta_3\|_{C^\alpha_\theta(\pC)}&\lesssim\eps e^{\,2 K\theta|\ln\eps|}|\ln\eps|^5+\eps^{-1}\|\mu\|^{\blu{2}}_{C^{0,\alpha}(\Gamma)},
\\
\|\chi\mathbf\Theta_4\|_{C^\alpha_\theta(\pC)}&\lesssim \eps e^{\,2 K\theta|\ln\eps|}|\ln\eps| +\eps \left(\|\mu\|_{C^{2,\alpha}(\Gamma)}+ \|\xi\|_{C^{2,\alpha}(\Gamma)}\right).
\end{aligned}
\]
In addition the function $(\mu, \xi)\mapsto \chi \bs E^{in}$ is Lipschitz and satisfies:
\[
\|\chi \bs E^{in}(\cdot; \mu,\xi)-\chi \bs E^{in}(\cdot; \mu+\delta\mu,\xi+\delta \xi)\|_{C^\alpha_\theta(\pC)}\lesssim \eps^{-1}  \|\delta\mu\|_{C^{0,\alpha}(\Gamma)}+\left(\|\delta\mu\|_{C^{2,\alpha}(\Gamma)}+ \|\delta\xi\|_{C^{2,\alpha}(\Gamma)}\right).
\]
\end{lemma}

Next we will estimate the remaining terms in the error $\bs g$:
\[
\bs j=\bs g-\bs {E}^{in},
\]
where $\bs j=\sum_{k=1}^5 \bs j_k$, 
\[
\begin{aligned}
\bs j_1&=- \wt\rho(\bs q-\bs q^0)\cdot  \wt{\bs v}, \\ \bs j_2&=\chi \left(\begin{array}{c}
F_1(\bs {\mU},x)+f_u(\w_1,x)(1-\chi)\delta_1+\eps^{-4} \mv_1\mv_2^2\\
F_2(\bs {\mU},x)+f_u(\w_2,x)(1-\chi)\delta_2+\eps^{-4} \mv^2_1\mv_2
\end{array}
\right) +\chi \bs {E}^{in}, \\ \bs j_3&=\hat \chi \bs N(\bs {\mU}+ \bs \rho,x) 
\end{aligned}
\]
(see (\ref{decomp nonlinear})), an $\bs j_4$, $\bs j_5$ are the vectors of  the commutator terms in  (\ref{rhs 4.10}):
\[
\bs j_4=\commutator{\Delta}{\chi}\bs \delta, \qquad \bs j_5=\commutator{\Delta}{\bar\chi}\bar{\bs w}.
\] 

\begin{lemma}\label{lem:est inner error 2}
With the above notation we have
\begin{equation}
\label{est:j1}
\begin{aligned}
\|\bs j_1\|_{C^\alpha_\theta(\pC)}&\lesssim \left(\eps^{-1}+\eps^{-2-\alpha}(\|\mu\|_{C^\alpha(\Gamma)}+\|\xi\|_{C^\alpha(\Gamma)})\right)|\ln\eps|^5\|\wt{\bs v}\|_{C^\alpha_\theta(\pC)},
\\
\|\bs j_2\|_{C^\alpha_\theta(\pC)}&\lesssim \eps e^{\,2K\theta|\ln\eps|}|\ln\eps|+\eps^{1-\alpha} |\ln\eps| \blu{e^{\,(-\hat\lambda+\theta)  K|\ln\eps|}} \left(\|\mu\|_{C^{\alpha}(\Gamma)}+ \|\xi\|_{C^{\alpha}(\Gamma)}\right)+\eps^{-4} e^{\,-c|\ln\eps|^2},
\\
\|\bs j_3\|_{C^\alpha_\theta(\pC)}&\lesssim \eps^{-3}|\ln\eps|\left(\|\wt{\bs v}\|_{C^\alpha_\theta(\pC)}^2+e^{\,2K\theta|\ln\eps|/\hat m}\sum \|\bar{w}_{ij}\|^2_{C^\alpha(\Omega_j)}\right),
\\
\|\bs j_4\|_{C^\alpha_\theta(\pC)}&\lesssim \eps (|\ln\eps|)^{3-\alpha}e^{\,2 K\theta|\ln\eps|}+\eps^{-1-\alpha} \blu{e^{\,(-\hat\lambda+\theta)  K|\ln\eps|}}\left(\|\mu\|_{C^{1,\alpha}(\Gamma)}+\|\xi\|_{C^{1,\alpha}(\Gamma)}\right),
\\
\|\bs j_5\|_{C^\alpha_\theta(\pC)}&\lesssim \eps^{-2-\alpha} e^{\,2K\theta|\ln\eps|/\bar m} \sum \|\bar{w}_{ij}\|_{C^{1,\alpha}(\Omega_j)}
\end{aligned}
\end{equation}
\end{lemma}

\begin{proof}
We begin by proving the first bound in (\ref{est:j1}). Denote, as above, $\delta_i=\mv_i-\w_i$ so that 
\[
\bs\mU=(\mv_1, \mv_2)-((1-\chi) \delta_1, (1-\chi)\delta_2).
\]
Directly from the expressions (\ref{def:qq}) and (\ref{def:qzero}) and (\ref{est:matching}) we find
\[
|\bs j_1|\lesssim \eps^3\wt\rho (|\mv^0|^2+|\mv^1|^2)+\eps\wt\rho (1-\chi)(|\mv^0|+|\mv^1|)(|\delta_1|+|\delta_2|).
\]
From this the required bound follows. 

Using (\ref{transf:out-in}), (\ref{est:matching}) we find
\[
\|\chi\chi_i \delta_i\|_{C^\alpha_\theta(\pC)}\lesssim (\eps|\ln\eps|)^{3-\alpha}e^{\,2 K\theta|\ln\eps|}+\blu{\eps^{1-\alpha}|\ln\eps| e^{\,(-\hat\lambda+\theta)  K|\ln\eps|}} \left(\|\mu\|_{C^\alpha(\Gamma)}+\|\xi\|_{C^\alpha(\Gamma)}\right).
\]
The $i$th component of $\bs j_2$ is
\[
\begin{aligned}
{j}_{2i}&=\chi f(\mv_i-(1-\chi)\delta_i,x)+f_u(\w_i, x)\chi (1-\chi)\delta_i \\
&\qquad-\eps^{-4}(\mv_i-(1-\chi)\delta_i)(\mv_i-(1-\chi)\delta_{3-i})^2+\eps^{-4}\chi \mv_i\mv_{3-i}^2+\chi E^{in}_i.
\end{aligned}
\]
The first two terms are estimated by
\[
|\chi f(\mv_i-\chi_i\delta_i,x)+f_u(\w_i, x)\chi (1-\chi)\delta_i|\lesssim \chi|\mv_i|+\chi(1-\chi)|\delta_i|
\]
and the last term by 
\[
\begin{aligned}
&\eps^{-4}\chi|\mv_i(2(1-\chi)\mv_{3-i}\delta_{3-i}-(1-\chi)^2\delta_{3-i}^2)-(1-\chi)\delta_i(\mv_{3-i}-(1-\chi)\delta_{3-i})^2| \\
&\qquad \lesssim \eps^{-4}\chi\left((1-\chi)|\mv_i||\mv_{3-i}||\delta_{3-i}|+ (1-\chi)^2|\mv_i||\delta_{3-i}|^2+(1-\chi)\delta_i|\mv_{3-i}|^2
\right).
\end{aligned}
\]
Note that 
\[
\chi(1-\chi)|\mv_i| |\delta_{3-i}|+\chi(1-\chi)|\mv_{3-i}||\delta_{i}|\lesssim e^{\,-c|\ln\eps|^2}, \quad i=1,2.
\] 
Using the above we obtain the second estimate in (\ref{est:j1}) in a straightforward way.

To how the third estimate in (\ref{est:j1}) we decompose  $\bs N$ as in (\ref{decomp nonlinear}). Using (\ref{est:nf}) we find
\[
\|\hat \chi \bs N_{f}(\bs{U}+\bs \rho,x)\|_{C^\alpha_\theta(\pC)}\lesssim \|\wt{\bs v}\|_{C^\alpha_\theta(\pC)}^2+e^{\,2K\theta|\ln\eps|/\hat m}\sum \|\bar{w}_{ij}\|^2_{C^\alpha(\Omega_j)}.
\]
From (\ref{est:m11}) replacing $\chi_i$ by $\hat \chi$ we get
\[
\hat \chi \bs M_i(\bs U+\bs \rho)=\hat \chi U_i\rho^2_{3-i}+2 \hat \chi U_{3-i}\rho_i\rho_{3-i}+\hat \chi \rho_i\rho_{3-i}^2,
\]
hence
\[
\eps^{-4}\|\hat \chi \bs M_i(\bs U+\bs \rho)\|_{C^\alpha_\theta(\pC)}\lesssim \eps^{-3}|\ln\eps|\left(\|\wt{\bs v}\|_{C^\alpha_\theta(\pC)}^2+e^{\,2K\theta|\ln\eps|/\hat m}\sum \|\bar{w}_{ij}\|^2_{C^\alpha(\Omega_j)}\right).
\]
The latter estimate gives the second bound in (\ref{est:j1}).

Next we estimate the first commutator term the definition of $\bs g$. Slight modification of (\ref{transf:com 1})  and (\ref{est:matching}) gives
\[
\left\|\commutator{\Delta}{\chi}\delta_i\right\|_{C^\alpha_\theta(\pC)}\lesssim \eps (|\ln\eps|)^{3-\alpha}e^{\,2 K\theta|\ln\eps|}+\eps^{-1-\alpha} \blu{e^{\,(-\hat\lambda+\theta)  K|\ln\eps|}} \left(\|\mu\|_{C^{1,\alpha}(\Gamma)}+\|\xi\|_{C^{1,\alpha}(\Gamma)}\right)
\]
The remaining estimate is straightforward at this point. 

\end{proof}

As for the Lipschitz dependence of $\bs j$ on the various unknown functions involved we state without a proof the following:
\begin{corollary}
\label{cor:lip g}
The function
\[
(\wt{\bs v},\bar{\bs w}, \mu, \xi)\longmapsto \bs j(\cdot; \wt{\bs v},\bar{\bs w}, \mu, \xi),
\]
is Lipschitz in the following sense
\[
\begin{aligned}
& \|\bs j(\cdot; \wt{\bs v},\bar{\bs w}, \mu, \xi)- \bs j(\cdot; \wt{\bs v}+\delta \wt{\bs v},\bar{\bs w}+\delta \bar{\bs w}, \mu+\delta\mu, \xi+\delta\xi)\|_{C^\alpha_\theta(\pC)}\lesssim 
\left(\eps^{-1}+\eps^{-\hat \varsigma -\alpha}\right)|\ln\eps|^5\|\delta\wt{\bs v}\|_{C^\alpha_\theta(\pC)}
\\
&\qquad +\eps^{-2-\alpha} e^{\,2K\theta|\ln\eps|/\bar m} \sum \|\delta\bar{w}_{ij}\|_{C^\alpha(\Omega_j)}\\
 &\qquad+
 \eps^{-1-\alpha} \blu{e^{\,(-\hat\lambda+\theta)  K|\ln\eps|}} \left(\|\delta\mu\|_{C^{1,\alpha}(\Gamma)}+\|\delta\xi\|_{C^{1,\alpha}(\Gamma)}\right)
 \\
 &\qquad + \eps^{-3}|\ln\eps|\left(\|\wt{\bs v}\|_{C^\alpha_\theta(\pC)}\|\delta\wt{\bs v}\|_{C^\alpha_\theta(\pC)}+e^{\,2K\theta|\ln\eps|/\hat m}\sum \|\bar{w}_{ij}\|_{C^\alpha(\Omega_j)}\|\delta\bar{w}_{ij}\|_{C^\alpha(\Omega_j)}\right)
 \end{aligned}
 \]
\end{corollary}

After these preparations  we will bring the inner equation (\ref{sys:O3}) to the form in which the theory described in section \ref{sec: line oper} applies. We use (\ref{exp:laplace inner}) to pass to the inner stretched variables. 
{\color{black} {To simplify suitably the inner equation it is convenient to introduce new unknown
\[
\bs{z}=b^{-1} \wt{\bs v}.
\]
The results of Lemma \ref{lem:est inner error 2} and Corollary \ref{cor:lip g} remain unchanged if we replace $\wt{\bs v}$ by $\bs z$ on the right hand sides of the error estimates.

}}
Recall the operator $\bs{\mathfrak L}$ defined in (\ref{def:inner lin oper}). Simple manipulations allow us to  write the equation for ${\bs z}$ in the form  
\begin{equation}
\label{eq:inner_1}
\bs{\mathfrak L} {\bs z}=\eps^2 b^{-3}(s)\bs i(\bs z)+\eps^{2}b(s)^{-3}{\bs g},
\end{equation}
where
\begin{equation}
\label{eq:inner_2}
{\bs i}(\bs z)=\wt\rho\left[\eps^{-2}b^{3}(s)\bs{\mathfrak L} (\bs z)+\Delta_{\tau, s}(b(s) \bs z)-\bs q^0\cdot (b(s)\bs z)\right],
\end{equation}
$\Delta_{\tau, s}$ is the operator on the right hand side of (\ref{exp:laplace inner}) and $\bs g$ is defined in (\ref{rhs 4.10}). To control the right hand side of (\ref{eq:inner_1}) we only need to estimate the fist term in (\ref{eq:inner_2}). We set
\[
\bs i({\bs z})=\bs i_{1}({\bs z})+\bs i_{2}({\bs z}),
\]
where
\[
\begin{aligned}
\bs{i}_1(\bs z)&=\wt\rho\left[-\left(\eps^{-2}b^3\partial_{\tau\tau}+b(s)\partial_{ss}\right)\bs z+\Delta_{\tau, s}(b(s)\bs z)\right], \\
 \bs i_2(\bs z)&=\eps^{-2}\wt\rho\left(\begin{array}{cc}
b_0^2 \mathbf V_2^2-(\mv^0_2)^2 &2b_0^2\mathbf V_1\mathbf V_2-2\mv^0_1\mv^0_2\\
2b_0^2\mathbf V_1\mathbf V_2-2 \mv^0_1\mv^0_2 & b_0^2 \mathbf V_2^1- (\mv^0_1)^2
\end{array}\right)\cdot (b(s)\bs z).
\end{aligned}
\]
Note that $\bs i_2$ depends additionally on the unknown modulation functions $\mu, \xi$. The  estimates gathered below are not hard to prove at this point.
\begin{lemma}\label{lem:est inner error 3}
It holds
\begin{equation}
\label{est:i}
\begin{aligned}
\|\bs i_{1}({\bs z})\|_{C^\alpha_\theta(\pC)}&\lesssim \eps^{-1}|\ln\eps|\|{\bs z}\|_{C^{2,\alpha}_\theta(\pC)}\\
\|\bs i_{2}({\bs z})\|_{C^\alpha_\theta(\pC)}&\lesssim \eps^{-1}|\ln\eps|^2\left(1+\eps^{-1}\|\mu\|_{C^\alpha(\Gamma)}+\eps^{-1}\|\xi\|_{C^\alpha(\Gamma)}\right)\|{\bs z}\|_{C^\alpha_\theta(\pC)}.
\end{aligned}
\end{equation}
Also, the  functions $({\bs z}, \mu, \xi)\mapsto \bs i_{i}(\bs z; \mu, \xi)$ are Lipschitz in the following sense
\[
\begin{aligned}
&\|\bs i_{1}(\bs z)-\bs i_{1}(\bs z+\delta {\bs z})\|_{C^\alpha_\theta(\pC)}\lesssim \eps^{-1}|\ln\eps|\|\delta{\bs z}\|_{C^{2,\alpha}_\theta(\pC)}\\
&\|\bs i_{2}(\bs z;\mu,\xi)-\bs i_{2}(\bs z+\delta {\bs z};\mu+\delta\mu,\xi+\delta\xi)\|_{C^\alpha_\theta(\pC)}\lesssim \eps^{-2}|\ln\eps|^2\left(\eps+\|\mu\|_{C^\alpha(\Gamma)}+\|\xi\|_{C^\alpha(\Gamma)}\right)\|\delta{\bs z}\|_{C^{\alpha}_\theta(\pC)}\\ 
&\qquad+\eps^{-2}|\ln\eps|^2\left(\|\delta\mu\|_{C^\alpha(\Gamma)}+\|\delta\xi\|_{C^\alpha(\Gamma)}\right)\|{\bs z}\|_{C^\alpha_\theta(\pC)}
\end{aligned}
\]
\end{lemma}

Anticipating later developments we observe that several error terms in Lemma \ref{lem:est inner error} and Lemma \ref{lem:est inner error 2} carry terms of the size $\mathcal O(\eps e^{\,2K\theta|\ln\eps|}|\ln\eps|^q)$ with some $q>0$. If we rely on 
(\ref{est:glob hold}) of Proposition \ref{prop: inner lineal} to solve (\ref{eq:inner_1}) then it seems hard to close the estimates consistently with (\ref{hyp:size v}). To overcome this technical difficulty we will use $C^{1,\alpha}$ estimates  to control the leading order term of the function ${\bs z}^\parallel$ (c.f  (\ref{est:glob hold_2}) of Proposition \ref{prop: inner lineal}.)  We state now the estimates needed later on. 
\begin{corollary}
\label{lem:c1beta}
It holds
\begin{equation}
\label{est:c1beta}
\|\chi \bs \Omega_1(\cdot;0)\|_{C^{1,\alpha}_\theta(\pC)}+\|\chi\mathbf\Theta_l(\cdot; 0)\|_{C^{1,\alpha}_\theta(\pC)}+\|\bs j_i(\cdot; 0)\|_{C^{1,\alpha}_\theta(\pC)}\lesssim \eps e^{\,2 K\theta|\ln\eps|}|\ln\eps|^5
\end{equation}
where $l=1,\dots, 4$, $i=2,4$ and  $(\cdot;0)$ indicates that we take $\mu=0$, $\xi=0$ in the definition of the corresponding function.
\end{corollary}
To prove (\ref{est:c1beta}) we note that all terms appearing on the left hand side depend only on already known functions which have sufficient smoothness to state $C^{1,\alpha}_\theta(\pC)$ estimate as claimed. We omit standard details similar to those in the proofs of Lemma \ref{lem:est inner error} and Lemma \ref{lem:est inner error 2}.

\subsection{The modulation equation}\label{sec:modulation}
So far the functions $\mu$ and $\xi$ were merely parameters. Now we will derive the modulation equation which determines them in terms of the remaining unknowns. This will complete the derivation of the nonlinear problem we need to solve.  
From the definition of the modulation functions (\ref{hyp:expl mod}) we see that in reality we will obtain a system of nonlinear equations for finitely many Fourier coefficients $\hat\mu_{\eps, k}$, $\hat \xi_{\eps,k}$, ${\color{black}{|k|}}\leq \bar K_\eps\sim \eps^{-1}$.     
For brevity we denote
\[
\hat \mu(s)=\left(\sum_{{\color{black}{-\bar K_\eps}}}^{\bar K_\eps} \hat \mu_{k}\psi_{k}(s)\right), \qquad \hat \xi(s)=\left(\sum_{{\color{black}{-\bar K_\eps}}}^{\bar K_\eps} \hat \xi_{k}\psi_{k}(s)\right), 
\]
so that 
\[
\mu(\tau, s)=\hat \mu(s)\phi_{\mu}(\tau), \qquad \xi(\tau, s)=\hat \xi(s)\phi_{\xi}(\tau).
\]
We set 
\[
\mathfrak D=-\partial_{\tau\tau}-\eps^2 b^{-2}(s)\partial_{ss}.
\]
Our heuristic is based on the fact that 
\[
\bs \Omega_0=-\eps^{-1} b^3 \left[\bs X\mathfrak D\xi+\bs Y\mathfrak D\mu- 4\bs X \partial_\tau \mu\blu{+} 2\bs V'' (\mu+\tau\partial_\tau \mu+ \partial_\tau\xi)\right]+\dots,
\]
where $\dots$ indicate terms which should be smaller 
c.f (\ref{def omega zero}). We denote
\[
\hat {\bs\Omega}_{0}(\mu, \xi)=\bs X \mathfrak D\xi+\bs Y \mathfrak D \mu - 4\bs X \partial_\tau \mu\blu{+}2\bs V'' (\mu+\tau\partial_\tau \mu+ \partial_\tau\xi).
\]
\blu{
\begin{lemma}\label{lem phi}
There exist even functions $\phi_\mu$, $\phi_\xi$ satisfying (\ref{hyp: phi}) and positive constants $\ell_\mu$ and $\ell_\xi$ such that
\begin{equation}
\label{eq: lem phi 1}
\begin{aligned}
\langle \hat {\bs\Omega}_{0}(\mu, \xi), \bs X\rangle&=\ell_\xi \hat\xi+\mathcal O(\bar \omega)\hat \xi,\\
\langle \hat {\bs\Omega}_{0}(\mu, \xi), \bs Y\rangle&=\ell_\mu\hat\mu+\mathcal O(\bar\omega)\hat \mu.
\end{aligned}
\end{equation}
\end{lemma}
We prove this lemma in section \ref{pr lem phi}.
}

We will first use the above to solve the leading order terms $\mu^0$, $\xi^0$ (in $\eps$) of the modulation functions $\mu, \xi$. They depend only on the already known functions and in particular on $b$ and $\zeta$ which are $C^{3,
\alpha}$ regular by assumption. Let us denote  (c.f. Corollary \ref{lem:c1beta})
\[
\bs \Psi_0= b^{-3}\left[\bs \Omega_1(\cdot; 0)\chi+\sum_{l=1}^4\mathbf\Theta_l(\cdot; 0)\chi+\sum_{i=2,4} \bs j_i(\cdot; 0)\right].
\]
Denote 
\[
\bs \Psi_0^\parallel= P_{\bar K_\eps}\bs \Psi_0, \qquad \bs \Psi_0^\perp=\bs\Psi_0-\bs\Psi_0^\parallel,
\]
(c.f notation of Proposition \ref{prop: inner lineal}).
Our goal is to solve the system
\begin{equation}
\label{sys:muxi0}
\begin{aligned}
\langle \chi \hat{\bs \Omega}_0(\mu^0,\xi^0), \bs X \rangle &=\eps \langle \bs \Psi^\parallel_0, \bs X\rangle,\\
\langle \chi \hat{\bs \Omega}_0(\mu^0,\xi^0), \bs Y\rangle &=\eps \langle \bs \Psi^\parallel_0, \bs Y\rangle,
\end{aligned}
\end{equation}
for the unknown functions $\mu^0$ and $\xi^0$ of the form 
\[
\mu^0(\tau, s)=\hat \mu^0(s)\phi_{\hat\lambda}(\tau), \qquad \xi^0(\tau, s)=\hat \xi^0(s)\phi_{\hat\lambda}(\tau).
\]
where
\[
\hat \mu^0(s)=\left(\sum_{\blu{k=-\bar K_\eps}}^{\bar K_\eps} \hat \mu^0_{k}\psi_{k}(s)\right), \qquad \hat \xi^0(s)=\left(\sum_{\blu{k=-\bar K_\eps}}^{\bar K_\eps} \hat \xi^0_{k}\psi_{k}(s)\right), 
\]
in the H\"older class  $C^{2,\alpha}$. For this purpose and also having in mind later developments we state:
\blu{
\begin{lemma}\label{lem:mod 1}
Consider the following equations 
\[
\langle \chi\hat {\bs\Omega}_{0}(a, b), \bs T \rangle =\langle \bs \gamma^\parallel, \bs T\rangle
\]
where $\bs T=\bs X$ or $\bs T=\bs Y$. There exists $\hat\lambda$ sufficiently large such that  for any $\alpha\in(0,1)$ we have
\[
\|\hat a\|_{C^{2,\alpha}(\Gamma)}+\|\hat b\|_{C^{2,\alpha}(\Gamma)}\lesssim \eps^{-1+\alpha}\|\bs \gamma^\parallel\|_{C^\alpha_\theta(\pC)}
\]
and this estimate can be improved 
\[
\|\hat a\|_{C^{2,\alpha}(\Gamma)}+\|\hat b\|_{C^{2,\alpha}(\Gamma)}\lesssim |\ln\eps| \|\bs \gamma^\parallel\|_{C^{1,\alpha}_\theta(\pC)}
\]
\end{lemma}
\begin{proof}
We have
\[
a =\left(\sum_{k=-\bar K_\eps}^{\bar K_\eps}\hat  a_k\right)\psi_k(s)\phi_\mu, \qquad b =\left(\sum_{k=-\bar K_\eps}^{\bar K_\eps}\right)\hat  b_k\psi_k(s)\phi_\xi.
\] 
We also decompose the right hand side
\[
\langle \bs \gamma^\parallel, \bs T\rangle=\sum_{k=-\bar K_\eps}^{\bar K_\eps}\hat  \gamma_k\psi_k(s).
\]
By Lemma \ref{lem phi} we get 
\[
|\hat a_k|+|\hat b_k|\lesssim |\hat  \gamma_k|\lesssim k^{-\alpha}\left\|\langle \bs \gamma^\parallel, \bs T\rangle\right\|_{C^\alpha((0, |\Gamma|))}\lesssim k^{-\alpha}\|\bs \gamma^\parallel\|_{C^\alpha_\theta(\pC)}.
\]
Using Lemma \ref{lemm apend unif} and summing up the modes we conclude. The second statement is proven similarly.

\end{proof}
}
By the above Lemma,   Corollary \ref{lem:c1beta}  we have:
\begin{lemma}\label{lem leading mod}
The system (\ref{sys:muxi0}) has a solution such that 
\begin{equation}
\label{est sys mod0}
\|\mu^0\|_{C^{2,\alpha}(\Gamma)}+\|\xi^0\|_{C^{2,\alpha}(\Gamma)}\lesssim \eps^2 e^{\,2 K\theta|\ln\eps|}|\ln\eps|^7
\end{equation}
{\color{black}{and this  estimate holds as well for $C^{3,\alpha}(\Gamma)$ norms by a  bootstrap argument.}}
\end{lemma}
Note that this is consistent with (\ref{hyp:size mod}) provided that $K\theta$ is sufficiently small.
Next we look for the modulation functions $\mu$ and $\xi$ 
\[
\mu=\mu^0+\mu^1, \qquad \xi=\xi^0+\xi^1,
\]
where the new unknowns are of the form
\[
\begin{aligned}
&\mu^1(\tau, s)=\hat \mu^1(s)\blu{\phi_{\mu}}(\tau), \qquad \hat \mu^1(s)=\left(\sum_{\blu{k=-\bar K_\eps}}^{\bar K_\eps} \hat \mu^1_{k}\psi_{k}(s)\right), \\
& \xi^1(\tau, s)=\hat \xi^1(s)\blu{\phi_{\xi}}(\tau),\qquad \hat \xi^1(s)=\left(\sum_{\blu{k=-\bar K_\eps} }^{\bar K_\eps} \hat \xi^1_{k}\psi_{k}(s)\right).
\end{aligned}
\]
We will now derive the modulation equation for $\mu^1, \xi^1$. In fact it is an equation of the form (\ref{sys:muxi0}) but with the right  hand side depending on the unknown functions of the problem coupling 
it with the inner and the outer problem described above. Recalling the error terms treated in section \ref{sec. 4.5} we write
\[
P_{\bar K_\eps}  \left[b^{-3} (s) \bs g+b^{-3}(s)\bs i(\bs z)\right]=-\eps^{-1}P_{\bar K_\eps}\hat {\bs{\Omega}}_0(\mu^1, \xi^1)+
{\bs{\Psi}}_1^\parallel,
\]
where 
\[
{\bs{\Psi}}_1^\parallel=P_{\bar K_\eps}\left[b^{-3} (s) \bs g+b^{-3}(s)\bs i(\bs z) -  \eps^{-1}\chi \hat {\bs{\Omega}}_0(\mu, \xi)-\chi \bs\Psi_0\right].
\]
Based on the results of Section \ref{sec. 4.5} the following is straightforward:
\begin{lemma}\label{lem modul 2}
It holds
\begin{equation}\label{est modul 2}
\begin{aligned}
\|{\bs{\Psi}}_1^\parallel\|_{C^\alpha_\theta(\pC)} &\lesssim (1+\eps^{-1-\alpha} e^{\,-(\hat\lambda-\theta)  K|\ln\eps|})\left(\|\mu\|_{C^{2,\alpha}(\Gamma)}+ \|\xi\|_{C^{2,\alpha}(\Gamma)}\right)
\\&\quad + \eps^{-1}\left(\|\mu\|^2_{C^{2,\alpha}(\Gamma)}+\|\xi\|^2_{C^{2,\alpha}(\Gamma)}\right)+ \eps^{-2-\alpha} e^{\,2K\theta|\ln\eps|/\bar m} \sum \|\bar{w}_{ij}\|_{C^{1,\alpha}(\Omega_j)}\\
&\quad +\left(\eps^{-1}+\eps^{-2-\alpha}(\|\mu\|_{C^\alpha(\Gamma)}+\|\xi\|_{C^\alpha(\Gamma)})\right)|\ln\eps|^5\|{\bs z}\|_{C^\alpha_\theta(\pC)}\\
&\quad +\eps^{-3}|\ln\eps|\left(\|{\bs z}\|_{C^\alpha_\theta(\pC)}^2+e^{\,2K\theta|\ln\eps|/\hat m}\sum \|\bar{w}_{ij}\|^2_{C^\alpha(\Omega_j)}\right).
\end{aligned} 
\end{equation}
\end{lemma}
\begin{proof}
We have $\mu=\mu^0+\mu^1$, $\xi=\xi^0+\xi^1$, with $\mu^0$, $\xi^0$ already determined and satisfying (\ref{est sys mod0}). 
We will estimate few  terms in ${\bs{\Psi}}_1^\parallel$ leaving the rest for the reader. For example using (\ref{def omega zero})  we have
\[
\bs \Psi^\parallel_{11}=P_{\bar K_\eps}\left[ b^{-3}(s)\bs \Omega_0(\mu, \xi)+\eps^{-1}\chi\hat {\bs{\Omega}}_0(\mu, \xi)\right]=P_{\bar K_\eps}\left[ \eps b^{-2} \bs Y (\A_\eps^{-1}-1)\partial_{ss} \mu+\eps b^{-2} \bs X(\A_\eps^{-1}-1)\partial_{ss}\xi\right]
\]
hence
\[
\| {\bs{\Psi}}_{11}^\parallel\|_{C^\alpha_\theta(\pC)} \lesssim \eps^2 (\|\mu\|_{C^{2,\alpha}(\Gamma)}+\|\xi\|_{C^{2,\alpha}(\Gamma)}).
\]
Next we consider the remaining terms treated in Lemma \ref{lem:est inner error}. We observe that all terms that do not depend on $\mu$ or $\xi$ on the right hands are removed from $\bs \Psi^\parallel_1$ by subtracting $\bs \Psi_0$.  The only complicated term is the one involving $\eps^{-1}\|\mu\|_{C^{0,\alpha}(\Gamma)}$ in  $\bs \Theta_3$ but this term is part of the operator $\hat{\bs\Omega}_0$.  In all, denoting by $\bs \Psi_{12}^\parallel$ the terms coming from $\bs \Omega_1$ and $\bs\Theta_j$, $j=1,\dots, 4$ we get
\[
\|{\bs{\Psi}}_{12}^\parallel\|_{C^\alpha_\theta(\pC)} \lesssim \left(\|\mu\|_{C^{2,\alpha}(\Gamma)}+ \|\xi\|_{C^{2,\alpha}(\Gamma)}\right)+ \eps^{-1}\left(\|\mu\|^2_{C^{2,\alpha}(\Gamma)}+\|\xi\|^2_{C^{2,\alpha}(\Gamma)}\right).
\]
We turn our attention to Lemma \ref{lem:est inner error 2}. Again, subtraction of $\bs \Psi_0$ removes some terms from $\bs j$ and denoting what is left by $\bs \Psi^\parallel_{13}$ we have
\[
\begin{aligned}
\|{\bs{\Psi}}_{13}^\parallel\|_{C^\alpha_\theta(\pC)}&\lesssim \eps^{-1-\alpha} e^{\,-(\hat\lambda-\theta)  K|\ln\eps|}\left(\|\mu\|_{C^{1,\alpha}(\Gamma)}+\|\xi\|_{C^{1,\alpha}(\Gamma)}\right)
\\
&\quad + \eps^{-2-\alpha} e^{\,2K\theta|\ln\eps|/\bar m} \sum \|\bar{w}_{ij}\|_{C^{1,\alpha}(\Omega_j)}
\\
&\quad +\left(\eps^{-1}+\eps^{-2-\alpha}(\|\mu\|_{C^\alpha(\Gamma)}+\|\xi\|_{C^\alpha(\Gamma)})\right)|\ln\eps|^5\|{\bs z}\|_{C^\alpha_\theta(\pC)}\\
&\quad +\eps^{-3}|\ln\eps|\left(\|{\bs z}\|_{C^\alpha_\theta(\pC)}^2+e^{\,2K\theta|\ln\eps|/\hat m}\sum \|\bar{w}_{ij}\|^2_{C^\alpha(\Omega_j)}\right).
\end{aligned}
\]
Finally, denoting by $\bs \Psi_{14}^\parallel$ the projection of $\bs i$ (Lemma \ref{lem:est inner error 3}) we get
\[
\|{\bs{\Psi}}_{14}^\parallel\|_{C^\alpha_\theta(\pC)}\lesssim \eps^{-1}|\ln\eps|^2\left(1+\eps^{-1}\|\mu\|_{C^\alpha(\Gamma)}+\eps^{-1}\|\xi\|_{C^\alpha(\Gamma)}\right)\|{\bs z}\|_{C^\alpha_\theta(\pC)}.
\]
This ends the proof.
\end{proof}
For later purpose we state the following
\begin{corollary}\label{cor modul 2}
Estimate (\ref{est modul 2}) holds with $\bs \Psi_1^\parallel$ replaced by 
\[
\bs \Psi_1=b^{-3} (s) \bs g+b^{-3}(s)\bs i(\bs z) +  \eps^{-1}\chi \hat {\bs{\Omega}}_0(\mu, \xi)-\chi \bs\Psi_0.
\]
\end{corollary}
We state the nonlinear system for the modulation functions $\mu^1$, $\xi^1$:
\begin{equation}\label{fixed point 1}
\begin{aligned}
\langle \chi\hat{\bs \Omega}_0(\mu^1, \xi^1), \bs X\rangle&=\eps \langle \bs \Psi^\parallel_1, \bs X\rangle\\
\langle \chi \hat{\bs \Omega}_0(\mu^1,\xi^1), \bs Y\rangle &=\eps \langle \bs \Psi^\parallel_1, \bs Y\rangle
\end{aligned}
\end{equation}

\subsection{Conclusion of the proof}\label{sec fixed point 2}
It this section we will parallel the developments of the previous section: we will first improve the inner approximation and then state the nonlinear problem for the correction. 
We look for the solution of (\ref{eq:inner_1}) in the form $\bs z=\bs z^0+\bs z^1$ where 
\begin{equation}\label{eq: zzero}
\mathfrak L \bs z^0=\eps^2\left(-\eps^{-1}\chi\hat{\bs \Omega}_0(\mu^0, \xi^0)+\bs \Psi_0\right).
\end{equation}
Observe that the right hand side is of class $C^{1,\beta}$ hence we can rely on (\ref{est:glob hold}) and (\ref{est:glob hold_2})   of Proposition \ref{prop: inner lineal} to obtain
\begin{equation}\label{est:zzero}
\|\bs z^0\|_{C^{2,\alpha}_\theta(\pC)}\lesssim \eps^2|\ln\eps|\left(\|\chi\hat{\bs \Omega}_0(\mu^0, \xi^0)\|_{C^{1,\alpha}(\pC)}+\|\bs \Psi_0\|_{C^{1,\alpha}(\pC)}\right)\lesssim \eps^3e^{\,2 K\theta|\ln\eps|}|\ln\eps|^{10}.
\end{equation}
Next we write the nonlinear problem for the correction $\bs z^1$:
\begin{equation}\label{eq:zone}
\mathfrak L \bs z^1=\eps^2\left(-\eps^{-1}\hat {\bs{\Omega}}_0(\mu^1, \xi^1)+{\bs{\Psi}}_1\right).
\end{equation}
To complete the nonlinear system we need to solve to finish the proof we restate the outer equation in (\ref{sys:O3}) in the form (\ref{eq:outer lin}) (see also Section \ref{sec exp non})
\begin{equation}
\label{eq: barw outer lin}
\begin{aligned}
-\Delta \bar w_i-f_{u}(w^0_i,x) \mathbbm{1}_{\Omega_{i}} \bar w_i+\eps \left(\bs{q}^{1}\cdot\bar{\bs{w}}\right)_i&= (1-\chi) h_i,\quad\mbox{in}\quad \Omega\setminus \Gamma,\\
 \bar w_i&=0, \quad \hbox{on}\quad  \partial\Omega. \\
\end{aligned}
\end{equation} 
The right hand sides of this system are described and estimated in Lemma \ref{lem out 2}. To solve the outer problem we use the results of Lemma \ref{lem out 1}.

Next we solve the system (\ref{fixed point 1}),  (\ref{eq:zone}),   (\ref{eq: barw outer lin}) using the  Banach Fixed Point Theorem. To this end we suppose that 
we are given the modulation functions $\tilde\mu, \tilde\xi$ satisfying (\ref{hyp:size mod}), the inner function $\tilde{\bs z}$ satisfying (\ref{hyp:size v}) and the outer function $\tilde{\bs w}=(\tilde w_1, \tilde w_2)$ satisfying (\ref{hyp:size w}). We replace $\mu$, $\xi$ and $\bs z$  on the right hand sides of the system by $\mu^0+\tilde \mu$, $\xi^0+\tilde \xi$ and $\bs z^0+\tilde{\bs z}$. Taking into account the $\mu^0$ and $\xi^0$ satisfy (\ref{est sys mod0}) and $\bs z^0$ satisfies (\ref{est:zzero}) and using the theory developed above (Proposition \ref{prop: inner lineal}, Lemma \ref{lem out 1} and Lemma \ref{lem:mod 1}) we obtain with $\alpha\in (0, 1)$
\begin{equation}
\label{est:fix point_1}
\begin{aligned}
\|\mu^1\|_{C^{2,\alpha}(\Gamma)}+\|\xi^1\|_{C^{2,\alpha}(\Gamma)} &\lesssim\eps^{\alpha}\Big\{\eps^{-1-\alpha} e^{\,-(\hat\lambda-\theta)K|\ln\eps|}(\eps^2 e^{\,2K\theta|\ln\eps|}|\ln\eps|^7+\eps^{2-\bar\varsigma})
\\ &\qquad\quad +\eps^3 e^{\,4K\theta|\ln\eps|}|\ln\eps|^{14}+\eps^{3-\hat\varsigma}+\eps^{2+\bar\varsigma-\alpha}e^{\,2K\theta|\ln\eps|/\bar m}\\
&\qquad\quad +(\eps^{-1}+\eps^{-\alpha}e^{\,2K\theta|\ln\eps|}|\ln\eps|^7)(\eps^3e^{\,2K\theta|\ln\eps|}|\ln\eps|^{10}+\eps^{2+\tilde\varsigma})|\ln\eps|^5
\\
&\qquad\quad +|\ln\eps|(\eps^3 e^{\,4K\theta|\ln\eps|}|\ln\eps|^{10}+\eps^{1+2\tilde\varsigma}+\eps^{5+2\bar\varsigma}e^{\,2K\theta|\ln\eps|/\hat m}\Big\}.
\end{aligned}
\end{equation}
The first line above can be made as small as we wish by choosing $\hat\lambda$ large. The second, the  third and the fourth lines can be  made of
size $o(\eps^2)$ if $\theta$ is chosen small and 
\begin{equation}
\label{par restr 1}
\alpha-\tilde\varsigma>-1, \qquad \tilde\varsigma+\hat\varsigma+\alpha>1, \qquad 2\tilde\varsigma+\alpha>1.
\end{equation}
Based on Corollary \ref{cor modul 2}  we see that  under the same restrictions 
\begin{equation}
\label{est:fix point_2}
\begin{aligned}
\|\bs z^1\|_{C^{2,\alpha}_\theta(\pC)}&\lesssim o(\eps^3).
\end{aligned}
\end{equation}
To estimate the size of the outer solution we rely  on Lemma \ref{lem out 1} and Lemma \ref{lem out 2}: 
\begin{equation}
\label{est:fix point_3}
\begin{aligned}
\|\bar w_i\|_{C^{1,\alpha}(\Omega_i)}+\|\bar w_i\|_{C^{1,\alpha}(\Omega_{3-i})}&\lesssim (\eps|\ln\eps|)^{5-\alpha}+\eps^{4-\alpha+2\bar\varsigma}\\
&\quad + (\eps^{\tilde\varsigma-\alpha}+\eps^3 e^{\,2K\theta|\ln\eps|}|\ln\eps|^{10})(e^{\,-2K\theta|\ln\eps|/\tilde m}+e^{\,-2K\theta|\ln\eps|/\hat m}).
\end{aligned}
\end{equation}
for $i=1,2$. The first two terms are of order $o(\eps^{4+\bar\varsigma})$ when 
\begin{equation}
\label{par restr 2}
1-\alpha>\bar\varsigma, \qquad \bar\varsigma>\alpha,
\end{equation}
and the second line can be made of the same size by choosing $\tilde m<\hat m<1$ small. It is easily checked that conditions (\ref{par restr 1}) and  (\ref{par restr 2}) are satisfied when 
\begin{equation}
\label{par restr 3}
1>\alpha+\bar\varsigma, \qquad \bar\varsigma>\alpha,\qquad  \tilde\varsigma+\hat\varsigma>1-\alpha, \qquad \tilde\varsigma>\frac{1-\alpha}{2}.
\end{equation}
These conditions should be satisfied together with 
\begin{equation}
\label{par restr 4}
\eps^2e^{\,2 K\theta|\ln\eps|}|\ln\eps|^{10}\lesssim o(\eps^{2-\hat\varsigma}), \qquad \eps^3e^{\,2 K\theta|\ln\eps|}|\ln\eps|^{10}\lesssim o(\eps^{2+\tilde\varsigma}).
\end{equation}
Explicitly we can chose $\alpha$ smaller than but close to $1/2$ and then $\bar\varsigma$ larger than, close to $1/2$ and $\hat\varsigma, \tilde\varsigma$ larger than $1/4$ but smaller than $1/2$. In fact it can be checked that for any $\alpha\in (0,1/2)$ it is possible to chose the other constants so that (\ref{par restr 4}) holds.  Finally we fix $K$ and chose $\theta$ small to satisfy (\ref{par restr 4}).
Summarizing we see that the nonlinear map
\[
(\tilde\mu,\tilde\xi, \tilde{\bs z}, \tilde{\bs w})\longmapsto (\mu^1,\xi^1, \bs z^1, \bar{\bs w})
\]
defined by (\ref{fixed point 1}),  (\ref{eq:zone}),   (\ref{eq: barw outer lin}) transforms the set
\[
\bs B(\hat\varsigma, \tilde\varsigma, \bar\varsigma)=\{(\tilde\mu,\tilde\xi, \tilde{\bs z}, \tilde{\bs w})\ \mbox{satisfying (\ref{hyp:size mod}), (\ref{hyp:size v}),(\ref{hyp:size w})} \},
\]
into itself. It remains to check that this map is also a contraction. This is standard given the results of Corollaries \ref{cor:lipschitz 2}, \ref{cor:lip g}  and Lemmas \ref{lem:est inner error}, \ref{lem:est inner error 3}. {\color{black}{We omit somewhat tedious details.}} To conclude the proof of the theorem we need to show (\ref{est: calpha}) and (\ref{est: c2alpha}). By construction of the solution it is evident that it suffices to show that corresponding facts replacing the function $\bs u=(u_1,u_2)$ by the approximate solution $\bs U=(U_1, U_2)$.
Straightforward direct argument using the explicit definition of the approximate solution then shows the required estimates. This ends the proof of the theorem.  
}}

\blu{

\section{Proof of Lemma \ref{lem phi}}\label{pr lem phi}

We have
\[
\begin{aligned}
\langle \hat {\bs\Omega}_{0}(\mu, \xi), \bs X\rangle&=\langle\bs X \mathfrak D\xi,\bs X\rangle+\langle\bs Y \mathfrak D \mu,\bs X\rangle - 4\langle \bs X \partial_\tau \mu,\bs X\rangle \\
&\qquad \blu{+}2\langle \bs V'' \mu,\bs X\rangle +2\langle \bs V''\tau\partial_\tau \mu,\bs X\rangle+ 2\langle \bs V''\partial_\tau\xi,\bs X\rangle.
\end{aligned}
\]
We observe that because $V_1(\tau)=V_2(-\tau)$ the product $\bs X\cdot\bs Y$ is odd while functions $\phi_\mu$ and $\phi_\xi$ are even. From this we see that 
\[
\langle \hat {\bs\Omega}_{0}(\mu, \xi), \bs X\rangle=\langle\bs X \mathfrak D\xi,\bs X\rangle+2\langle \bs V''\partial_\tau\xi,\bs X\rangle
\]
We have 
\[
\begin{aligned}
\langle\bs X \mathfrak D\xi,\bs X\rangle&=-\hat\xi\langle \bs X \phi_\xi'', \bs X\rangle +\mathcal O(\bar \omega)\hat \xi=-\hat\xi\int (|X|^2)_{\tau\tau}\phi_\xi+\mathcal O(\bar \omega)\hat \xi\\
&=-2\hat\xi\int(V_1''V_1'+V_2'' V_1')_\tau\phi_\mu+\mathcal O(\bar \omega)\hat \xi
=-2\hat\xi\int(\bs V''\cdot \bs X)_\tau \phi_\xi+\mathcal O(\bar \omega)\hat \xi.
\end{aligned}
\]
It follows 
\[
\langle \hat {\bs\Omega}_{0}(\mu, \xi), \bs X\rangle=-4\hat\xi \int(\bs V''\cdot \bs X)_\tau \phi_\xi+\mathcal O(\bar \omega)\hat \xi.
\]
We set
\[
\phi_\xi=-(\bs V''\cdot \bs X)_\tau.
\]
We get explicitly
\[
\int(\bs V''\cdot \bs X)_\tau \phi_\xi=\int|(V_1V_1' V_2^2+V_2 V_2' V_1^2)_\tau|^2\,d\tau>0
\]
provided that $\phi_\xi\not\equiv 0$. 
To show that $\phi_\xi\not\equiv 0$ we note that 
\[
\phi_\xi(0)=-2(1-V'(0)^2). 
\]
In \cite{sourdis 1} it was shown that $V'_1(\infty)^2=\frac{1}{2}$. Using the equation it is not hard to show that
\[
V_1'(0)^2=\frac{1}{2}(1+V_1'(\infty)^2)=\frac{1}{2}\left(1+\frac{1}{2}\right)<1.
\] 
hence $\phi_\xi(0)<0$. Note that 
\[
|\phi_\xi(\tau)|\lesssim e^{\,-c|\tau|^2}.
\]
This ends the proof of the first identity in (\ref{eq: lem phi 1}). 

Next we will show the second identity in (\ref{eq: lem phi 1}). Similarly as above
\[
\begin{aligned}
\langle \hat {\bs\Omega}_{0}(\mu, \xi), \bs Y\rangle&=\langle\bs X \mathfrak D\xi,\bs Y\rangle+\langle\bs Y \mathfrak D \mu,\bs Y\rangle - 4\langle \bs X \partial_\tau \mu,\bs Y\rangle \\&\qquad \blu{+}2\langle \bs V'' \mu,\bs Y\rangle +2\langle \bs V''\tau\partial_\tau \mu,\bs Y\rangle+ 2\langle \bs V''\partial_\tau\xi,\bs Y\rangle\\
&=\hat\mu\left(-\int (|\bs Y|^2)_{\tau\tau}\phi_\mu +\int (\bs X\cdot \bs Y)_\tau \phi_\mu+2\int \bs V''\cdot\bs Y \phi_\mu- 2\int (\tau \bs V''\cdot Y)_\tau \phi_\mu\right)\\
&\qquad +\mathcal O(\bar \omega)\hat\mu.
\end{aligned}
\]
We set 
\[
\phi_\mu(\tau) =\left(- (|\bs Y|^2)_{\tau\tau}+(\bs X\cdot \bs Y)_\tau+2\bs V''\cdot\bs Y-2(\tau \bs V''\cdot \bs Y)_\tau\right)\sech(\hat\lambda \tau).
\]
With this definition $\phi_\mu$ is even by the symmetry of $\bs V$. To show that $\phi_\mu\not\equiv 0$ we note that after some elementary calculations we get
\[
- (|\bs Y|^2)_{\tau\tau}+(\bs X\cdot \bs Y)_\tau+2\bs V''\cdot\bs Y-2(\tau \bs V''\cdot \bs Y)_\tau=-4 |\bs V'(\tau)|^2+\mathcal O(e^{\,-c|\tau|^2}),
\]
hence for $\tau\to \infty$ 
\[
\phi_\mu(\tau)=-4 |\bs V'(\infty)|^2\sech(\hat\lambda \tau)+\mathcal O(e^{\,-c|\tau|^2})=-2\sech(\hat\lambda \tau)+\mathcal O(e^{\,-c|\tau|^2})\not\equiv 0.
\]
This end the proof of the Lemma. 
}
\section{The rate of decay of Fourier modes for  H\"older continuous functions}\label{sec: 5}
\setcounter{equation}{0}
Using  Liouville Normal Form of the Sturm-Liouville operator 
\[
y(s)=\int_0^s {b_0(\sigma)}\,d\sigma, \qquad \tilde \psi(y)=\psi(s(y))\sqrt{b_0(s(y))}, \qquad \tilde q =\frac{{\frac{d^2}{dy^2}} \sqrt{b_0(s(y))}}{\sqrt{b_0(s(y))}},
\]
we can transform  (\ref{eq: per 1}) to 
\[
-\partial_{yy}\tilde \psi+ \tilde q \tilde \psi=\omega^2\tilde \psi, \qquad y\in [0, \tilde \ell], \qquad \tilde \ell =\int_0^{|\Gamma|} {b_0(\sigma)}\,d\sigma
\]
with periodic boundary conditions. Under this change of variables, for any two functions $u_1, u_2$ in $[0, |\Gamma|]$ we have
\[
\int_0^{|\Gamma|} u_1(s) u_2(s)\,b^2_0(s)ds=\int_0^{\tilde\ell} \tilde u(y)\tilde v(y)\,dy, \qquad \tilde u_j(y)=u_j(s(y))\sqrt{b_0(s(y))}.
\]

Consider, more generally,  the following operator $L_q=-\partial_{ss}+ q(s)$ in $\mathbb S^1$, and let us assume that $q$ is such that $L_q\geq 0$. 
Let $\omega_k^2\geq 0$, {\color{black}{$k\in \mathbb Z$,}} be the eigenvalues and $\psi_k$ the eigenfunctions of $L_q$.  Note that  by the normal form transformation and possible further scaling of the independent variables (\ref{eq: per 1}) becomes an eigenvalue 
problem for the operator of the form $L_q$. From now on we will consider the operator $L_q$ on $\mathbb S^1$. From the preceding discusion the result of Proposition \ref{lem apend 1} (below) applies in the original case. 

{\color{black}{
The following summarizes the contents of Lemma 1.2.1 and Corollary 1.4.1 in \cite{leviatan}
\begin{lemma}\label{lemm apend unif}
Let $\omega_k$ be the eigenvalues and $\psi_k$ normalized eigenfunctions of $L_q$. 
\begin{itemize}
\item[(i)]
The following asymptotic formula holds
\[
\omega_k^2= k^2+\mathcal O(k^{-2}), \quad k\in \mathbb Z
\]
as $k\to \infty$.
\item[(ii)]
\[
\psi_k=a_k \cos\omega_k x+{b_k}\sin \omega_k x+\mathcal O(|\omega_k|^{-1}),
\]
where $a_k=\mathcal O(1)$ and $b_k=\mathcal O(1)$. 
\end{itemize}
\end{lemma}
For any $u\in L^2(\ES^1)$ we set
\[
\hat u_k=\int_{\ES^1} u(x)\psi_k(x)\, dx.
\]
Now we prove:
\begin{proposition}\label{lem apend 1}
For any $u\in C^{\alpha}(\ES^1)$,  $\alpha\in (0,1)$  it holds
\[
|u_k|\lesssim |k|^{-\alpha}\|u\|_{C^{\alpha}(\ES^1)}.
\] 
\end{proposition}
\begin{proof}
By a well known result in harmonic analysis: 
\[
\left|\int_{\ES^1} u(x) e^{\,-ikx}\,dx\right|\lesssim |k|^{-\alpha}\|u\|_{C^{\alpha}(\ES^1)}
\]
From this and Lemma \ref{lemm apend unif} our claim follows easily. 
\end{proof}
We are thankful to  J. Jendrej \cite{jacek} who pointed out to us this simple argument.

}}


\begin{thebibliography}{99}
\bibitem{sourdis 1} A. Aftalion and C. Sourdis, \emph{Interface layer of a two-component Bose-Einstein system} Comm. Contemp. Math. 2016.
\bibitem{blwz} H. Berestycki, T.-C. Lin, J. Wei and C. Zhao, \emph{On phase-separation models: asymptotics and qualitative properties}, Arch. Ration. Mech. Anal. 208 (1) (2013) 163--200.
\bibitem{BE} J.Behrndt, A.F.M. ter Elst, \emph{Dirichlet to Neumann maps on bounded Lipschitz domains}, JDE 259   (2015) 5903-5926.
\bibitem{cl}L.A. Caffarelli, F. Lin, Singularly perturbed elliptic systems and multi-valued harmonic functions with free boundaries, J. Am. Math. Soc. 21 (2008) 847–862.
\bibitem{cs} J.-B. Casteras, C. Sourdis, Construction of a solution for the two-component radial Gross-Pitaevskii system with a large coupling parameter, J. Funct. Anal. 279 (2020) 108674
\bibitem{ctv} M. Conti, S. Terracini, G. Verzini, Nehari's problem and competing species system, Ann. Inst. Henri Poincaré, Anal. Non Linéaire 19 (2002) 871–888.

\bibitem{defig} D.G. De Figueiredo and E. Mitideri, (1990) \emph{Maximum principles for linear elliptic systems}, in: Costa D. (eds) Djairo G. de Figueiredo - Selected Papers. Springer, Cham. (2014)

\bibitem{11} B.D. Esry, C.H. Greene, J.P. Burke Jr., J.L. Bohn, Hartree–Fock theory for double condensates, Phys. Rev. Lett. 78 (1997) 3594–3597.
\bibitem{12} B.D. Esry, C.H. Greene, Spontaneous spatial symmetry breaking in two-component Bose–Einstein condensates, Phys. Rev. A 59 (1999)
1457–1460.
 \bibitem{15}  D.S. Hall, M.R. Matthews, J.R. Ensher, C.E. Wieman, E.A. Cornell, Dynamics of component separation in a binary mixture of Bose–Einstein
condensates, Phys. Rev. Lett. 81 (1998) 1539–1542.

\bibitem{jacek} J. Jendrej, Private communication. 

\bibitem{KPV} M. Kowalczyk, A. Pistoia, G. Vaira, Maximal solution of the Liouville equation in doubly connected domains, J. Functional Analysis, 277 (2019), no 9, 2997—3050.

\bibitem{leviatan} B. M. Levitan and  I.S. Sargsjan, \emph{Sturm-Liouville and Dirac Operators}, Mathematics and Its Applications (Soviet series), Kluver Academic Publishers (1991). 

\bibitem{WM} W. McLean, \emph{Strongly Elliptic System and Boundary Integral Equations}, Cambridge University Press (2000).
 \bibitem{22} C.J. Myatt, E.A. Burt, R.W. Ghrist, E.A. Cornell, C.E. Wieman, Production of two overlapping Bose–Einstein condensates by sympathetic
cooling, Phys. Rev. Lett. 78 (1997) 586–589

\bibitem{nttv} B. Noris, H. Tavares, S. Terracini, G. Verzini; Uniform Hölder bounds for nonlinear Schrödinger systems with strong competition, Comm. Pure Appl. Math, 63 (2010), 267–302.
 \bibitem{23}A.S. Parkins, D.F. Walls, The Physics of trapped dilute-gas Bose–Einstein condensates, Phys. Rep. 303 (1998) 1–80.
\bibitem{sz} N. Soave, A. Zilio, On phase separation in systems of coupled elliptic equations: asymptotic analysis and geometric aspects, Ann. Inst. Henri Poincaré, Anal. Non Linéaire 34 (2017) 625–654.
 \bibitem{tt}H. Tavares, S. Terracini, Sign-changing solutions of competition-diffusion elliptic systems and optimal partition problems, Ann. Inst. Henri Poincaré, Anal. Non Linéaire 29 (2012) 279–300.  
\bibitem{ttvw}H. Tavares, S. Terracini, G. Verzini, T. Weth, Existence and nonexistence of entire solutions for non-cooperative cubic elliptic systems, Commun. Partial Differ. Equ. 36 (2011) 1988–2010.  
\bibitem{tv}S. Terracini, G. Verzini, Multipulse phase in k-mixtures of Bose-Einstein condensates, Arch. Ration. Mech. Anal. 194 (2009) 717–741.

  \bibitem{25} E. Timmermans, Phase separation of Bose–Einstein condensates, Phys. Rev. Lett. 81 (1998) 5718–5721.
    \bibitem{q}F. Quinn, Transversal approximation on Banach manifolds, Global Analysis (Proc. Sympos. Pure Math., Vol. XV, Berkeley, Calif., 1968), Amer. Math. Soc., Providence, R.I., 1970, pp. 213–222. 
\bibitem{st}J.-C. Saut and R. Temam, Generic properties of nonlinear boundary value problems, Comm. Partial Differential Equations 4 (1979), 293–319.
\bibitem{u}K. Uhlenbeck, Generic properties of eigenfunctions, Amer. J. Math. 98 (1976), 1059–1078. MR0464332
 \end{thebibliography}
 \end{document}